\newif\ifspringer
\springerfalse
\newif\ifsharelatex
\sharelatextrue
\makeatletter
\providecommand*{\input@path}{}
\g@addto@macro\input@path{{./}{include/}{../include/}}
\makeatother

\ifspringer
  \RequirePackage{amsmath}
  \documentclass[smallextended,envcountsame,nospthms]{svjour3}     % onecolumn (ditto)
  \smartqed  % flush right qed marks, e.g. at end of proof
  \usepackage{PHA_springer}
  \journalname{Mathematical Programming Computation}
  \title{Partial hyperplane activation for generalized intersection cuts%\thanks{Grants or other notes
  }
\author{Aleksandr M. Kazachkov \and
    Selvaprabu Nadarajah \and
    Egon Balas \and
    Fran\c{c}ois Margot %
  }
  
  \institute{E. Balas \and A. Kazachkov \and F. Margot \at
                Tepper School of Business,
                Carnegie Mellon University \\
                Pittsburgh, PA 15213, USA \\
                \email{akazachk@alumni.cmu.edu}           %  \\
             \and
             E. Balas \at
               \email{eb17@andrew.cmu.edu}
             \and
             F. Margot \at
               \email{fmargot11@gmail.com} %
             \and
             S. Nadarajah \at
                College of Business Administration,
                University of Illinois at Chicago \\
                Chicago, IL 60608, USA \\
                \email{selvan@uic.edu}
  }
  
\else
  \documentclass[12pt]{article} % comment out if using Springer template
  \usepackage{PHA}
  \title{Partial hyperplane activation for generalized intersection cuts%\thanks{Grants or other notes
  }
  \author[1]{Aleksandr M. Kazachkov\thanks{\href{mailto:akazachk@cmu.edu}{\nolinkurl{akazachk@alumni.cmu.edu}}}}
  \author[2]{Selvaprabu Nadarajah}
  \author[1]{Egon Balas}
  \author[1]{\authorcr Fran\c{c}ois Margot}
  \affil[1]{Tepper School of Business,
    Carnegie Mellon University,
    Pittsburgh, PA}
  \affil[2]{College of Business Administration,
    University of Illinois at Chicago,
    Chicago, IL}

\fi

\usepackage{import}
\usepackage{subfiles}
\graphicspath{{./}{./figs/}{../figs/}}

\usepackage{subcaption}

\ifsharelatex
\def\biblio{\bibliographystyle{named}\bibliography{{\mainDir/\bibDirName/akazachk}}}
\else
 \def\biblio{\bibliographystyle{named}\bibliography{\mainDir/\bibDirName/akazachk}}  % *Modification: added `\main/` to specify relative file location.
\fi

    \usepackage{verbatim}
    
\def\zX{0}
\def\zY{0}
\newcommand{\xangle}{0}
\newcommand{\yangle}{90}
\newcommand{\zangle}{225}

\newcommand{\xlength}{1}
\newcommand{\ylength}{1}
\newcommand{\zlength}{1}

\pgfmathsetmacro{\xx}{\xlength*cos(\xangle)}
\pgfmathsetmacro{\xy}{\xlength*sin(\xangle)}
\pgfmathsetmacro{\yx}{\ylength*cos(\yangle)}
\pgfmathsetmacro{\yy}{\ylength*sin(\yangle)}
\pgfmathsetmacro{\zx}{\zlength*cos(\zangle)}
\pgfmathsetmacro{\zy}{\zlength*sin(\zangle)}

\def\zX{\zx}
\def\zY{\zy}

\makeatletter
\algrenewcommand\ALG@beginalgorithmic{\small}
\makeatother

\usepackage{setspace}
\ifspringer
\else
\fi
\hypersetup{final}

\date{}

\begin{document}	

\def\biblio{}
\ifspringer
\renewcommand\subclassname{{\bfseries Mathematics Subject Classification}\enspace}
\fi
\maketitle
\begin{abstract}
The generalized intersection cut (GIC) paradigm is a recent framework for generating cutting planes in mixed integer programming with attractive theoretical properties. We investigate this computationally unexplored paradigm and observe that a key hyperplane activation procedure embedded in it is not computationally viable. To overcome this issue, we develop a novel replacement to this procedure called partial hyperplane activation (PHA), introduce a variant of PHA based on a notion of hyperplane tilting, and prove the validity of both algorithms. We propose several implementation strategies and parameter choices for our PHA algorithms and provide supporting theoretical results. We computationally evaluate these ideas in the COIN-OR framework on MIPLIB instances. Our findings shed light on the the strengths of the PHA approach as well as suggest properties related to strong cuts that can be targeted in the future.
\ifspringer
\fi
\end{abstract}
%\newpage

%% SECTION 1: INTRODUCTION
\section{Introduction} % Non-technical
% What are you solving
% Why is it important
% What are other people doing
% What are your contributions

An important aspect of modern integer optimization solvers is the use of 
cutting planes to strengthen a given formulation~\cite{AchWun13}.
Finding methods for generating stronger general-purpose cutting planes 
has been an active topic of research in the past decade.
As part of this trend,
\citet{BalMar13} introduced the \emph{generalized intersection cut} (GIC) paradigm with the motivation of finding stronger cutting planes that also possess favorable numerical properties, such as avoiding numerical inaccuracies that arise in traditional cutting plane approaches~\cite{ZanFisBal11}.
Despite offering some theoretical advantages to other modern cutting planes,
GICs have remained unexplored computationally.
In this paper, we observe that generating GICs as outlined in \cite{BalMar13}
is computationally intractable
due to the exponential size of the linear program used to generate cuts.
We offer a solution by extending the notion of a GIC and devising new algorithms to generate such GICs that scale well with the size of the input data.
We prove the validity of our algorithms,
provide theoretical results that guide our implementation choices,
and perform the first computational investigation with GICs. %the characteristics of strong GICs.
Our investigation identifies properties related to strong cuts that can be targeted in future approaches within the paradigm.
%We also provide theoretical results that guide some of our implementation choices and shed light on the strength of GICs.

Let 
  $P := \{x \in \R^n \suchthat Ax \ge b,\ x \ge 0\}$
and
  $P_I := \{x \in P \suchthat x_j \in \Z \text{ for all $j \in I$}\}$
for a set $I \subseteq \{1,\ldots,n\}$,
where all data is rational.
Let $C$ be a relaxation of $P$ defined by a subset of the inequalities defining $P$.
\citet{BalMar13} produce GICs from a collection of intersection points and rays obtained by intersecting the edges of $C$ with the boundary of a convex set $S$ containing no points from $P_I$ in its interior.
The simplest GICs are the \emph{standard intersection cuts} (SICs)~\cite{Balas71},
which use as $C$ a simple polyhedral cone $\optcone$ with apex at a vertex $\lpopt$ of $P$.
GICs generalize SICs by using a tighter relaxation of $P$,
by \emph{activating} additional hyperplanes of $P$ that are not included in the description of $\optcone$.
However, hyperplane activation poses two computational issues:
first, 
maintaining a description of $C$ becomes challenging,
and second,
the number of points and rays can quickly grow too large for any practical use.

This paper introduces a new method, called \emph{partial hyperplane activation} (PHA), 
that addresses the issues with the aforementioned \emph{full} hyperplane activation proposed in \cite{BalMar13}.
The key insight underlying PHA is
that we can forgo a complete description of the polyhedron $C$ by always activating hyperplanes on the initial cone $\optcone$, instead of on an iteratively refined relaxation.
We show the PHA approach is not only valid, but also generates a collection of intersection points and rays that grows quadratically in size with the number of hyperplane activations, in contrast to the exponential growth exhibited by the full activation procedure.

Activating a hyperplane on $\optcone$ creates new vertices lying at most one edge of $P$ away from $\lpopt$.
Consequently, intersection points are obtained from edges originating at $\lpopt$ or these distance $1$ vertices.
Higher-distance PHA methods, with accompanying higher computational cost, can be easily defined by incorporating hyperplane intersections with edges a larger distance from $\lpopt$,
eventually recovering the full activation procedure. In other words, while we focus in this paper on the details of a distance 1 PHA procedure, its underlying ideas extend to a hierarchy of PHA procedures delivering progressively "stronger" sets of intersection points.

An important observation related to the effectiveness of PHA is that weak intersection points are created whenever a hyperplane being activated intersects rays of $\optcone$ that do not intersect the boundary of $S$.
To mitigate this issue, we introduce a class of \emph{tilted} hyperplanes that can be used to avoid such rays. 
We also show that tilted hyperplanes offer more control over the number of intersection points generated with essentially no additional overhead, but they require a subtle condition to ensure that these points lead to cuts \emph{valid} for $P_I$, i.e., to inequalities that do not cut off any points of $P_I$. 
We use this idea to design a modified PHA algorithm that creates a point-ray collection that grows linearly with the number of hyperplane activations but quadratically in the number of rays of $\optcone$ being cut. 
The notion of tilting may also have independent merit outside of the GIC paradigm whenever a $\mathcal{V}$-polyhedral partial description could be useful, i.e., describing a polyhedron by some of its vertices and edges (see, e.g., \cite{Ziegler95}).

Implementing PHA or its variant with tilting requires several algorithmic design choices, such as the decision of which particular hyperplanes to activate and which objective functions to use when optimizing over the derived collection of intersection points and rays. Our theoretical contributions provide geometric and structural insights into making some of these choices.
Our experiments test the strength of GICs with respect to root node gap closed, a commonly used measure of cut strength, based on three hyperplane activation rules and four types of objective directions.

The first hyperplane activation rule we consider is based on the full activation method, which sequentially activates hyperplanes by pivoting to neighbors of $\lpopt$.
The second rule targets points that lie deep relative to the SIC.
The third rule is motivated by a pursuit of a certain class of intersection points called \emph{final}, relating to facet-defining inequalities for the \emph{$\Sk$-closure}, which is a relaxation of the split closure obtained by refining $P$ through the addition of one simple split disjunction on a variable $x_k$.

We then discuss the objective directions evaluated in our implementation, which contribute to a better understanding of the process by which a collection of strong cuts can be generated.
The first two sets of objective functions, as with the hyperplane rules, are motivated by finding cuts that improve over the SIC, as well as generalizing the usual objective of maximizing violation with respect to $\lpopt$.
We then give a result motivating the next set of directions, by showing that the optimal objective value over the $\Sk$-closure can be obtained using only final intersection points and describing how to attain this same value using cuts.
The fourth objective function type capitalizes on information that can be used from the typical empirical setup in which multiple split disjunctions are applied in parallel.

Finally, we present the first computational results for any algorithm within the GIC paradigm,
implemented in the open source COIN-OR framework~\cite{COIN-OR} and using a set of benchmark instances from MIPLIB~\cite{MIPLIB}. We find that our PHA approach creates a manageable number of intersection points from which we can generate GICs that effectively generalize SICs:
one round of GICs generated from our methodology improves the \emph{integrality gap} closed by one round of SICs 
on 75\% of the instances tested, closing on average an additional 5.1\% of the gap over SICs, on instances for which either SICs or GICs close any gap.
Though our procedure is intended to be nonrecursive, we do test a second round of GICs, the outcome of which is some indication of the lack of tailing off by GICs as compared to SICs.
Underlying these results is a detailed analysis of the various hyperplane activation and objective function choices for the linear program used to generate GICs. 
These experiments build an understanding of structural properties of strong GICs. 
In particular, we find that the objective functions inspired by final intersection points frequently find the strongest cuts from a given point-ray collection compared to the other objectives tested. 
However, such final intersection points appear infrequently in the point-ray collections generated by PHA, yielding a concrete direction for future work: to target final intersection points more directly.

Our research on GICs adds to the growing literature that generalizes and extends
cutting plane methods based on intersection cuts and Gomory cuts.
SICs are equivalent to Gomory cuts~\cite{Gomory69,GomJoh72a,GomJoh72b},
a relationship surveyed in \citet{ConCorZam11},
and are among the simplest, yet most effective, cuts used in practice.
As SICs are generated using information from only one row of the simplex tableau,
GICs are a natural generalization by using information from multiple rows.
Using such information has been studied extensively in the past decade
\cite{AndLouWeiWos07,BasBonCorMar11a,DasGun13,DasGunVie14,DasGunMol15,DasGunMor16,DeyLodTraWol14,Espinoza10,LouPoiSal15}.
The focus of these approaches is to obtain stronger cuts by using a better
cut-generating set obtained from the optimal solution to the continuous relaxation.
GICs instead attempt to strengthen SICs using other rows but the same cut-generating set.
The collection of intersection points and rays used to derive GICs can also
be seen as defining a linear system for verifying the validity of cutting planes~\cite{DeyPok14}.

Broadly, GICs fall into the class of algorithms using the 
classical theoretical approach of polarity to generate cuts~\cite{Balas79}.
This includes similarities to local or target cuts~\cite{BucLieOsw08}, which use the polar of a set of points and rays to generate cuts,
and to the procedures in \cite{PerBal01,LouPoiSal15}, which use row generation to individually certify the validity of each generated cut.
Though utilizing some of the same theoretical tools, GICs derive points and rays from a fundamentally different perspective and in a way that immediately guarantees cut validity.

This paper is organized as follows.
In Section~\ref{sec:FHA}, we revisit full hyperplane activation from \cite{BalMar13}
and demonstrate the impediment to a practical GIC algorithm inherent in that approach.
Section~\ref{sec:PHA} then gives the main contribution of this paper, the PHA scheme.
Section~\ref{sec:tilting} shows the idea of tilting and gives an additional algorithm that can be used within the PHA procedure.
Implementation choices and supporting theory are discussed in Section~\ref{sec:theory}.
Lastly, we give the results of our computational experiments in Section~\ref{sec:computation}.
Additional supporting theoretical results are contained in the appendix.

%% SECTION 2: FULL HYPERPLANE ACTIVATION
\section{Full hyperplane activation}
\label{sec:FHA}
% Introduce notation
% Full hyperplane activation
% Why it is exponential
% Emphasize: that this will not work
% TODO: Include CDD attempt (up to 4 hyperplanes activated)

In this section, we show that the full hyperplane activation approach to generate GICs as
proposed by \citet{BalMar13} is impractical,
which provides the motivation for developing our new methodology in later sections.
Given objective vector $c \in \Q^n$, the optimization problem is
  \begin{align}
    \min \{c^\T x : x \in P_I\}. %= \min \{c^\T x : x \in \conv(P_I)\}. 
    \tag{IP}\label{IP}
  \end{align}
Let \eqref{LP} denote the \emph{continuous} or \emph{linear relaxation} of \eqref{IP},
obtained by removing any integrality restrictions on the variables:
  \begin{align}
    \min \{c^\T x : x \in P \}. \tag{LP}\label{LP}
  \end{align}
Let $\lpopt$ denote an optimal basic feasible solution to \eqref{LP}.
For simplicity, we assume throughout that $P$ is a full-dimensional pointed polyhedron,
but our results extend to the general case with minor modifications.

Let $S$ be a \emph{$P_I$-free convex set}, a closed convex set
such that its \emph{interior}, denoted $\interior{S}$, contains no points of $P_I$,
and suppose that $\lpopt \in \interior{S}$.
A commonly used such set is
formed from a \emph{simple} split disjunction $(x_k \le \floor{\lpopt_k}) \lor (x_k \ge \ceil{\lpopt_k})$ on a variable $x_k$, $k \in I$:
  \[
    \Sk := \{x: \floor{\lpopt_k} \le x_k \le \ceil{\lpopt_k}\}.
  \]
Let $\opt\NB := \NB(\lpopt)$ denote the set of nonbasic variables at $\lpopt$.
Let $\optcone := C(\opt\NB)$ be the polyhedral cone with apex at $\lpopt$ and defined by the $n$
hyperplanes corresponding to the nonbasic variables $\opt\NB$;
denote these hyperplanes by $\optconehplanes$.
Let $\optconerays$ be the rays of $\optcone$.

GICs generalize SICs, as the cone $\optcone$ is but one possible relaxation of $P$.
$\optcone$ has $n$ extreme rays that can be intersected with $\bd S$
to obtain a set of intersection points $\initpointset$
and a set of rays $\initrayset$ of $\optcone$ that do not intersect $\bd S$, 
so that
  \begin{align*}
    &\initrayset :=  \{r \in \optconerays: r \cap \bd S = \emptyset\},\\
    &\initpointset :=  \{p^r: p^r := r \cap \bd S,\ r \in \optconerays \setminus \initrayset\}.
  \end{align*}
The intersection points in $\initpointset$ and rays in $\initrayset$ uniquely define the 
SIC, $\initcut$, obtained from $\optcone$ and $S$~\cite{Balas79}.

Letting $\hplaneset$ denote the set of hyperplanes of $P$,
$\optcone$ can be replaced by a tighter relaxation of $P$, denoted $C$,
obtained by activating
a subset of hyperplanes from $\hplaneset \setminus \optconehplanes$, i.e.,
of those that define $P$ but not $\optcone$. 
Subsequently, edges of $C$ can be intersected with $\bd S$ to 
obtain a set of intersection points $\pointset$, 
and edges that do not intersect $\bd S$ yield a set of rays $\rayset$.
In contrast to the situation when using $\optcone$, the number of intersection 
points and rays obtained from $C$ intersected with $\bd S$ can be strictly greater than $n$,
meaning the point-ray collection
defines not just one cut, but a collection of valid cuts.

Theorem~4 in \cite{BalMar13} defines valid GICs as these cuts
and propose to generate them as
any $(\alpha, \beta) \in \R^n \times \R$ 
satisfying $\alpha^\T \lpopt < \beta$ and 
  \begin{align} \label{CutRegion}
    \begin{aligned}
      p^\T \alpha &\ge \beta &&\text{for all $p \in \pointset$} \\
      r^\T \alpha &\ge 0 &&\text{for all $r \in \rayset$}.
    \end{aligned}    
  \end{align}
For a fixed $\beta$, define $\AlphaSystem(\beta,\pointset,\rayset)$ 
as the vectors $\alpha$ feasible to the above system,
i.e., the coefficients for inequalities with lower-bound $\beta$ that are valid for $\pointset$ and $\rayset$.
It suffices to consider $\beta \in \{-1,0,1\}$ to obtain all possible cuts.

Unfortunately, 
the above method requires maintaining the $\mathcal{V}$-polyhedral description of $C$.
The size of this $\mathcal{V}$-polyhedral description,
as well as the number of rows of \eqref{CutRegion},
as shown in Proposition~\ref{prop:FHA-exponential},
can grow exponentially large in the number of hyperplanes defining $C$.
Together, these two issues make the full hyperplane activation approach unviable.
\begin{proposition}
\label{prop:FHA-exponential}
  Let $C$ be formed by adding $k_h$ hyperplanes to the description of $\optcone$,
  and let $(\pointset,\rayset)$ denote the point-ray collection obtained 
  from intersecting the edges of $C$ with $\bd S$.
  The cardinality of $\pointset$ can grow exponentially large in $k_h$.
\end{proposition}
\begin{proof}
  Suppose we start with $\optcone$ and consider the activation of a halfspace $H^+$.
  Let $\raysIntByH(H)$ denote the rays from $\optconerays$ that are intersected by $H$ before $\bd S$.
  For each ray $r \in \raysIntByH(H)$ that intersects $H$, 
  $n - 1 - \card{\raysIntByH(H)}$ new edges are created 
  (not counting the edge back to $\lpopt$ 
  or the edges between new vertices created by activating $H$).
  If $\card{\raysIntByH(H)} \approx {n/2}$, 
  the size of the new point-ray collection will be $O(n^2)$.
  The desired result follows inductively.
\end{proof}

%% SECTION 3: PARTIAL HYPERPLANE ACTIVATION
\section{Partial hyperplane activation}
\label{sec:PHA}
In this section, we propose PHA, 
an approximate activation method that overcomes the exponential growth in the size of the point-ray collection when using full hyperplane activation.
Specifically, our algorithm generates at most $O(k_r^2 k_h)$ intersection points and rays,
where $k_r$ and $k_h$ are both parameters;
$k_r \le n$ is the maximum number of initial intersection points 
and rays that we remove via hyperplane activations
and $k_h$ is the number of activated hyperplanes.

Showing the validity of PHA requires a strict extension of the tools used in \cite{BalMar13} to prove the validity of full hyperplane activation.
We give this extension in Section~\ref{sec:proper} and employ it in Section~\ref{sec:RHA}, in which we describe a concrete variant of PHA called {\PHAOne} and prove its validity.
To facilitate reading, Table~\ref{tab:notation} gives a summary of some of the most frequently used notation that is used in the paper.
\begin{table}[ht!]
  \small
  \centering
  \caption{Summary of frequently used notation.}%\vspace{\tablecaptionvspace}
  \label{tab:notation}
  \begin{tabular}{ll}
    \toprule
    Notation & Description \\
    \midrule
    $\hplaneset$ & Hyperplanes defining $P$ \\
    $H_h$, $H_h^+$, $H_h^-$ & $H_h = \{x: a_h^\T x = b_h\}$, $H_h^+ = \{x: a_h^\T x \ge b_h\}$, $H_h^- = \R^n \setminus H_h^+$ \\
    $\opt\NB$ & Set of nonbasic variables defining $\lpopt$ \\
    $\optcone$ & Simple cone defined by the nonbasic variables at $\lpopt$ \\
    $\optconerays$ & Rays of $\optcone$ \\
    $\opt\hplanes$ & Hyperplanes defining $\optcone$ \\
    $(\pointset,\rayset)$ & Set of points on $\bd S$ and rays not intersecting $\bd S$ \\
    $(\initpointset, \initrayset)$ & 
    	Initial point-ray collection from $\optcone \cap \bd{S}$ \\
    $\AlphaSystem(\bar\beta, \pointset, \rayset)$ & Feasible region of \eqref{CutLP} \\
    $\rayset_\pointset$ & $\{p - \lpopt : p \in \pointset\}$ \\
    $\pointcone(\pointset,\rayset)$ & $\lpopt + \cone(\rayset_\pointset \cup \rayset)$ \\
    $\cutpolyhedron(\pointset,\rayset)$ & $\conv(\pointset) + \cone(\rayset)$ \\
    $\hplanes(r)$& Hyperplanes defining a ray $r$ \\
    $\raysIntByH(H_h)$ & Rays from $\optconerays$ intersected by $H_h$ before $\bd{S}$ \\
    \bottomrule
  \end{tabular}
\end{table}

\subsection{Proper point-ray collections}
\label{sec:proper}
We define a point-ray collection as a pair $(\pointset,\rayset)$ of points and rays obtained from intersecting the polyhedron $C$ with $\bd{S}$,
where $C$ is a relaxation of $P$.
Let $\mathcal{K}$ denote the skeleton of the polyhedron $C$.
Let $\mathcal{K}'$ denote the connected component of $\mathcal{K} \cap \interior S$ that includes $\lpopt$
and $\mathcal{K}''$ denote the union of the other components of $\mathcal{K} \cap \interior S$.

\begin{definition} \label{defn:proper}
  The point-ray collection $(\pointset, \rayset)$
  is called \emph{proper} if $\alpha^\T x \ge \beta$ is valid for $P_I$ whenever
  $(\alpha,\beta)$ is feasible to \eqref{CutRegion} and $\alpha^\T v < \beta$ for some
    $
      v \in \mathcal{K}'.
    $
\end{definition}

With this definition,
Theorem~\ref{thm:validIneq} shows that full hyperplane activation creates a proper point-ray collection,
which follows easily from the proof of Theorem~4 in \cite{BalMar13}.

\begin{theorem}\label{thm:validIneq}
  The point-ray collection $(\pointset, \rayset)$ 
  obtained from intersecting all edges of $C$ with $\bd S$, 
  where $C$ is defined by a subset of the hyperplanes defining $P$, 
  is proper.
\end{theorem}

Given a proper point-ray collection,
a GIC is defined as a facet of the convex hull of points and rays,
  $
    \cutpolyhedron(\pointset,\rayset) := \conv(\pointset) + \cone(\rayset),
  $
referred to by $\cutpolyhedron$ for short,
that cuts off a vertex of $\mathcal{K}'$.
\citet{BalMar13} show that these facets 
can be generated
by solving the following 
linear program formed from the given points and rays and optimized with respect to
a given objective direction $w \in \R^n$
and a fixed $\beta \in \{-1,0,1\}$:
\begin{align}\label{CutLP}\tag{PRLP} 
  \min\{w^\T \alpha : \alpha \in \AlphaSystem(\beta,\pointset,\rayset)\}.
\end{align}  

In fact, we show something stronger in Lemma~\ref{lem:valid_inequalities}:\ 
the cuts we obtain are not only valid for $P_I$, but also for $\conv(C \setminus \interior S)$.
This extends Theorem~\ref{thm:validIneq}
and will be invoked in proving the results in Section~\ref{sec:RHA}.
For simplicity, we assume that the vertex $v \in \mathcal{K}'$ being cut in Definition~\ref{defn:proper} is $\lpopt$ in the rest of the paper unless specified otherwise.

\begin{lemma} \label{lem:valid_inequalities}
  Let $(\pointset, \rayset)$ be the point-ray collection
  obtained from intersecting all edges of $C$ with $\bd S$.
  If $(\alpha,\beta)$ is feasible to $\eqref{CutRegion}$
  and $\alpha^\T \lpopt < \beta$, then
    \[ F := \{x \in C : \alpha^\T x < \beta\} \subseteq \interior S. \]
\end{lemma}
\begin{proof}
%%  Argument from 2011 paper:
%%%  Consider the region $F := \{x \in Q: \bar\alpha^\T x < \bar\beta\}$.
%%  Assume for the sake of contradiction that $F \cap \bd S \ne \emptyset$.
%%  We claim that there is an intersection point or ray that is in $F \cap \bd S$.
%%  This clearly would be the desired contradiction.
%%
  Assume for the sake of contradiction that $F \setminus \interior S$ is nonempty.
  Then $F \cap \bd S$ is nonempty.
  We show that this implies an intersection point or ray of $C$ will violate $\alpha^\T x \ge \beta$.
  
  First, note that $\cl F \cap \bd S$ cannot contain any vertex violating the inequality $\alpha^\T x \ge \beta$,
  as each vertex of $\cl F \cap \bd S$ is an intersection point of $C$.
  Hence, as we are also assuming $F \cap \bd S$ is nonempty, 
  the linear program $\min \{\alpha^\T x : x \in \cl F \cap \bd S\}$ must be unbounded.
  Consider initializing this program at $\lpopt$ and proceeding by the simplex method
  using the steepest edge rule.
  As the simplex method follows edges of $C$ and it will never increase the objective value,
  we will either eventually intersect $\bd S$ or end up on a ray of $C$ that is parallel to $\bd S$.
  This is the desired contradiction:
  in the former case, this is an intersection point that violates the inequality,
  and in the latter case, this is a ray of $C$ that does not intersect $\bd S$ and violates the inequality.
\end{proof}

\subsection{Algorithm and validity}
\label{sec:RHA}

We describe the {\PHAOne} approach for activating a hyperplane partially and prove that it yields a proper point-ray collection.
The algorithm proceeds similarly to the full hyperplane activation procedure,
in that it activates a hyperplane valid for $P$, but unlike the procedure from \cite{BalMar13}, 
{\PHAOne} can use any hyperplane valid for $P$ (not necessarily one that defines $P$),
and the activation is performed on the initial relaxation $\optcone$ 
instead of an iteratively refined relaxation that includes previously activated hyperplanes.
We show this is not only valid but also computationally advantageous.
Part of this advantage comes from that fact that {\PHAOne} only requires primal simplex pivots in the tableau of the LP relaxation to calculate new intersection points and rays.

The tradeoff for the computational advantages of {\PHAOne} is that it potentially creates a weaker point-ray collection compared to full activation,
in the sense that the cuts generated from this collection may be weaker.
This is because all vertices created from the partial activations are restricted to being only one pivot away from $\lpopt$,
i.e., at (edge) distance~1 from $\lpopt$ along the skeleton of $P$.
For this reason, we refer to Algorithm~\ref{alg:RHA1} as {\PHAOne}.
A stronger version of {\PHAOne}, such as
a valid distance 2 (or higher distance) procedure, 
is not difficult to create. 
In other words, one could define a hierarchy of PHA procedures of increasing distance using the principles presented below,
eventually leading to the full activation procedure.

  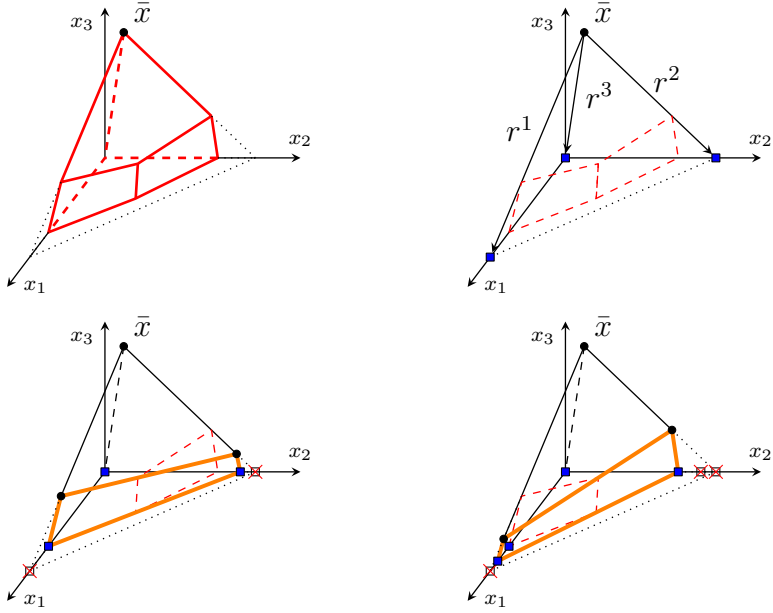
\begin{figure}
  \centering
	%%%%%%%%%%%%%%%%%%%%%%
	%Picture of P
	%%%%%%%%%%%%%%%%%%%%%%
    \begin{tikzpicture}[line join=round,x={(1 cm, 0 cm)},y={(0 cm, 2 cm)},z={(-0.5 cm,-0.66 cm)},>=stealth,scale=2] 
     %% Points
     \coordinate (barx) at (1/4,1/2,1/4); 

     \coordinate (p1) at (1,0,0);
     \coordinate (p2) at (0,0,1);
     \coordinate (p3) at (0,0,0);
     
     \coordinate (p4) at (3/4,0,0);
     \coordinate (p5) at (0,0,3/4);
     \coordinate (p6) at (9/22,0,9/22);
     
     \coordinate (p7) at (0,0,9/10);
     \coordinate (p8) at (9/10,0,0);
     
     \coordinate (v1) at (3/4,1/6,1/12);
     \coordinate (v2) at (1/12,1/6,3/4);
     \coordinate (v3) at (7/16,1/8,7/16);
     \coordinate (v4) at (1/28,1/14,25/28);
     \coordinate (v5) at (25/28,1/14,1/28);

      %% axes
      \draw[axis] (p4) -- (xyz cs:x=1.3) node[label={[label distance=-4pt]90:\scriptsize $x_2$}] {};
      \draw[axis] (xyz cs:y=-0) -- (xyz cs:y=0.5) node[label={[label distance=-6pt]-135:\scriptsize $x_3$}] {};
      \draw[axis] (p5) -- (xyz cs:z=1.3) node[label={[label distance=-2pt]0:\scriptsize $x_1$}] {};
      
      %%inequalities
      \draw [polyhedron_edge,red,line width=1pt] (v1) -- (barx) -- (v2) -- (v3) -- (v1);
      \draw [polyhedron_edge,red,line width=1pt] (v1) -- (p4) -- (p6) -- (p5) -- (v2);
      \draw [polyhedron_edge,red,line width=1pt] (p6) -- (v3) node [very near start,fill=none,left=0pt] {};
      \draw [polyhedron_edge,red,dashed,line width=1pt] (barx) -- (p3) node [very near start,fill=none,left=0pt] {};
      \draw [polyhedron_edge,red,dashed,line width=1pt] (p4) -- (p3) -- (p5);
      
      %Mesh of Cbar
       \draw [polyhedron_edge,dotted] (v2) -- (p2) ;
       \draw [polyhedron_edge,dotted] (v1) -- (p1) ;
       \draw [polyhedron_edge,dotted] (p1) -- (p2) ;
        \draw [polyhedron_edge,dotted] (p4) -- (p1) ;
        \draw [polyhedron_edge,dotted] (p5) -- (p2) ;
        
      \node [draw, point, fill=black, label={[label distance=-2pt]45:\normalsize $\bar x$}] at (barx) {};
    \end{tikzpicture}
    \qquad\qquad
	%%%%%%%%%%%%%%%%%%%%%%
	%Picture of Cbar
	%%%%%%%%%%%%%%%%%%%%%%
    \begin{tikzpicture}[line join=round,x={(1 cm, 0 cm)},y={(0 cm, 2 cm)},z={(-0.5 cm,-0.66 cm)},>=stealth,scale=2] 
      %% axes
      \draw[axis] (xyz cs:x=0) -- (xyz cs:x=1.3) node[label={[label distance=-4pt]90:\scriptsize $x_2$}] {};
      \draw[axis] (xyz cs:y=-0) -- (xyz cs:y=0.5) node[label={[label distance=-6pt]-135:\scriptsize $x_3$}] {};
      \draw[axis] (xyz cs:z=0) -- (xyz cs:z=1.3) node[label={[label distance=-2pt]0:\scriptsize $x_1$}] {};
      
      %%inequalities
      \draw [polyhedron_edge,dotted] (p1) -- (p2) node [very near start,fill=none,left=0pt] {};
      \draw [ray] (barx) -- (p1) node[label={[label distance=-4pt,pos=.5]45:\normalsize $r^2$}] {};
      \draw [ray] (barx) -- (p2) node[label={[label distance=-4pt,pos=.5]135:\normalsize $r^1$}] {};
      \draw [ray] (barx) -- (p3) node[label={[label distance=0pt,pos=.5]0:\normalsize $r^3$}] {};
      
      %Mesh of P
      \draw [polyhedron_edge,red,dashed] (p6) -- (p4) -- (v1) -- (v3) -- cycle;
      \draw [polyhedron_edge,red,dashed] (p5) -- (v2) -- (v3) -- (p6) -- cycle;
      
      %Vertices
       \node [draw, point, fill=black, label={[label distance=-2pt]45:\normalsize $\bar x$}] at (barx) {};
       
      %Intersection points
      \node [draw, intersection_point, fill=blue] at (p1) {};
     \node [draw, intersection_point, fill=blue] at (p2) {};
     \node [draw, intersection_point, fill=blue] at (p3) {};
    \end{tikzpicture}
  
	%%%%%%%%%%%%%%%%%%%%%%%%
	%Full activation with H4
	%%%%%%%%%%%%%%%%%%%%
    \begin{tikzpicture}[line join=round,x={(1 cm, 0 cm)},y={(0 cm, 2 cm)},z={(-0.5 cm,-0.66 cm)},>=stealth,scale=2] 
      %% axes
%      \draw[axis] (p8) -- (xyz cs:x=1.3) node[label={[label distance=-4pt]90:\scriptsize $x_2$}] {};
      \draw[axis] (p3) -- (xyz cs:x=1.3) node[label={[label distance=-4pt]90:\scriptsize $x_2$}] {};
      \draw[axis] (xyz cs:y=-0) -- (xyz cs:y=0.5) node[label={[label distance=-6pt]-135:\scriptsize $x_3$}] {};
%      \draw[axis] (p5) -- (xyz cs:z=1.3) node[label={[label distance=-2pt]0:\scriptsize $x_1$}] {};
      \draw[axis] (p3) -- (xyz cs:z=1.3) node[label={[label distance=-2pt]0:\scriptsize $x_1$}] {};
      
      %%inequalities
      \draw [polyhedron_edge,black] (v2) -- (barx) node [very near start,fill=none,left=0pt] {};
      \draw [polyhedron_edge,dashed] (barx) -- (p3) node [very near start,fill=none,left=0pt] {};
%      \draw [polyhedron_edge,dashed] (p3) -- (p5) node [very near start,fill=none,left=0pt] {};
%      \draw [polyhedron_edge,dashed] (p3) -- (p8) node [very near start,fill=none,left=0pt] {};
      \draw [polyhedron_edge] (barx) -- (v5) node [very near start,fill=none,left=0pt] {};
      
      %Mesh of P      
      \draw [polyhedron_edge,red,dashed] (p6) -- (p4) -- (v1) -- (v3) -- cycle;

      %H4 boundary
      \draw [polyhedron_edge,line width=1.5pt,orange] (v2) -- (p5) -- (p8) -- (v5) -- cycle;
      
       %Mesh of Cbar
       \draw [polyhedron_edge,dotted] (v2) -- (p2) ;
       \draw [polyhedron_edge,dotted] (v5) -- (p1) ;
       \draw [polyhedron_edge,dotted] (p1) -- (p2) ;
        \draw [polyhedron_edge,dotted] (p8) -- (p1) ;
        \draw [polyhedron_edge,dotted] (p5) -- (p2) ;
        
      %Vertices
      \node [draw, point, label={[label distance=-2pt]45:\normalsize $\bar x$}] at (barx) {};
       
      \node [draw, point] at (v2) {};
      \node [draw, point] at (v5) {};
      %Intersection points
     \node [draw, intersection_point, fill=blue] at (p3) {};
     
     \node [draw, intersection_point, fill=blue] at (p8) {};
     \node [draw, intersection_point, fill=blue] at (p5) {};
     
     \node [draw, intersection_point, fill=none] at (p1) {};
     \node [draw, cross=3pt, red] at (p1) {};
     \node [draw, intersection_point, fill=none] at (p2) {};
     \node [draw, cross=3pt, red] at (p2) {};
    \end{tikzpicture}
    \qquad\qquad
	%%%%%%%%%%%%%%%%%%%%%%%%
	%Partial activation with H5
	%%%%%%%%%%%%%%%%%%%%
    \begin{tikzpicture}[line join=round,x={(1 cm, 0 cm)},y={(0 cm, 2 cm)},z={(-0.5 cm,-0.66 cm)},>=stealth,scale=2] 
      %% axes
%      \draw[axis] (p4) -- (xyz cs:x=1.3) node[label={[label distance=-4pt]90:\scriptsize $x_2$}] {};
      \draw[axis] (p3) -- (xyz cs:x=1.3) node[label={[label distance=-4pt]90:\scriptsize $x_2$}] {};
      \draw[axis] (xyz cs:y=-0) -- (xyz cs:y=0.5) node[label={[label distance=-6pt]-135:\scriptsize $x_3$}] {};
%      \draw[axis] (p2) -- (xyz cs:z=1.3) node[label={[label distance=-2pt]0:\scriptsize $x_1$}] {};
      \draw[axis] (p3) -- (xyz cs:z=1.3) node[label={[label distance=-2pt]0:\scriptsize $x_1$}] {};
      
      %%inequalities
      \draw [polyhedron_edge,black] (v4) -- (barx);
      \draw [polyhedron_edge,dashed] (barx) -- (p3);
%      \draw [polyhedron_edge,dashed] (p3) -- (p2);
%      \draw [polyhedron_edge,dashed] (p3) -- (p4);
      \draw [polyhedron_edge] (barx) -- (v1);
      
      %Mesh of P
      \draw [polyhedron_edge,red,dashed] (p5) -- (v2) -- (v3) -- (p6) -- cycle;

      %H5 boundary
      \draw [polyhedron_edge,line width=1.5pt,orange] (p7) -- (v4) -- (v1) -- (p4) -- cycle;
      
       %Mesh of Cbar
       \draw [polyhedron_edge,dotted] (v2) -- (p2) ;
       \draw [polyhedron_edge,dotted] (v1) -- (p1) ;
       \draw [polyhedron_edge,dotted] (p1) -- (p2) ;
        \draw [polyhedron_edge,dotted] (p4) -- (p1) ;
        \draw [polyhedron_edge,dotted] (p5) -- (p2) ;
        
      %Vertices
      \node [draw, point, label={[label distance=-2pt]45:\normalsize $\bar x$}] at (barx) {};
       
      \node [draw, point] at (v1) {};
      \node [draw, point] at (v4) {};
      %Intersection points
     \node [draw, intersection_point, fill=blue] at (p3) {};
     
     \node [draw, intersection_point, fill=none] at (p8) {};
     \node [draw, cross=3pt, red] at (p8) {};
     \node [draw, intersection_point, fill=none] at (p1) {};
     \node [draw, cross=3pt, red] at (p1) {};
     \node [draw, intersection_point, fill=none] at (p2) {};
     \node [draw, cross=3pt, red] at (p2) {};
     
     \node [draw, intersection_point, fill=blue] at (p7) {};
     \node [draw, intersection_point, fill=blue] at (p4) {};
     \node [draw, intersection_point, fill=blue] at (p5) {};
    \end{tikzpicture}
    \caption{This figure illustrates the sequence of activations performed by \PHAOne{},
    where the cut-generating set is $\{x \in \R^3: 0 \le x_j \le 1, j \in \{1,2\}\}$.
    		The first panel shows $P$.
		The second panel depicts $\optcone$. %(which is also shown with dashed lines in the first panel).
		The bottom two panels show the sequential activations of the two hyperplanes not defining $\optcone$.
	}
    \label{fig:partial-activation}
  \end{figure} % partial activation

Before proceeding, we first provide an illustrative example, shown in Figure~\ref{fig:partial-activation}, of our procedure that guides the intuition for the subsequent formal proof.
The polytope in the first panel shows the feasible region of 
  $P := \{x \in \R^3 \suchthat -2 x_2 + x_3 \le 0; -2 x_1 + x_3 \le 0; 12 x_1 + 10 x_2 - 5 x_3 \le 9; 10 x_1 + 12 x_2 - 5 x_3 \le 9; x_1 + x_2 + x_3 \le 1\}$.
We denote the hyperplanes defining $P$ by $H_1$--$H_5$ in the order that they are listed in the definition of $P$. 
The next panel shows the cone $\optcone$, which is defined by $H_1$--$H_3$. 
The rays of $\optcone$ are denoted $r^1$ to $r^3$. 
The cut-generating set $S$ is the box, 
  $\{x \in \R^3: 0 \le x_j \le 1, j \in \{1,2\}\}$. 
We use circular and square nodes to represent vertices and intersection points, respectively. 
The remaining two panels illustrate the full and partial activation of hyperplanes $H_4$ and $H_5$ that do not play a role in defining $\optcone$.
The third panel shows the full activation of $H_4$. 
This activation cuts off the two intersection points corresponding to $r^1$ and $r^2$ intersected with $\bd{S}$, and replaces them with two intersection points on $H_4$. 
There are also two new vertices that are created --- since activation always occurs on $\optcone$, these vertices lie on rays of this cone.
The last panel shows $H_5$ partially activated on $\optcone$.
%This activation creates two new intersection points, while removing one of the intersection points from the activation of $H_4$.
This second activation creates a stronger intersection point along axis $x_2$, as well as a \emph{weaker} intersection point along axis $x_1$, and it removes one of the intersection points from the activation of $H_4$.

We now define some notation used in the algorithm description.
For a given $(a_h, b_h) \in \R^n \times \R$,
let $H_h := \{x : a_h^\T x = b_h\}$ be a hyperplane such that the corresponding halfspace
$H_h^+ := \{x : a_h^\T x \ge b_h\}$ is valid for $P$.
Every element of a proper point-ray collection $(\pointset,\rayset)$ arises from the intersection of an edge of some relaxation of $P$ with $\bd{S}$; if the starting vertex of that edge lies on a ray $r^j \in \optconerays$,
then we say that the point or ray \emph{originates} from $r^j$.
For each $r^j \in \optconerays$, let the distance to $H_h$ along the ray $r^j$ starting at $\lpopt$ be 
  \[
    \distToHplane{H_h}{r^j} := 
      \begin{cases}
        {(b_h - a_h^\T \lpopt})/{a_h^\T r^j} & \mbox{if $a_h^\T r^j > 0$}, \\
        \infty & \mbox{otherwise.}
      \end{cases}
  \]
We will slightly overload this notation to also refer to the distance along ray $r^j$ to $\bd{S}$, $\distToHplane{\bd{S}}{r^j}$.
This will be the distance along $r^j$
from $\lpopt$ to the nearest facet of $S$ (or $\infty$ if no facet of $S$ is intersected).
With these definitions, we present Algorithm~\ref{alg:RHA1},
which takes as an input the polyhedron $P$, the cut-generating set $S$,
the hyperplane $H_h$ being activated,
a set of rays $\rayset_A$ from $\optconerays$ that will be permitted to be cut by $H_h$ (a parameter we will use only later),
and any proper point-ray collection to be modified via the activation of $H_h$.
We also assume that along with the point-ray collection, we have kept a history of which edge led to each point and ray to be added to the collection.
This is used in the algorithm in step~\ref{step:RHA1:remove-viol-points} when deciding which points and rays to remove.

\begin{algorithm}[ht!]
\caption{Distance 1 Partial Hyperplane Activation}\label{alg:RHA1}
\begin{algorithmic}[1]
\Input Polyhedron $P$ defined by hyperplane set $\hplaneset$;
  $P_I$-free convex set $S$;
  hyperplane $H_h$ valid for $P$;
  set of rays $\rayset_A \subseteq \optconerays$;
  proper point-ray collection $(\pointset,\rayset)$
  contained in $\optcone$.
\Function{\PHAOne}{$P, S, H_h, \rayset_A, (\pointset,\rayset)$} \label{step:RHA1:setup}
    \ForAll{$r \in \rayset_A$ such that $\distToHplane{H_h}{r} < \distToHplane{\bd{S}}{r}$} \label{step:RHA1:act-hplane}
     \State \label{step:RHA1:remove-viol-points}
        Remove all points and rays in $(\pointset,\rayset)$
        that violate $H_h^+$ and originate from
        \Statex[3]
          a vertex on $r$
          but do not lie on or coincide 
          with any ray from $\optconerays \setminus \rayset_A$.
      \ForAll{new edges $e$ originating at vertex $r \cap H_h$} \label{step:RHA1:process-ray}
        \If
        %{$\distToHplane{H}{r'} < \distToHplane{\bd{S}}{r'}$ for some $H \in \optconehplanes \setminus \hplanes(r)$} 
        {$e$ (emanating from $r \cap H_h$) does not intersect any $H \in \optconehplanes \setminus \hplanes(r)$ prior to $\bd{S}$} \label{step:RHA1:is-ray-edge}
%       \State
%            Skip this ray.
        \If{$e \cap \bd{S} \ne \emptyset$}
            Add intersection point of $e$ with $\bd{S}$ to $\pointset$. %$p := r' \cap \bd{S}$ to $\pointset$.
            \label{step:RHA1:inter-pt}
        \Else \label{step:RHA1:inter-ray}
            Add the ray $e$ to $\rayset$.
          \EndIf
        \EndIf
      \EndFor
    \EndFor
  \State\Return $\pointset$ and $\rayset$. \label{step:RHA1:return}
\EndFunction
\end{algorithmic}
\end{algorithm}

Theorem~\ref{thm:RHA_valid} states the validity of Algorithm~\ref{alg:RHA1} for the special case when all rays of $\optcone$ are permitted to be cut.
We show how to relax this restriction in Section~\ref{sec:tilting}.

\begin{theorem} \label{thm:RHA_valid}
  Algorithm~\ref{alg:RHA1} when $\rayset_A = \optconerays$ outputs a proper point-ray collection.
\end{theorem}

In the rest of this section, we will build to a proof of Theorem~\ref{thm:RHA_valid} with the help of several useful intermediate results (Lemmas~\ref{lem:PinCbar}--\ref{lem:redundantPoints}).
The validity of {\PHAOne} follows from proving that the inequalities that can be generated from \eqref{CutLP} are valid for a relaxation of $P_I$ related only to $\lpopt$ and the points and rays in the collection.
Unlike the proof of full hyperplane activation, this does not rely on any particular relaxation of $P$ used to obtain the points and rays.
Given a point-ray collection $(\pointset, \rayset)$, let 
  $\rayset_\pointset := \{p - \lpopt : p \in \pointset\}$ 
and
  $
    \pointcone(\pointset,\rayset) := \lpopt + \cone(\rayset_\pointset \cup \rayset).
  $
That is, $\pointcone(\pointset,\rayset)$ is the polyhedral cone with apex at $\lpopt$
and rays including $\rayset$ and a ray from $\lpopt$ through each intersection point in $\pointset$.

\begin{lemma} \label{lem:PinCbar}
  If $(\pointset,\rayset)$ is a proper point-ray collection,
  then
    $
      P_I \subseteq \pointcone(\pointset,\rayset).
    $
\end{lemma}
\begin{proof}
  Otherwise there exists an inequality valid for $\cutpolyhedron$ that cuts both $\lpopt$ and a point of $P_I$.
  This contradicts the assumption that $(\pointset,\rayset)$ is proper.
\end{proof}

The next two lemmas state inclusion properties for the sets of cuts obtainable from two different point-ray collections.

\begin{lemma} \label{lemma:pointset_inclusion_pointcone}
  For $i \in \{1,2\}$, let $\pointset^i$ denote a set of points on $\bd{S}$
  and $\rayset^i$ a set of rays.
  If for every $(\alpha,\beta)$ such that $\alpha$ is feasible to $\AlphaSystem(\beta,\pointset^2,\rayset^2)$ and $\alpha^\T \lpopt < \beta$, 
  $\alpha$ is also feasible to
    $\AlphaSystem(\beta,\pointset^1,\rayset^1)$,
  then 
    \[
      \pointcone( \pointset^1, \rayset^1) 
        \subseteq \pointcone( \pointset^2, \rayset^2).
    \]
\end{lemma}
\begin{proof}
  Let $\pointcone^i := \pointcone( \pointset^i, \rayset^i)$, $i \in \{1,2\}$.  Assume for the sake of contradiction that
  there exists an extreme ray $r$ of $\pointcone^1$ that is not in $\pointcone^2$.
  If $r$ does not intersect $\bd{S}$, define $q := r$;
  otherwise,
  let $q := r \cap \bd{S}$, and note that $q \notin \pointcone^2$.
  Then there exists an inequality $\widehat\alpha^\T x \ge \widehat\beta$ that separates $\lpopt$ and $q$ from 
    $\cutpolyhedron(\pointset^2,\rayset^2)$.
  This is a contradiction, as $\widehat\alpha$ is feasible to $\AlphaSystem(\widehat\beta,\pointset^2,\rayset^2)$
  and $\widehat\alpha^\T \lpopt < \widehat\beta$, but it is not feasible to
    $\AlphaSystem(\widehat\beta,\pointset^1,\rayset^1)$,
  since it violates the inequality corresponding to $q$. 
\end{proof}

\begin{lemma} \label{lemma:relaxation_activation}
  Let $P^1$ and $P^2$ be rational convex polyhedra such that $P^1 \subseteq P^2$,
  and $\lpopt \in P^2$.
  Let $\pointset^i$ and $\rayset^i$ denote the intersection points and rays generated
  from the intersection of $P^i$ with $\bd{S}$, for $i \in \{1,2\}$.
  Then, for a given constant $\beta$,
  a vector $\alpha$ feasible to 
  $\AlphaSystem(\beta,\pointset^2,\rayset^2)$ and cutting $\lpopt$ 
  will also be feasible to $\AlphaSystem(\beta,\pointset^1,\rayset^1)$.
\end{lemma}
\begin{proof}
  Since $P^1 \subseteq P^2$, we have 
    $
      P^1 \cap \bd{S} \subseteq P^2 \cap \bd{S} \subseteq P^2 \setminus \interior S.
    $
  By Lemma~\ref{lem:valid_inequalities}, 
    $P^2 \setminus \interior S \subseteq \{x : \alpha^\T x \ge \beta\}$,
  which completes the proof.
\end{proof}

Next, we remark on one subtlety of Algorithm~\ref{alg:RHA1}, in step~\ref{step:RHA1:act-hplane}.
Namely, whenever the hyperplane being activated intersects a ray of $\optcone$ outside of $\interior{S}$, the algorithm skips that ray.
In Lemma~\ref{lem:redundantPoints}, we show this is valid because such intersections would lead to
only \textit{redundant} intersection points, in the sense that these points will never be part of a valid inequality that cuts away $\lpopt$.
We show the result only for points,
as rays are never redundant.
Recall that $\mathcal{K}$ denotes the skeleton of the polyhedron $C$,
$\mathcal{K}'$ denotes the connected component of $\mathcal{K} \cap \interior S$ that includes $\lpopt$,
and $\mathcal{K}''$ denotes the union of the other components of $\mathcal{K} \cap \interior S$.
Let $\pointset'$ and $\pointset''$ denote the intersection points between $\bd S$ 
and the edges of the closure of $\mathcal{K}'$ and $\mathcal{K}''$, respectively.

\begin{lemma} \label{lem:redundantPoints}
  Any intersection points in 
    $\pointset'' \setminus \pointset'$
  are redundant.
\end{lemma}
\begin{proof}
  Let $\alpha$ be any feasible solution to $\CutRegion(\beta,\pointset',\rayset)$ such that
  ${\alpha}^\T \lpopt < \beta$.
  Suppose for the sake of contradiction that a point $w \in \pointset'' \setminus \pointset'$ violates the inequality.
  Since $w \in \mathcal{K}''$ and $\lpopt \in \mathcal{K}'$,
  any path from $\lpopt$ to $w$ along the skeleton of $C$
  must intersect $\bd{S}$ at some point in $\mathcal{P}'$.
  Since $C$ is convex and both $\lpopt$ and $w$ are cut,
  there must exist such a path entirely in $\{x : \alpha^\T x < \beta\}$,
  implying some intersection point in $\pointset'$ is cut, a contradiction.
\end{proof}

We now use these results to prove Theorem~\ref{thm:RHA_valid}.

\begin{proof}[Proof of Theorem~\ref{thm:RHA_valid}]
  Let $(\pointset^*,\rayset^*)$ be the point-ray collection that Algorithm~\ref{alg:RHA1} outputs.
  Let $\pointcone^* := \pointcone(\pointset^*,\rayset^*)$.  
  Observe that intersecting the edges of $\pointcone^*$ with $\bd{S}$ yields exactly the same point-ray collection $(\pointset^*,\rayset^*)$.
  Then, using Lemma~\ref{lem:valid_inequalities} with $\pointcone^*$, 
  it follows that all obtainable cuts will be valid for $\pointcone^* \setminus \interior S$.
  To prove the theorem, it suffices to show the inclusion 
    $
      P_I \subseteq \pointcone^*,
    $
  which we do next.
  
  From Lemma~\ref{lem:PinCbar}, we have that $P_I \subseteq \pointcone := \pointcone(\pointset,\rayset)$,
  where $(\pointset,\rayset)$ is the proper point-ray collection given as an input to the algorithm.
  We need to show that the point-ray collection after activating $H_h$ is proper.
  Let $\pointset^c$ and $\rayset^c$ denote the set of points and rays
  from $(\pointset,\rayset)$ that violate $H_h^+$.

  First we consider intersection points that lie on $H_h$.
  Let $P^1 := \pointcone \cap H_h$ and $P^2 := \optcone \cap H_h$,
  with $\pointset^i$ and $\rayset^i$ the corresponding intersection points and rays
  of $P^i$ with $\bd{S}$, for $i \in \{1,2\}$.
  Since $\pointset$ and $\rayset$ are contained in $\optcone$,
  $\pointcone \subseteq \optcone$.
  As a result, we can apply Lemma~\ref{lemma:relaxation_activation},
  which implies that every $\alpha$
  (for a $\beta$ such that $\alpha^\T \lpopt < \beta$) that is feasible to
    $\AlphaSystem(\beta, \pointset^2, \rayset^2)$
  will also be feasible to
    $\AlphaSystem(\beta, \pointset^1, \rayset^1)$.

  Now we consider the entire halfspace $H_h^+$. 
  The set of intersection points of $\pointcone \cap H_h^+$ with $\bd{S}$ is 
    $\overline\pointset := \pointset^1 \cup \pointset \setminus \pointset^c$,
  and the set of rays is 
    $\overline\rayset := \rayset^1 \cup \rayset \setminus \rayset^c$.
  In addition, we have that
    $\pointset^* = \pointset^2 \cup \pointset \setminus \pointset^c$,
  and 
    $\rayset^* = \rayset^2 \cup \rayset \setminus \rayset^c$.
  Here we are assuming that all of the points in $\pointset^2$ belong to the same part of the skeleton of $\optcone \cap H_h^+$ that contains $\lpopt$.
  This is without loss of generality as a result of Lemma~\ref{lem:redundantPoints}.
  
  Hence, for any $\beta$, 
  suppose $\alpha$ is feasible to 
    $\AlphaSystem(\beta,\allowbreak \pointset^*,\allowbreak \rayset^*)$
  and satisfies $\alpha^\T \lpopt < \beta$.
  Then
  $\alpha$ will also be feasible to
    $\AlphaSystem(\beta,\overline\pointset,\overline\rayset)$.
  By Lemma~\ref{lemma:pointset_inclusion_pointcone},
    $\pointcone(\overline\pointset,\overline\rayset) \subseteq \pointcone^*$.
  Moreover, since
    $P_I \subseteq \pointcone$
  and $P_I \subseteq H_h^+$, we have that
    $P_I \subseteq \pointcone \cap H_h^+$.
  The intersection points and rays obtained from intersecting $\pointcone(\overline\pointset,\overline\rayset)$
  with $\bd{S}$ are precisely $\overline\pointset$ and $\overline\rayset$.
  As a result, by Lemma~\ref{lem:PinCbar},
    $P_I \subseteq \pointcone(\overline\pointset,\overline\rayset)$.
  It follows that $P_I \subseteq \pointcone^*$, as desired.
\end{proof}

Performing Algorithm~\ref{alg:RHA1} $k_h$ times is dramatically more efficient than the 
$O(n^{k_h})$ complexity of full hyperplane activation.
For every activation using Algorithm~\ref{alg:RHA1},
letting $k_r := \card{\rayset_A}$, the hyperplane $H_h$ can cut $O(k_r)$ rays to create $O(k_r^2)$ new intersection points.
Therefore, if $k_h$ hyperplanes are activated, 
the number of intersection points generated by partial activation is $O(k_r^2 k_h)$. 

In the presence of rays in the initial point-ray collection,
we next describe the role of the parameter $\rayset_A$ in {\PHAOne},
which restricts the set of rays from $\optcone$ that can be cut
by any activated hyperplane.
Edges of $C$ that do not intersect $\bd S$ can be viewed as intersection points infinitely far away, 
and \citet{BalMar13} show that this view suffices when extending $\AlphaBetaSystemPR$ 
with only points to a valid system with both points and rays. 
However, when generating GICs, rays play a fundamentally different role than points. 
  In particular,
  when activating a new hyperplane cuts some $p \in \initpointset$, 
  the new intersection points that are created are either on the SIC or are deeper relative to $p$
  (see Theorem~3 in \cite{BalMar13}).
  In contrast, any intersection points created as a result of cutting a ray from $\initrayset$
  will all lie on the SIC, as shown in Proposition \ref{prop:parallel-rays}.
\begin{proposition}
\label{prop:parallel-rays}
  Suppose $r \in \initrayset$ is a ray that does not intersect $\bd S$, and
  $H$ is an activated hyperplane that intersects $r$ at $v := r \cap H$.
  Denote by $Y$ the set of new rays on $H$ that originate from $v$. 
  Then the subset of rays $Y'$ in $Y$ that intersect $\bd S$ create intersections points on the 
  hyperplane $\initcutcoeff^\T x = \initcutRHS$, where $\initcut$ is the SIC from $\optcone$ and $S$.
  The rays $r'$ in $Y \setminus Y'$ satisfy $\initcutcoeff^\T r' = 0$.
\end{proposition}
\begin{proof}
  Let $\initrayset^c \subseteq \initrayset$ denote the set of initial rays cut by $H$ before $\bd S$.
  When ray $r$ intersects $H$, it creates $n - 1 - \abs{\initrayset^c}$ new rays emanating from $v$,
  which we denote by $Y$.
  There are $n-1$ hyperplanes that define $r$.
  Any ray $r' \in Y$ lies on $n-2$ of these hyperplanes, as well as on $H$.
  There exists a ray $\widehat r \in \initrayset \setminus \initrayset^c$ that also lies on these $n-2$ facets,
  and $r'$ can be written as a nontrivial conic combination of $r$ and $\widehat r$.
  Let $F$ denote the two-dimensional face of $\optcone$ 
  defined by all conic combinations of $r$ and $\widehat r$.
  
  Suppose $\widehat r$ intersects $\bd S$ and let $\widehat p := \widehat r \cap \bd S$.
  By the remark above,
  the half-line 
  $L := \widehat p + \theta r$, $\theta \ge 0$ is contained in 
    $\{x: \initcutcoeff^\T x = \initcutRHS\} \cap \bd S$,
  since $r \in \initrayset$ does not intersect $\bd S$.
  Further, $L$ is also contained in $F$, which is parallel to $r$, $\widehat r$, and $r'$.
  Since $r'$ is a nontrivial conic combination of $\widehat r$ and $r$, it intersects $\bd S$.
  The resulting intersection point is on $L$, which lies on the SIC. %and hence on the SIC.%
  
  Now suppose $\widehat r$ does not intersect $\bd S$.
  Then $\initcutcoeff^\T r = \initcutcoeff^\T \widehat r = 0$.
  Since $r'$ is a nontrivial conic combination of $r$ and $\widehat r$, it too satisfies 
    $\initcutcoeff^\T r' = 0.$ 
\end{proof}

The consequence of Proposition~\ref{prop:parallel-rays} is that activating hyperplanes
on rays of $\optcone$ that do not intersect $\bd{S}$ will create weak points,
as these points actually lie on the SIC, not deeper as desired.
Thus, the input $\rayset_A$ to Algorithm~\ref{alg:RHA1} will, in our implementation, always be chosen from the set of rays of $\optcone$  that intersect $\bd{S}$, i.e., from $\optconerays \setminus \initrayset$.
Choosing $\rayset_A$ arbitrarily may in general lead to invalid cuts;
the validity of our choice follows from the results in Section~\ref{sec:tilting}.
The full details of how we implement {\PHAOne} to generate cuts are given in Algorithm~\ref{alg:PHA}.

\begin{algorithm}[ht!]
\caption{Generalized Intersections Cuts by {\PHAOne}}
\label{alg:PHA}
\begin{algorithmic}[1]
\Input Polyhedron $P$ defined by a set of hyperplanes $\hplaneset$;
%  set $\intvars \subseteq [n]$ denoting the integer variables;
  a vertex $\lpopt$ of $P$;
  indices of fractional integer variables $\sigma$;
  hyperplane selection criterion $\mathcal{SC}$;
%  objective coefficient vector $c$;
 objectives $\mathcal{O}$;
 number hyperplanes $k_h$.
% fractionality threshold $\epsilon_{\text{frac}}$.
\Function{{\PHAOne}CutGenerator}{$P, \lpopt, \sigma, \mathcal{SC}, \mathcal{O}, k_h$} %\epsilon_{\text{frac}}$} \label{step:PHA:setup}
  \State
    $\mathcal{C} \gets \emptyset$.
  \For{$k \in \sigma$}
  	\State
		$S_k \gets \{x \suchthat \floor{\lpopt_k} \le x_k \le \ceil{\lpopt_k}\}$.
    \State
      $\rayset_k^{\parallel} \gets \{r \in \optconerays \suchthat \distToHplane{\bd{S_k}}{r} = \infty\}$;
      $\rayset_k \gets \rayset_k^{\parallel}$;
      $\pointset_k \gets \{r \cap \bd{S_k} \suchthat r \in \optconerays \setminus \rayset_k\}$.
  \EndFor
  \For{$h \in \{1,\ldots,k_h\}$} \label{step:PHA:RHA-start}
    \For{$k \in \sigma$}
      \State
      \label{step:PHA:RHA-choose-hplane}
        Choose a hyperplane $H$ to activate according to selection criteria $\mathcal{SC}$.
      \State \label{step:PHA:RHA-act}
        $(\pointset_k, \rayset_k) \gets \text{\PHAOne}(P, S_k, H, \optconerays \setminus \rayset_k^{\parallel}, (\pointset_k,\rayset_k))$.
    \EndFor
    \State \label{step:PHA:RHA-end}
        Add to $\mathcal{C}$ valid cuts by
        solving \eqref{CutLP}
        with objective types $\mathcal{O}$, always ensuring $\lpopt$ is cut.
  \EndFor
  \State\Return $\mathcal{C}$.
\EndFunction
\end{algorithmic}
\end{algorithm}

\section{Partial hyperplane activation with tilted hyperplanes}
\label{sec:tilting}

An outcome of proving the validity of \PHAOne{} is that the hyperplane given as an input to Algorithm~\ref{alg:RHA1} can be an arbitrary valid hyperplane for $P$, not necessarily one from the hyperplane description of $P$.
One application of this insight
is to choose hyperplanes that only cut rays of $\optcone$ that intersect $\bd{S}$,
to avoid the weak intersection points that would otherwise be created as shown in Proposition~\ref{prop:parallel-rays}.
However, it is not practical to search for arbitrary valid hyperplanes that satisfy specific desired properties such as which particular rays of $\optcone$ are or are not cut.
Instead, we propose to activate the hyperplanes defining $P$, but only on a subset of the rays that they intersect.
The validity of this follows from a geometric argument, in which Algorithm~\ref{alg:RHA1} is applied to a \emph{tilted} hyperplane.
Crucially, we show that we never need to explicitly compute this tilted hyperplane;
rather, the activation can be computed using the original non-tilted hyperplane, and therefore without incurring any additional computational burden.
A byproduct of the tilting theory is that choosing $\rayset_A$ as in Algorithm~\ref{alg:PHA} is valid.

Using the concept of tilting, in Algorithm~\ref{alg:targeted-tilting}, we provide an alternative, but not mutually exclusive, approach to Algorithm~\ref{alg:PHA} for using {\PHAOne}.
This algorithm uses tilting to reduce the size of the point-ray collection when the number of rays of $\optcone$ that are cut is large, which is desirable when seeking cuts stronger than the SIC.
When cutting $k_r$ rays and activating $k_h$ hyperplanes, 
the algorithm creates $O(k_r k_h^2)$ points and rays, which has a clear advantage to the $O(k_r^2 k_h)$ points and rays that are created from Algorithm~\ref{alg:PHA} when $k_r$ is large.
In our implementation in Section~\ref{sec:computation}, we compare the two ideas when used individually and together.

 \begin{figure}
 \centering
  %%%%%%%%%%%%%%%%%%%%%%%%
  %Partial activation with H4
  %%%%%%%%%%%%%%%%%%%%
    \begin{tikzpicture}[line join=round,x={(1 cm, 0 cm)},y={(0 cm, 2 cm)},z={(-0.5 cm,-0.66 cm)},>=stealth,scale=2] 
      %% Points
     \coordinate (barx) at (1/4,1/2,1/4); 

     \coordinate (p1) at (1,0,0);
     \coordinate (p2) at (0,0,1);
     \coordinate (p3) at (0,0,0);
     
     \coordinate (p4) at (3/4,0,0);
     \coordinate (p5) at (0,0,3/4);
     \coordinate (p6) at (9/22,0,9/22);
     
     \coordinate (p7) at (0,0,9/10);
     \coordinate (p8) at (9/10,0,0);
     
     \coordinate (p9) at (3/7,0,3/7);
     
     \coordinate (v1) at (3/4,1/6,1/12);
     \coordinate (v2) at (1/12,1/6,3/4);
     \coordinate (v3) at (7/16,1/8,7/16);
     \coordinate (v4) at (1/28,1/14,25/28);
     \coordinate (v5) at (25/28,1/14,1/28);
     \coordinate (v6) at (9/20,1/10,9/20);

      %% axes
%      \draw[axis] (p1) -- (xyz cs:x=1.3) node[label={[label distance=-4pt]90:\scriptsize $x_2$}] {};
		\draw[axis] (p3) -- (xyz cs:x=1.3) node[label={[label distance=-4pt]90:\scriptsize $x_2$}] {};
      	\draw[axis] (xyz cs:y=-0) -- (xyz cs:y=0.5) node[label={[label distance=-6pt]-135:\scriptsize $x_3$}] {};
%      \draw[axis] (p5) -- (xyz cs:z=1.3) node[label={[label distance=-2pt]0:\scriptsize $x_1$}] {};
		\draw[axis] (p3) -- (xyz cs:z=1.3) node[label={[label distance=-2pt]0:\scriptsize $x_1$}] {};
      
      %%inequalities
      \draw [polyhedron_edge] (v2) -- (barx) node [very near start,fill=none,left=0pt] {};
      \draw [polyhedron_edge,dashed] (barx) -- (p3) node [very near start,fill=none,left=0pt] {};
%      \draw [polyhedron_edge,dashed] (p3) -- (p5) node [very near start,fill=none,left=0pt] {};
%      \draw [polyhedron_edge,dashed] (p3) -- (p4) node [very near start,fill=none,left=0pt] {};
      \draw [polyhedron_edge,dashed] (p4) -- (p8) node [very near start,fill=none,left=0pt] {};
      \draw [polyhedron_edge] (barx) -- (p1) node [very near start,fill=none,left=0pt] {};
      %\draw [-,line width=0.5pt,black] (v5) -- (v1) node [very near start,fill=none,left=0pt] {};
      
      %% P 
      \draw [polyhedron_edge,red,dashed] (p6) -- (p4) -- (v1) -- (v3) -- cycle;
      \draw [polyhedron_edge,red,dashed] (p5) -- (v2) -- (v3) -- (p6) -- cycle;
       
       %H4 boundary
%       \draw [polyhedron_edge,orange,line width=0.25pt] (v2) -- (p5) -- (p8) -- (v5) -- cycle;
     
      %Partial H4 boundary
      \draw [polyhedron_edge,line width=1.5pt,orange] (v2) -- (p5) -- (p1) -- cycle;
            
       %Mesh of Cbar
       \draw [polyhedron_edge,dotted] (v2) -- (p2) node [very near start,fill=none,left=0pt] {};
       \draw [polyhedron_edge,dotted] (p1) -- (p2) node [very near start,fill=none,left=0pt] {};
        \draw [polyhedron_edge,dotted] (p8) -- (p1) node [very near start,fill=none,left=0pt] {};
        \draw [polyhedron_edge,dotted] (p5) -- (p2) node [very near start,fill=none,left=0pt] {};
        
        %Vertices
       \node [draw, point, label={[label distance=-2pt]45:\normalsize $\bar x$}] at (barx) {};
       \node [draw, point, fill=black] at (v2) {};
    
      %Intersection points
     \node [draw, intersection_point, fill=blue] at (p3) {};
     \node [draw, intersection_point, fill=blue] at (p1) {};
     \node [draw, intersection_point, fill=blue] at (p5) {};
     
     \node [draw, intersection_point, fill=none] at (p2) {};
     \node [draw, cross=3pt, red] at (p2) {};
    \end{tikzpicture}
    \qquad\qquad
    \begin{tikzpicture}[line join=round,x={(1 cm, 0 cm)},y={(0 cm, 2 cm)},z={(-0.5 cm,-0.66 cm)},>=stealth,scale=2] 
      %% axes
%      \draw[axis] (p4) -- (xyz cs:x=1.3) node[label={[label distance=-4pt]90:\scriptsize $x_2$}] {};
      \draw[axis] (p3) -- (xyz cs:x=1.3) node[label={[label distance=-4pt]90:\scriptsize $x_2$}] {};
      \draw[axis] (xyz cs:y=-0) -- (xyz cs:y=0.5) node[label={[label distance=-6pt]-135:\scriptsize $x_3$}] {};
%      \draw[axis] (p2) -- (xyz cs:z=1.3) node[label={[label distance=-2pt]0:\scriptsize $x_1$}] {};
      \draw[axis] (p3) -- (xyz cs:z=1.3) node[label={[label distance=-2pt]0:\scriptsize $x_1$}] {};
      
      %%inequalities
      \draw [polyhedron_edge,black] (v4) -- (barx);
      \draw [polyhedron_edge,dashed] (barx) -- (p3);
%      \draw [polyhedron_edge,dashed] (p3) -- (p2);
%      \draw [polyhedron_edge,dashed] (p3) -- (p4);
      \draw [polyhedron_edge] (barx) -- (v1);
      
      %Mesh of P
      \draw [polyhedron_edge,red,dashed] (p5) -- (v2) -- (v3) -- (p6) -- cycle;

      %H5 boundary
      \draw [polyhedron_edge,line width=1.5pt,orange] (p7) -- (v4) -- (v1) -- (p4) -- cycle;
      
       %Mesh of Cbar
       \draw [polyhedron_edge,dotted] (v2) -- (p2) ;
       \draw [polyhedron_edge,dotted] (v1) -- (p1) ;
       \draw [polyhedron_edge,dotted] (p1) -- (p2) ;
        \draw [polyhedron_edge,dotted] (p4) -- (p1) ;
        \draw [polyhedron_edge,dotted] (p5) -- (p2) ;
        
      %Vertices
      \node [draw, point, label={[label distance=-2pt]45:\normalsize $\bar x$}] at (barx) {};
       
      \node [draw, point] at (v1) {};
      \node [draw, point] at (v4) {};
      %Intersection points
     \node [draw, intersection_point, fill=blue] at (p3) {};
     \node [draw, intersection_point, fill=blue] at (p7) {};
     \node [draw, intersection_point, fill=blue] at (p4) {};
     \node [draw, intersection_point, fill=blue] at (p5) {};
     
     \node [draw, intersection_point, fill=none] at (p1) {};
     \node [draw, intersection_point, fill=none] at (p2) {};
     \node [draw, cross=3pt, red] at (p1) {};
     \node [draw, cross=3pt, red] at (p2) {};
    \end{tikzpicture}
    \caption{Tilted activations of the two hyperplanes not defining $\optcone$ from Figure~\ref{fig:partial-activation}.}
    \label{fig:tilted-activation}
  \end{figure}
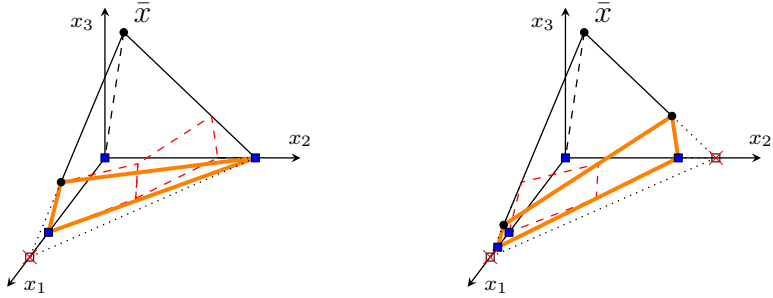

The essential motivation for this section is that we wish to activate a given hyperplane using a {subset} of the rays of $\optcone$ that the hyperplane intersects, while ignoring its intersections with the remaining rays.
However, not picking this subset with care can easily result in invalid cuts.
We will show that a simple rule suffices to guarantee validity: if a ray $r \in \optconerays$ has been previously cut and a hyperplane being activated intersects $r$, then $r$ must be included in the rays that are cut when performing the activation.
Our idea is illustrated in Figure~\ref{fig:tilted-activation} on the same example as in Figure~\ref{fig:partial-activation}.
The panel on the left shows a partial activation of $H_4$, where we choose to intersect the hyperplane with ray $r^1$ of $\optcone$, but not ray $r^2$.
This involves only adding the intersection point obtained from intersecting $H_4$ with $r^1$, while keeping the original intersection point from $r^2 \cap \bd{S}$ intact.
The picture illustrates how this corresponds to a tilting of $H_4$.
In the next panel, we activate hyperplane $H_5$.
Though we are not obliged to, we select to cut ray $r^2$.
However, because $r^1$ was previously cut, our rule states that we \emph{must} cut the ray $r^1$ when activating $H_5$.
Ignoring that ray (i.e., not adding the intersection point on the $x_1$ axis) results in an improper point-ray collection, as shown in Appendix~\ref{app:PHA}.

We now describe the tilting operation formally.
Let $H$ be a hyperplane defining $P$.
We assume that $H$ is not tight at $\lpopt$ for ease of exposition; the degenerate case follows similar reasoning and is handled explicitly in Appendix~\ref{app:tilting}.

The hyperplane $H$ can be uniquely defined by the $n$ affinely independent points or rays obtained by its intersection with $\optcone$.
For $r \in \optconerays$, define $d(H,r) := \min\{\allowbreak\distToHplane{H}{r},\allowbreak-\distToHplane{H}{-r}\}$ as the distance along $r$ (or $-r$) to $H$.
Let $v(H,r)$ denote the intersection of $r$ (or $-r$) with $H$:
  \[
    v(H,r) := 
    		\begin{cases}
      		\lpopt + d(H,r) \cdot r & \text{if $d(H,r) \ne \infty$,} \\
			r & \text{otherwise.} \\
    		\end{cases}
  \]
Then the affine hull of $\{v(H,r) \suchthat r \in \optconerays\}$ is precisely $H$.
The affine independence of these points and rays is a direct result of the affine independence of the rays of $\optcone$.

To tilt $H$, we need only to change some of these points;
there are many ways to tilt, so we refer to our particular method as \emph{targeted tilting}.
As before, let $\raysIntByH(H)$ denote the rays of $\optcone$ that are intersected by $H$ before $\bd{S}$, i.e., those rays $r \in \optconerays$ for which $\distToHplane{H}{r} < \distToHplane{\bd{S}}{r}$.
Let $\rays' \subseteq \raysIntByH(H)$ be an arbitrary subset of these rays that we select to be cut.
Overloading notation again, we will use $v(\bd{S},r)$ to denote the intersection point of ray $r \in \optconerays$ with $\bd{S}$, where $v(\bd{S},r) = r$ if the ray does not intersect $\bd{S}$.
The \emph{targeted tilting} of $H$ (with respect to $\rays'$) is defined as the unique hyperplane $\tilted{H}$ going through 
	$v(\bd{S},r)$ for all $r \in \raysIntByH(H) \setminus \rays'$,
and
	$v(H,r)$ otherwise.

The first thing we observe is that $\tilted{H}^+$ is valid for $P$, as it is a relaxation of $H^+$. %\amk{Obvious...? or does this require proof?}
This immediately implies, by Theorem~\ref{thm:RHA_valid}, that we can activate $\tilted{H}$ using Algorithm~\ref{alg:RHA1}; i.e., given a proper point-ray collection $(\pointset,\rayset)$, the point-ray collection outputted by $\text{\PHAOne}(P,S,\tilted{H},\optconerays,(\pointset,\rayset))$ is proper.
This requires knowing $\tilted{H}$ explicitly.
We next show that computing $\tilted{H}$ is actually unnecessary:
we can calculate exactly the same points and rays via $\text{\PHAOne}(P,S,H,\rays',(\pointset,\rayset))$.

\begin{theorem}
\label{thm:compute-tilted-points}
  Let $H$ be a hyperplane defining $P$ and $\rays^c \subseteq \optconerays$
  the set of rays previously cut,
  $\tilted{H}$ be a hyperplane obtained via a targeted tilting of $H$ on a set of rays $\rays' \subseteq \raysIntByH(H)$ such that $\rays^c \cap \raysIntByH(H) \subseteq \rays'$.
  Given a proper point-ray collection $(\pointset,\rayset)$,
  the point-ray collection returned by
    $\text{\PHAOne}(P, S, H, \rays', (\pointset,\rayset))$
  is proper.
\end{theorem}
\begin{proof}
  We show that
		$\text{\PHAOne}(P, S, H, \rays', (\pointset,\rayset))$.
	produces the same point-ray collection as
	    $\text{\PHAOne}(P, S, \tilted{H}, \optconerays, (\pointset,\rayset))$.
	In particular, we analyze each iteration of step~\ref{step:RHA1:act-hplane} of Algorithm~\ref{alg:RHA1} and show that both algorithms produce the same point-ray collection after each iteration.
	Observe that the only rays of $\optcone$ processed by the algorithm are those that belong to $\rayset_A$ and are cut by $H$ prior to intersecting $\bd{S}$.
	Our targeted tilting method implies that $\raysIntByH(\tilted{H}) = \rays'$, all of which are intersected by both $H$ and $\tilted{H}$ before $\bd{S}$.
	The result is that activation is performed on the same set of rays $\rays'$ for both $H$ and $\tilted{H}$.
	Let us fix a ray $r \in \rays'$ being processed.
	Throughout, the reader may find it useful to refer to Figure~\ref{fig:tilting-theorem}, which illustrates the two-dimensional face of $\optcone$ defined by rays $r \in \rayset'$ and $r' \in \raysIntByH(H) \setminus \rayset'$, the hyperplane $H$ and its tilted version $\tilted{H}$, as well as the result of a prior activation of a hyperplane $H'$ that was not tilted on ray $r'$.

	{
		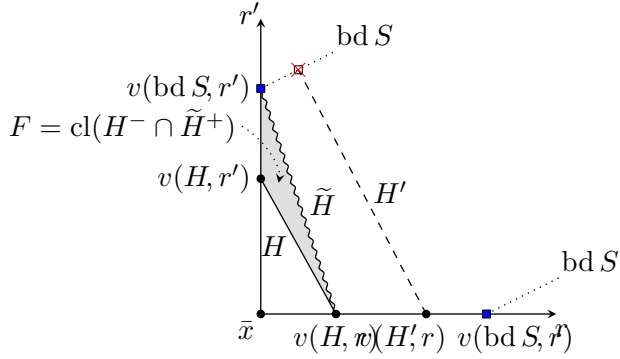
\begin{figure}
			\centering
			\begin{tikzpicture}[line join=round,x={(1 cm, 0 cm)},y={(0 cm, 1 cm)},z={(\zX cm,\zY cm)},>=stealth,scale=2]
				\coordinate (barx) at (0,0);
				\coordinate (rend) at (2,0);
				\coordinate (r'end) at (0,2);
				\coordinate (r) at ($(rend)-(barx)$);
				\coordinate (r') at ($(r'end)-(barx)$);
				\coordinate (pj) at ($(barx)+0.75*(r)$);
				\coordinate (pk) at ($(barx)+0.75*(r')$);
      			\coordinate (vHr) at ($(barx)+0.25*(r)$);
				\coordinate (vHr') at ($(barx)+0.45*(r')$);
			    \coordinate (vH'r) at ($(barx)+0.55*(r)$);
				\coordinate (vH'r') at ($(barx)+(r')$);
				\coordinate (p) at (1/4,13/8);
				
				\fill [gray,opacity=0.25] (vHr) -- (vHr') -- (pk) -- cycle;

				\draw [ray,black] (barx) -- (rend) node [pos=1,fill=none,below=0pt,font=\small] {$r$};
				\draw [ray,black] (barx) -- (r'end) node [pos=1,fill=none,left=-4pt,font=\small] {$r'$};
				\draw [dotted,polyhedron_edge,black] (pj) -- ($(pj)+0.25*(2,1)$) node [pos=1,fill=none,above=-2pt,font=\small] {\rlap{$\bd{S}$}};
				\draw [dotted,polyhedron_edge,black] (pk) -- ($(pk)+0.25*(2,1)$) node [pos=1,fill=none,above=-2pt,font=\small] {\rlap{$\bd{S}$}};
				\draw [polyhedron_edge,black] (vHr) -- (vHr') node [pos=0.5,fill=none,left=0pt,font=\small] {$H$};
				\draw [polyhedron_edge,black,dashed] (vH'r) -- (p) node [pos=0.5,fill=none,right=0pt,font=\small] {$H'$};
				\draw [polyhedron_edge,decorate,decoration={snake,amplitude=.2mm,segment length=1.5mm},black] (vHr) -- (pk) node [pos=0.5,fill=none,right=0pt,font=\small] {$\tilted{H}$};
      
				\node [draw, point, fill=black, label={[label distance=-4pt]-135: \small $\lpopt$}] at (barx) {};
				\node [draw, intersection_point, fill=blue, label={[label distance=-2pt]-90: \small \ \ \ \ \ \ $v(\bd{S}\!,r)$}] at (pj) {};
				\node [draw, intersection_point, fill=blue, label={[label distance=-2pt]180: \llap{\small $v(\bd{S}\!,r')$}}] at (pk) {};
				\node [draw, intersection_point, fill=none, label={[label distance=-5pt]135: \small }] at (p) {};
				\node [draw, cross=3pt, red] at (p) {};
				\node [draw, point, fill=black, label={[label distance=-2pt]-90: \small $v(H'\!\!,r)$\ \ \ \ \ \ }] at (vH'r) {};
				\node [draw, point, fill=black, label={[label distance=-2pt]-90: \small $v(H\!,r)$}] at (vHr) {};
				\node [draw, point, fill=black, label={[label distance=-2pt]180: \llap{\small $v(H\!,r')$}}] at (vHr') {};
				
		\path (-1/8,5/4) edge [->, shorten >= 0.25cm, line width=0.5pt, black, dotted, bend left=35] node [pos=0,left=-2pt] {\llap{\small $F = \cl(H^- \cap \tilted{H}^+)$}} (3/32,.75);
			\end{tikzpicture}
		\caption{Reference illustration for the proof of Theorem~\ref{thm:compute-tilted-points}, activating a tilted version of $H$ (after an earlier activation of a hyperplane $H'$).} %$H$ is represented by the line segment $\ell$, $\tilted{H}$ by $\tilted{\ell}$, and $H'$ by $\ell'$.}
		\label{fig:tilting-theorem}
		\end{figure} % end figure tilting-theorem
	}
	
	We first discuss the set of points and rays that are added by Algorithm~\ref{alg:RHA1} in each of the cases,
	starting at step~\ref{step:RHA1:process-ray}, in which the edges emanating from $v(H,r)$ are considered.
	Note that there is one edge per ray $r' \in \optconerays \setminus \{r\}$, because $\optcone$ is a simplicial cone.

	Consider an edge corresponding to a ray $r'$ whose intersection with $H$ and $\tilted{H}$ remains unchanged ($v(H,r') = v(\tilted{H},r')$), i.e., for $r' \in \rays' \cup (\optconerays \setminus \raysIntByH(H))$.
	This is an edge that exists in both $H \cap \optcone$ and $\tilted{H} \cap \optcone$.
	Since the edge is the same, the resulting intersection point or ray will be the same in either case.
	
	Now suppose that $r' \in \raysIntByH(H) \setminus \rays'$, as in Figure~\ref{fig:tilting-theorem}.
	When activating $H$, the corresponding edge is skipped in step~\ref{step:RHA1:is-ray-edge}, because $r' \in \raysIntByH(H)$ so that $v(H,r') \in \interior{S}$.
	When activating $\tilted{H}$, $v(\tilted{H},r') = v(\bd{S},r')$ (the intersection point of $r'$ with $\bd{S}$) instead of $v(H,r')$.
	Thus, the new intersection point or ray created is precisely $v(\bd{S},r')$.
	However, because $r' \notin \rays'$, it must be the case that $r' \notin \rays^c$, i.e., the ray $r'$ has not been cut,
	which implies that $v(\bd{S},r')$ already belongs to $(\pointset,\rayset)$.
%	Thus, the edge between $v(H,r)$ and $v(H,r')$ of $H \cap \optcone$ is replaced by the edge between $v(H,r)$ and $v(\bd{S},r')$ for $\tilted{H} \cap \optcone$.
	This proves that the same points and rays are added whether $H$ or $\tilted{H}$ is activated.
	
	Lastly, we show that the same points and rays are removed in step~\ref{step:RHA1:remove-viol-points}.
	Consider the two-dimensional face of $\optcone$ defined by $r$ and $r'$.
	As before, when the edge of $H$ and $\tilted{H}$ is the same on this face, exactly the same points and rays of $(\pointset,\rayset)$ that lie on this face are cut when activating either hyperplane.
	Thus, assume that $r' \in \raysIntByH(H) \setminus \rays'$.
	Since $\tilted{H}$ is weaker than $H$, anything cut by $\tilted{H}$ is certainly cut by $H$.
	
	It remains to show the converse, that any point or ray removed by $H$ is also removed by $\tilted{H}$.
	Observe (as depicted in Figure~\ref{fig:tilting-theorem}) that the only part of the face defined by $r$ and $r'$ that is cut by $H$ but not by $\tilted{H}$ is in $F := \conv(v(H,r), v(H,r'), v(\bd{S},r'))$.%
	\footnote{There is one intersection point/ray in $F$: $v(\bd{S},r')$. It is true that, when activating a prior hyperplane $H'$, $v(H',r')$ could coincide with $v(\bd{S},r')$, and in step~\ref{step:RHA1:remove-viol-points}, this would be cut by $H$ but not by $\tilted{H}$. However, this point/ray is a duplicate of the intersection point of $r'$ with $\bd{S}$, which would remain in $(\pointset,\rayset)$.}
	Since $\relint(F) \subseteq \interior{S}$, it contain no points or rays from $(\pointset,\rayset)$, completing the proof.
\end{proof}

One effect of Theorem~\ref{thm:compute-tilted-points} is that it is valid to simply ignore all intersections of hyperplanes with rays in $\initrayset$,
as is done in Algorithm~\ref{alg:PHA}.
The activation performed in step~\ref{step:PHA:RHA-act} of Algorithm~\ref{alg:PHA} is valid because the rays in $\initrayset$ are always ignored, so that the conditions of Theorem~\ref{thm:compute-tilted-points} apply.

We now operationalize the above theorem in a different way.
The motivation for targeted tilting comes not only from Proposition~\ref{prop:parallel-rays}, to avoid cutting the rays in $\initrayset$, but also from our desire to find GICs strictly stronger than the SIC from the same $P_I$-free convex set $S$.
This may require cutting many rays of $\optcone$.
However, using Algorithm~\ref{alg:PHA} for this purpose could involve unnecessarily cutting some rays of $\optcone$ multiple times.
This will happen when we have already cut a certain ray of $\optcone$ but a hyperplane being subsequently activated intersects it again;
not only might this effort be wasted, but also it may create redundant or weak intersection points.
As an alternative, we present the targeted tilting algorithm, Algorithm~\ref{alg:targeted-tilting}, which is tailored to the task of cutting the initial intersection points and rays more efficiently.

\begin{algorithm}[h]
\caption{Generalized Intersection Cuts by {\PHAOne} and Targeted Tilting}\label{alg:targeted-tilting}
\begin{algorithmic}[1]
\Input Polyhedron $P$ defined by a set of hyperplanes $\hplaneset$;
  a vertex $\lpopt$ of $P$;
  indices of fractional integer variables $\sigma$;
  hyperplane selection criterion $\mathcal{SC}$;
  objectives $\mathcal{O}$;
  set $\rayset_A \subseteq \optconerays$ for the rays of $\optcone$ to cut.
\Function{{\PHAOne}CutGeneratorWithTargetedTilting}{$P, \lpopt, \sigma, \mathcal{SC}, \mathcal{O}, \rayset_A$} \label{step:tt:setup}
  \For{$k \in \sigma$}
  	\State
		$S_k \gets \{x \suchthat \floor{\lpopt_k} \le x_k \le \ceil{\lpopt_k}\}$.
    \State
      $\rayset_k \gets \{r \in \optconerays \suchthat \distToHplane{\bd{S_k}}{r} = \infty\}$;
      $\pointset_k \gets \{r \cap \bd{S_k} \suchthat r \in \optconerays \setminus \rayset_k\}$.
  \EndFor
  \State $j \gets 0$; $\rays' \gets \emptyset$.
    \For{the next ray $r \in \rayset_A$} \Comment{Assume sorted by decreasing reduced cost} \label{step:PHA:tt-start}
	 \State $j \gets j+1$; $\rays' \gets \rays' \cup \{r\}$; $\rayset_A \gets \rayset_A \setminus \{r\}$.
      \For{$k \in \sigma$} \label{step:tt:iterate-through-splits}
          \State
            \label{step:tt:init}
            $\hplaneset^j \gets \{H \in \hplaneset \setminus \optconehplanes : \distToHplane{H}{r} < \distToHplane{\bd{S_k}}{r}\}$.
          \If{$\hplaneset^j \ne \emptyset$}
            \State \label{step:tt:choose-hplane}
              Select a hyperplane
      		$H_j \in \hplaneset^j$ according to selection criterion $\mathcal{SC}$.
        \State
          \label{step:tt-act}
          $(\pointset_k, \rayset_k) \gets \text{\PHAOne}(P, S_k, H_j, {\rayset}', (\pointset_k,\rayset_k))$.
        \EndIf  
      \EndFor
      \If{$\log(j) \in \Z$ or $\rayset_A = \emptyset$}
        \State
        \label{step:tt-end}
        Add to $\mathcal{C}$ valid cuts by
        solving \eqref{CutLP}
        with objective types $\mathcal{O}$, always ensuring $\lpopt$ is cut.
      \EndIf
    \EndFor
\EndFunction
\end{algorithmic}
\end{algorithm}

Algorithm~\ref{alg:targeted-tilting} focuses on cutting one ray of $\optcone$ at a time.
The hyperplane chosen for each ray is tilted in step~\ref{step:tt-act} in a way that avoids cutting all other rays, other than the ones that have been previously cut, as required.
However, this may result in cuts that do not get monotonically stronger as more rays of $\optcone$ are cut, as might be expected.
%create non-monotonicity when performing targeted tilting,
%in the sense that cutting more rays of $\optcone$ could actually lead to generating weaker cuts.
This is due to our observation that 
cutting a ray multiple times may lead to the addition of weak intersection points to the collection.
We counteract this in step~\ref{step:tt-end} of Algorithm~\ref{alg:targeted-tilting},
by performing cut generation before all hyperplanes have been chosen and activated.
To reduce computational effort, this step is only performed $\ceil{\log(n)}$ times.

Although targeted tilting may reduce the number of weak intersection points that result from cutting the same rays of $\optcone$ repeatedly, it may also miss improving the point-ray collection in certain directions.
This creates an inherent tradeoff when deciding whether to use {\PHAOne} as part of Algorithm~\ref{alg:PHA} or Algorithm~\ref{alg:targeted-tilting}.
We discuss this as part of our numerical study in Section~\ref{sec:computation}.

%% SECTION 4: COMPUTATIONAL THEORY FOR PHA
\section{Implementation choices for \texorpdfstring{\PHAOne}{PHA1}}
\label{sec:theory}

Algorithms~\ref{alg:PHA} and \ref{alg:targeted-tilting} both involve steps in which a hyperplane is selected and then, after the points and rays have been collected, a set of objective directions is used for generating GICs from the resulting \eqref{CutLP}.
In this section, we discuss the choices we make for these steps and provide theoretical motivation for the decisions.

\subsection{Choosing hyperplanes to activate}
\label{subsec:hplane-choices}

We first state the three criteria we consider for choosing hyperplanes to activate 
as the parameter $\mathcal{SC}$ (used in Algorithm~\ref{alg:PHA} step~\ref{step:PHA:RHA-choose-hplane} and Algorithm~\ref{alg:targeted-tilting} step~\ref{step:tt:choose-hplane}).
For an index $k \in I$, 
we consider optimizing over the $\Sk$-closure, i.e., $\conv(P \setminus \interior{\Sk})$.

\noindent\mbox{}\\
\textbf{(H1) Choose the hyperplane that first intersects some ray $r \in \optconerays \setminus \initrayset$ before $\bd{\Sk}$.}\ \ 
This corresponds to pivoting to the nearest neighbor from $\lpopt$ along the ray $r$,
which would be the same hyperplane selected in the procedure from \cite{BalMar13}.

\noindent\mbox{}\\
\textbf{(H2) Choose the hyperplane that yields a set of points with highest average depth.}\ \ 
Here, \textit{depth} is calculated as the Euclidean distance to the SIC (using the same split set).
(H2) seeks a set of points that are far from the cut we are trying to improve upon.
The idea is that the resulting cuts will then be deeper as well.

\noindent\mbox{}\\
\textbf{(H3) Choose the hyperplane creating the most {final} intersection points.}\ \ 
A \emph{final} intersection point is defined as follows.

\begin{definition} \label{defn:finalpoint}
  Suppose $(\pointset\splitindex{k}, \rayset\splitindex{k})$ is a proper point-ray collection.
  An intersection point in $\pointset\splitindex{k}$ or a ray in $\rayset\splitindex{k}$ is \emph{final}
  (with respect to $S\splitindex{k}$)
  if it belongs to $P$ and $\bd\Sk$, meaning it cannot be cut away by any valid hyperplane activations.
\end{definition}

We denote the set of final intersection points by $\pointset\final\splitindex{k}$
  and final rays by $\rayset\final\splitindex{k}$.

\begin{proposition} \label{prop:Skclosure}
  Suppose $C = P$.
  Then, the intersection points in $\pointset\final\splitindex{k}$ define vertices of the $\Sk$-closure 
  and the rays $\rayset\final\splitindex{k}$ define extreme rays of $\Sk$-closure.
\end{proposition}
\begin{proof}
  The $\Sk$-closure is defined as
    $ 
      \Pk = \conv(P \setminus \interior \Sk).
    $ 
  Therefore, the points $p \in \pointset\final$ are in the $\Sk$-closure 
  since they belong to $P \cap \bd{\Sk}$.
  Furthermore, the rays in $\rayset\final$ are in the $\Sk$-closure since they 
  are extreme rays of $P$ that do not intersect $\bd{\Sk}$.

  Consider an arbitrary side of the split disjunction, $\Sk\onfloor := \{x: x_k = \floor{\lpopt_k}\}$.
  Suppose that a vertex $p = r \cap \Sk\onfloor$ in $\pointset\splitindex{k}$ is not a vertex of the $\Sk$-closure. 
  Then it can be written as a convex combination of other vertices of $P \cap \Sk\onfloor$ 
  that are also intersection points created from edges of $P$. 
  This implies that $r$, an edge of $P$, can be written as a conic combination of edges in $P$ 
  that intersect $\Sk\onfloor$, which is a contradiction. 
  Showing that rays in $\rayset\splitindex{k}$ are extreme rays of $\Sk$-closure follows similar reasoning. %\qed
\end{proof}

Thus, (H3) targets a set of points that lead to facet-defining inequalities for the set $\conv(P \setminus \interior S)$.
It turns out to be useful to distinguish between intersection points that are in $P$ and those that are not.

\subsection{Choosing objective functions}
\label{subsec:obj-choices}

Even if we know that strong cuts exist from a given point-ray collection,
it remains to generate these GICs using \eqref{CutLP}.
To do this, we need to appropriately choose objective function coefficient vectors.
We consider the following objective directions:
\begin{enumerate}[labelindent=0pt,leftmargin=!] %\widthof{(B)}%\bfseries
  \item[($R$)] Ray directions of $\optcone$ (i.e., the initial intersection points)
  \item[($V$)] Vertices that are generated from hyperplane activations
  \item[($T$)] Intersection points obtained for $\Sk$ (the \emph{tight point} heuristic)
  \item[($S$)] Intersection points from other splits
%  \item[($B$)] Objectives chosen via bilinear program solved iteratively
\end{enumerate}

The first set of objectives, $(R)$, is incentivized by the observation that obtaining GICs that strictly dominate the SIC from the same $P_I$-free convex set requires cutting many of the initial intersection points.
This observation is made concrete in the Theorem~\ref{thm:strictDom} in Appendix~\ref{app:theory};
we show that finding a strictly dominating cut requires reducing the dimension of the convex hull of the initial intersection points and rays.
Another objective typically used for a cut-generating linear program is to minimize $\lpopt^\T \alpha$,
to obtain the most violated cut with respect to $\lpopt$.
We generalize this with $(V)$, by optimizing to find the most violated cuts with respect to each of the vertices created on $\optcone$ by hyperplane activations.

The third set of directions, $(T)$, is referred to as the tight point heuristic as we are simply trying to find a cut tight on each of the rows of \eqref{CutLP} corresponding to points.
In other words, we minimize $\alpha^\T p$
for every $p \in \pointset$ from the current split.
Obviously, we will not be able to cut away $p$;
the purpose of this objective function is to place a cut as close to the chosen intersection point as possible.
The fourth set, $(S)$, capitalizes on the fact that multiple split sets are considered simultaneously in practice,
so that we can share information across the sets to obtain stronger cuts.

The sets of objective directions $(T)$ and $(S)$ switch the perspective typically used for linear programs that generate cuts.
Instead of finding a cut with maximal violation, the cuts we obtain from these objectives aim to approximate the convex hull of intersection points and rays
from every direction, thereby obtaining as many of the facets of that convex hull as possible and a set of cuts with a wider diversity of angles, which is a desirable computational quality.

Next, we discuss the theoretical motivation for $(T)$.
Consider optimizing over the $\Sk$-closure: $\min\{c^\T x \suchthat x \in \conv(P \setminus \interior{\Sk})\}$. 
We show that intersection points can be used to find an optimal solution to this problem without using \eqref{CutLP},
and we address how the heuristic $(T)$ aims to obtain the bound implied by this optimal solution.

Theorem~\ref{thm:optOverSkWithP} shows what bounds can be computed on the optimal value over the $\Sk$-closure
using the point-ray collection when $P \subseteq C$, i.e., not all hyperplanes may have been activated.
We define 
  $\overline{p}^k \in \argmin_{p \in \pointset\final\splitindex{k}} c^\T p$
(when $\pointset\final\splitindex{k}$ is nonempty)
and 
  $\underline{p}^k \in \argmin_{p \in \pointset\splitindex{k}} c^\T p$.
Let $\overline{z} = c^\T \overline{p}^k$ 
(defined to be $+\infty$ when $\pointset\final\splitindex{k}$ is empty)
and $\underline{z} = c^\T \underline{p}^k$.

\begin{theorem}\label{thm:optOverSkWithP}
  Suppose $\pointset\final\splitindex{k}$ is nonempty. 
  If $\underline{z} < \overline{z}$, 
  then $\underline{z}$ is a lower bound and $\overline{z}$ is an upper bound 
  on the minimum over the $\Sk$-closure. 
  Otherwise, $\overline{z} \le \underline{z}$ 
  and $\overline{z}$ is the minimum over the $\Sk$-closure.
\end{theorem}
\begin{proof}
  By Proposition~\ref{prop:Skclosure}, $\overline{p}^k$ is in the $\Sk$-closure. 
  Therefore, $\overline{z}$ always provides an upper bound on $c^\T p^*$, 
  where $p^*$ is a minimum over the $\Sk$-closure.

  When $\underline{z} < \overline{z}$, $\underline{z} = \min\{c^\T x \suchthat x \in C \cap \bd{\Sk}\}$ 
  provides a valid lower bound on $c^\T p^*$, since $P \cap \bd{\Sk} \subseteq C \cap \bd{\Sk}$. 
  On the other hand, when $\overline{z} \le \underline{z}$, 
  we have $\overline{z} = \min\{c^\T x \suchthat x \in C \cap \bd{\Sk}\}$, 
  which implies that $\overline{z}$ provides a lower bound on $c^\T p^*$.
  Therefore, it must be a minimum over the $\Sk$-closure. %\qed
\end{proof}

Corollary~\ref{cor:optOverSkWithP} follows immediately for the special case when $C = P$.

\begin{corollary}\label{cor:optOverSkWithP}
  If $C = P$, then
  $\overline{p}^k$
  is an optimal solution over the $\Sk$-closure.
\end{corollary}
% \begin{proof}
%   Note that we are assuming \eqref{LP} is bounded, which implies that the optimization problem
%   over $P \setminus \Sk$ is bounded.
%   Assume for sake of contradiction that $p^*$ is not optimal.
%   Then there exists an $x'$ in the set $P \setminus \Sk$
% %    \[
% %      \left\{
% %        \left( P \cap x_j < \floor{\lpopt_j} \right) 
% %        \lor 
% %        \left( P \cap x_j > \ceil{\lpopt_j} \right)
% %      \right\}
% %    \]
%   such that 
%     $c^\T \lpopt \le c^\T x' < c^\T p^*,$ 
%   where $\lpopt$ is the optimal solution to \eqref{LP}. 
%   By convexity of $P$, there exists a point on the line between $\lpopt$ and $x'$
%   that is on $\bd \Sk$ and has a strictly lower cost than $p^*$. 
%   Since this point is in the $\Sk$-closure, the region $\{x \in \bd \Sk \suchthat c^\T x < c^\T p^*\} \ne \emptyset$,
%   which is a contradiction, as there will be an intersection point in that region.
%   Hence $p^*$ is optimal.
% \end{proof}

Thus, $\underline{p}^k$ is readily available from $\pointset\splitindex{k}$ and implies a lower bound on the value
of the optimal solution over the $\Sk$-closure.
The same bound $\underline{z}$ implied by $\underline{p}^k$ can be obtained through GICs, 
but it may require many cuts generated from \eqref{CutLP}.
We may be able to obtain one inequality tight at $\underline{p}^k$ by using $\underline{p}^k$ 
itself as the objective coefficient vector, since $\alpha^\T \underline{p}^k \ge \beta$ for all $(\alpha,\beta)$
feasible to \eqref{CutRegion}, and there will exist a facet-defining inequality 
	$\bar{\alpha}^\T x \ge \bar{\beta}$ 
for $\cutpolyhedron$ satisfying 
	$\bar{\alpha}^\T \underline{p}^k = \bar{\beta}$.
Note that for validity of this inequality to be guaranteed, 
we must additionally verify that the inequality cuts a point $v \in \mathcal{K}'$.
However, even if this is satisfied, it is unlikely that the one cut will imply the bound
$c^\T x \ge \underline{z}$ on the objective value.
In the worst case, $n$ facets of $\cutpolyhedron$ that are tight at $\underline{p}^k$ 
may be required to obtain this bound via cuts.
To get these other facets tight at $\underline{p}^k$, 
we can use points in $\pointset\splitindex{k}$ that lie close to $\underline{p}^k$ as the objective vectors.
This is precisely what we do in approach $(T)$,
though we do not only use $\underline{p}^k$ and points in its vicinity,
but also other points from $\pointset\splitindex{k}$ to encourage diversity of the cut collection.

Finally, we give more details and motivation for the objective directions $(S)$.
When using a single split set $\Sk$,
no intersection point generated from that split can be cut by any inequality
generated through \eqref{CutLP} from the point-ray collection for $\Sk$.
However, different splits (more generally, different $P_I$-free convex sets)
give rise to different point-ray collections,
and the buildup of intersection points and rays
can and should be done in parallel for several splits.
The first reason for this is computational efficiency:\ 
one can intersect an edge with the boundaries of more than one split
set and store these intersection points or rays separately.
The second reason is that,
when using multiple split disjunctions, an intersection point on the boundary of one split set
may lie in the interior of another
and hence can be cut away by a facet of the point-ray collection from this second split.
With this in mind,
let $\sigma$ be the indices of a set of integer variables that are fractional at $\lpopt$,
i.e., $\sigma \subseteq \{j \in I\suchthat \lpopt_j \notin \Z\}$.
For every intersection point $p$ generated from intersecting an edge of $C$ with the boundary
of some split disjunction, we follow the edge to find the \emph{last} split disjunction $\Sk$, 
$k \in \sigma$, that this edge intersects.
If $p$ lies in the interior of $\Sk$, then we will use $p$ as an objective for \eqref{CutLP}
with feasible region determined by the point-ray collection from $\Sk$.

%% SECTION 5: COMPUTATIONAL RESULTS
\section{Computational results}
\label{sec:computation}

This section contains the results from computational experiments with {\PHAOne} as used in Algorithm~\ref{alg:PHA} from Section~\ref{sec:PHA} and Algorithm~\ref{alg:targeted-tilting} from Section~\ref{sec:tilting}.
The purpose of these experiments is exploratory;
that is, we seek conditions and implementation choices that lead to strong GICs.
To do this, we
evaluate the effect of a variety of parameters, including those discussed in Section~\ref{sec:theory},
and identify structural properties of instances that can be taken advantage of to find stronger cuts.
In particular, our results indicate that it is beneficial to use objective functions targeting GICs that are tight on the points and rays in the point-ray collection.
Another observation from our experiments is that \PHAOne{} leads to few final points (as defined in Section~\ref{subsec:hplane-choices}), which, in combination with the intuitive importance of these points, indicates a limitation of our procedure that can perhaps be a target for improvement via future research.

% the objective functions $(T)$ discussed in Section~\ref{subsec:obj-choices} tend to have the greatest impact on the strength of the resulting cuts, indicating that a 
%In particular, our results indicate that it can be beneficial to seek a multitude of intersection points that are deep or final as discussed in Section~\ref{subsec:hplane-choices}, and to subsequently target GICs tight on these points through objective functions.

We test Algorithm~\ref{alg:PHA} and Algorithm~\ref{alg:targeted-tilting} when used independently, as well as in combination.
The purpose of this is to test the strength of targeted tilting relative to non-tilted hyperplanes as in Algorithm~\ref{alg:PHA}.
When used together, we first perform Algorithm~\ref{alg:targeted-tilting} and afterwards perform non-tilted activations.
This is equivalent to inserting the steps of Algorithm~\ref{alg:targeted-tilting} before step~\ref{step:PHA:RHA-start} of Algorithm~\ref{alg:PHA}.

\subsection{Experimental setup}

\textbf{Parameters.}\ \ 
Several parameters are fixed throughout the experiments, while others are varied, summarized in Table~\ref{tab:params} and elaborated on below.

\begin{table}[htp]
%  \small
    \centering
    \caption{Parameters that are varied and the values considered.}
    \label{tab:params}
    \begin{tabular}{ll}
        Parameter & Values Considered  \\
        \hline
        % Cut-generating sets & All simple splits \\
        % Overall time limit & 3600 seconds \\
        % Cut solver time limit & 5 seconds \\
        % Presolve & Off \\
        % Max number cuts & 1,000 \\
        % Zero $\epsilon$ & $10^{-7}$ \\
        % Fractionality $\epsilon$ & $10^{-3}$ \\
        % Max dynamism & $10^6$ \\
        Hyperplane scoring function, $\mathcal{SC}$ & \{H1, H2, H3\} \\
        Objective functions, $\mathcal{O}$ & $\{(R),(V),(T),(S)\}$  \\
        Targeted tilting & \{On, Off\} \\
        \# non-tilted hyperplanes, $k_h$ & $\{0,1,2,3,4\}$
        % \# rays cut & $\{4,\text{all}\}$ \\
        % Limit cuts per split & \{On, Off\} \\
    \end{tabular}
\end{table}

The fixed parameters are as follows.
The sets used for cut generation are all the simple splits.
The time limits are one hour per set of parameters
and at most five seconds per objective function for each time
\eqref{CutLP} is resolved.
The instances are not preprocessed,
and presolve is turned off for \eqref{CutLP},
as using presolve mildly reduced the average gap closed and led to erratic solver behavior, such as occasional crashes.
At most 1,000 GICs are generated per instance
(a limit seldom attained).
As the cardinality of $\pointset$ may be large, using all the intersection points as objective directions for \eqref{CutLP} may be prohibitively expensive; as a result, we limit the number of points used for objective functions to 1,000 per instance.
The overall zero tolerance is set to $10^{-7}$,
while the fractionality tolerance is set to $10^{-3}$; i.e.,
a component is considered fractional if it is at least $10^{-3}$ from the nearest integer.
%Every time a solution to \eqref{CutLP} is obtained with poor numerical properties, it is rejected.
%
%We also perform some standard checks to avoid numerically unstable cuts.
To avoid numerical instability, cuts with dynamism higher than $10^6$ are rejected,
where dynamism is the ratio of the largest and smallest cut coefficients.
%Additionally, for each candidate cut that is generated, its orthogonality with every existing cut is checked.
%Given two cuts $\alpha_1^\T x \ge \beta_1$ and $\alpha_2^\T x \ge \beta_2$, their orthogonality is computed as 
%  $1 - ({\alpha_1 \cdot \alpha_2})/({ \norm{\alpha_1}_2 \norm{\alpha_2}_2})$.
%If the cut already in the pool that is most parallel to the candidate cut has orthogonality less than 0.01 and one of these two cuts is sparser or has a larger Euclidean distance from the LP optimal solution, then only that cut is kept.
A one hour time limit is used per instance.

The parameters we modify are used to evaluate the effect of the choice of
hyperplanes to be activated and objective functions to be used for~\eqref{CutLP}.
In Algorithm~\ref{alg:PHA}, we vary $k_h$ from $0$ to $4$, where $k_h = 0$ implies we do not perform any non-tilted activations.
Section~\ref{sec:params:hplanes} contains computational results relating to the hyperplane activation rules described in Section~\ref{subsec:hplane-choices}.
Section~\ref{sec:params:cut-heurs} presents the results for the objective functions discussed in Section~\ref{subsec:obj-choices}.

\noindent\mbox{}\\
\textbf{Cut generation.}\ \ 
The indices $\sigma$ used for defining the split sets for cut generation are those of the fractional integer variables at $\lpopt$.
First, we generate one SIC for each $\Sk$, $k \in \sigma$.
We then generate GICs using Algorithms \ref{alg:PHA} and \ref{alg:targeted-tilting}, as described above.

We generate cuts in the nonbasic space. 
The vertex $\lpopt$ of $\bar{C}$ in this space corresponds to the zero vector, and the $j$th ray of $\optcone$ in the nonbasic space has a single nonzero $j$th entry. 
As a result, the intersection points or rays defining SICs (one pivot from $\lpopt$) and GICs from PHA on $\optcone$ (two pivots from $\lpopt$) have one and two nonzero coordinates, respectively. 
This feature leads to a sparse constraint coefficient matrix of \eqref{CutLP}, which
improves the time complexity of implementing Algorithm~\ref{alg:RHA1}.
In addition, it implies that it is sufficient to consider cuts with $\beta = 1$ to obtain all valid inequalities that cut $\lpopt$.

The objective functions tested are those discussed in Section~\ref{subsec:obj-choices}.
Assuming all objective functions are utilized, the order in which these are used in the code is $(S)$, $(T)$, $(R)$, then $(V)$.
%The effect of the objective function choices is explored in detail in Section~\ref{sec:params:cut-heurs}.

As our procedure is not recursive, all the generated cuts have split rank 1 with respect to $P$, and in particular, they all valid for the first split closure.
Neither the SICs nor the GICs are strengthened through the integral nonbasic variables, as conventional strengthening approaches require information not available through \eqref{CutLP}.

\noindent\mbox{}\\
\textbf{Environment.}\ \ 
All algorithms are implemented in \Cpp\ in the COIN-OR framework~\cite{COIN-OR}, 
using \texttt{Clp\,version\,1\!.\!16} as the linear programming solver.
The machine used is a shared 64-bit \texttt{PowerEdge\,R515} with 48GB
of memory and twelve \texttt{AMD\,Opteron\,4176} processors clocked at 2.4GHz.
%of which we use only one per instance. 
The operating system is \texttt{Red Hat Enterprise\,Linux\,7\!.\!6}
and the compiler is \texttt{g++\,version\,4\!.\!8\!.\!5\,20150623\,(Red Hat 4\!.\!8\!.\!5\!-\!36)}.
Experiments with instance \texttt{mod008} and the second round of cuts were run on a separate shared \texttt{SuperServer\,6019P-WTR} machine with 512GB of memory running \texttt{Orace\,Linux\,7\!.5} on 32 \texttt{Intel\,Xeon\,Gold\,6142} 2.6GHz processors.

\noindent\mbox{}\\
\textbf{Cut evaluation.}\ \ 
The metric we use to evaluate the cuts obtained from a specific set of parameters is \emph{percent gap closed}.. 
The gap is defined as the difference between the optimal values of the integer program
and its linear programming relaxation.
Denoting the optimal value of the integer program by $\OPTIP$,
of its linear relaxation by $\OPTLP$, and of the linear relaxation with cuts added by $\OPTCuts$,
we have
  \[
    \GAP := 100 \times ({\OPTCuts - \OPTLP})/({\OPTIP - \OPTLP}).
  \]
The baseline we use is the percent
gap closed by SICs, which are the simplest GICs.
We then add GICs along with the SICs
to assess what additional effect GICs
have on the percent gap closed
in the presence of the SICs.

\noindent\mbox{}\\
\textbf{Instance selection.}\ \ 
We test forty instances selected from {MIPLIB}~\cite{MIPLIB,MIPLIB3,MIPLIB2003,MIPLIB2010} (all versions) 
based on the following criteria meant to identify small problems
so that many different parameter settings can be tested (over 200 in our experiments):
		(1)~The number of rows and number of columns must be no more than 500 each.
		(2)~The instance has to be integer-feasible 
        with $\OPTLP < \OPTIP$,
        and the gap closed by SICs is not 100\%.
		(3)~There must be at least one non-final intersection point 
        created from intersecting the rays of $\optcone$ with $\bd{S}$.
        (4)~The instance must not be known to have 0\% gap closed from split cuts
        based on previous experiments~\cite{BalSax08,DasGunLod10}.
Criterion~3 exists because we do not cut rays of $\optcone$ that do not intersect $\bd S$,
so if all intersection points are final, no hyperplanes will be activated.

We modify the \texttt{stein15}, \texttt{stein27}, and \texttt{stein45} instances 
to reduce symmetry, by replacing the objective $\sum_{j=1}^n x_j$ with $\sum_{j=1}^n j x_j$.
We also remove the cardinality constraint $\sum_{j=1}^n x_j \ge (n-1)/2$,
as this is not present in the initial formulation of these instances~\cite{FulNemTro74}.
In addition, we exclude instances \texttt{go19} and \texttt{pp08aCUTS}: the former because its continuous relaxation solves exceptionally slowly, and the latter because it is simply a strengthened version of \texttt{pp08a}.

\subsection{Point-ray collection statistics}
\label{subsec:pr-stats}

The following statistics are averaged across all the instances and all the splits used for each instance.
First, of the rays of the initial cone $\optcone$, 55\% belong to $\initrayset$, i.e., do not intersect $\bd{S}$.
This %shows the high prevalence of rays of $\optcone$ that do not intersect $\bd{S}$, which 
demonstrates the impact of Proposition~\ref{prop:parallel-rays} and our resulting decision to always set the parameter $\rayset_A$ of Algorithm~\ref{alg:RHA1} to 
a subset of the rays of $\optcone$ that intersect $\bd{S}$.
An additional 5\% of the initial rays on average lead to final intersection points (with a range of 0\% to 34\%),
leaving possibly few rays that can be cut for some instances.
For six instances, less than 10\% of the rays are able to be cut.

We can also provide an idea of the number of rows in \eqref{CutLP} in practice, i.e., the number of points and rays we generate, for the best combination of parameters for each instance.
On average across all splits, each \eqref{CutLP} contains about 2,000 points 200 rays.
There exist instances with split sets leading to as few as 2 and as many as 46,000 points,
and as few as 0 and as many as 4,000 rays.
On average, fewer than 10\% of the points and rays in the collection are final (with a range of 0.1\% to 75\%).
Figure~\ref{fig:pointsvsrays} plots the number of generated points across instances (for the best parameter combination for each instance) as a function of the number of rays of $\optcone$ that can be cut;
we observe a quadratic relationship, as predicted by our analysis in Section~\ref{sec:PHA}.
%
%From the point-ray collection at the end of our procedure, 
%If we look at points only (and not rays), then on average 14\% are final (with a range of 0.1\% to 88\%).

\begin{figure}[htp]
    \centering
    \includegraphics[scale=.4]{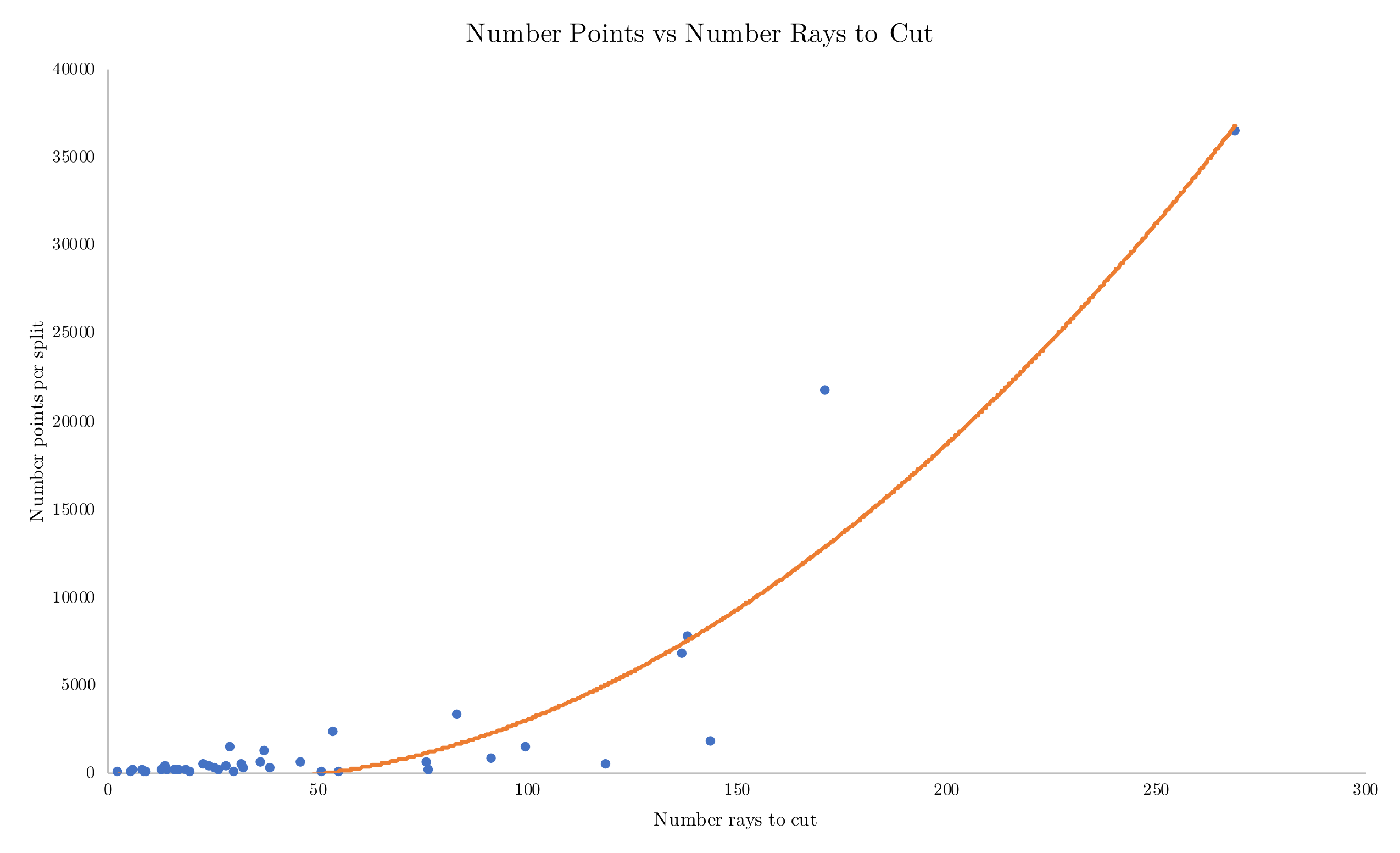}
    \caption{Number generated points versus number of rays that can be cut.}
    \label{fig:pointsvsrays}
\end{figure}

\subsection{Effect of hyperplanes activated}
\label{sec:params:hplanes}

We now analyze how the choice of hyperplanes to activate, as described in Section~\ref{subsec:hplane-choices}, affects the percent gap closed by GICs.
%We also attempt to isolate the effect of
%targeted tilting by running {\PHAOne} with and without targeted tilting enabled.
%with and without employing Algorithm~\ref{alg:targeted-tilting} to select hyperplanes.
%In addition, we consider whether targeted tilting leads to better cuts.
In addition, in Section~\ref{sec:tilting}, we discussed that, theoretically, neither Algorithm~\ref{alg:PHA} nor targeted tilting is strictly stronger than the other.
%
% the relative advantages and disadvantages of {\PHAOne} versus targeted tilting.
%In theory, the targeted tilting approach may sometimes lead to strictly stronger cuts, and sometimes it could lead to strictly weaker cuts.
%This is because targeted tilting allows more rays of $\optcone$ to be cut and reduces the number of times any individual ray is cut, but tilting the hyperplanes makes them weaker and hence could sometimes lead to weaker point-ray collections.
As a result, the computational experiments we perform are necessary to assess the practical strength of each approach.

%
% The two procedures that are tested are step~\ref{step:RHA1:choose-hplanes} of
% Algorithm~\ref{alg:RHA1} in which targeted tilting (Algorithm~\ref{alg:targeted-tilting}) is used in 
%
% Clearly (H2) and (H3) should take more time, since they involve performing a fake
% activation step to calculate the set of points that will be created.
% This is true, yet, since this mainly affects the hyperplane selection phase of the algorithm,
% ultimately (H2) is only $13\%$ slower than (H1), 
% and (H3) is only $6\%$ slower than (H1).
% HH1: 0.025, HH2: 2.11, HH3: 1.09

% In PHA, we may activate more than one hyperplane on each ray. 
% %We do this for a subset of all rays of $\bar{C}$. {\color{red}{(Selva: Can we make this more specific with respect to a set?)}}
% %The resulting point-ray collection will still be proper, but it may contain redundant points.
% %Thus, we also consider the possibility of activating additional hyperplanes,
% Once hyperplane activations to intersect specific rays are complete, we perform a second type of activation. 
% Here, for each split, a hyperplane is chosen for activation 
% that maximizes the number of additional final intersection points created.
% This hyperplane is activated without tilting (i.e., {\PHAOne} only) 
% on all rays of $\optcone$ intersecting the hyperplane before $\bd{S}$.
% Average depth is used to break ties.
%If no more final intersection points are possible, then average depth is used as the criterion.

\begin{singlespace}
  \begin{adjustbox}{
	max totalheight =\textwidth,
%    max totalheight = 0.95\textheight,
%    max width = \textwidth,
	rotate=90,
    center,
    captionabove={Best percent gap closed by hyperplane activation choice for the 30 instances with any improvement over SICs across all settings. 
    		We highlight the cells that achieve the best result, but only if it is the first time this result is attained using that particular combination of Algorithms~\ref{alg:PHA} and \ref{alg:targeted-tilting}.
    		Some values differ in the thousandths digit.},
    label={tab:hplane-analysis},
    float=table
}
    \begin{tabular}{lccH|*{3}{c}|*{5}{c}|*{4}{c}}
    \toprule
      {} & {} & {} & {} & \multicolumn{3}{c|}{Only Alg. 3} & \multicolumn{5}{c|}{Alg. 3 with Alg. 2} & \multicolumn{4}{c}{Only Alg. 2}\\%\cmidrule(r){5-7}\cmidrule(r){8-12}\cmidrule(r){13-16}
      {Instance} & {SIC} & {Best} & {Diff} & {H1} & {H2} & {H3} & {T+0} & {T+1} & {T+2} & {T+3} & {T+4} & {+1} & {+2} & {+3} & {+4}\\
      \midrule
	  {bell3a} & {44.74} & {59.52} & {14.77} & {58.94} & {58.49} & {58.94} & {58.94} & {58.94} & {58.94} & {58.94} & {58.94} & {59.07} & \goodcell{59.52} & {59.52} & {59.52}\\
      {bell3b} & {44.57} & {60.34} & {15.76} & {51.39} & {53.45} & {51.39} & {53.45} & {53.45} & {53.45} & {53.45} & {53.45} & {60.29} & {60.30} & \goodcell{60.34} & {60.34}\\
      {bell4} & {23.37} & {26.54} & {3.17} & {23.82} & {26.35} & {23.83} & {26.36} & {26.35} & \goodcell{26.54} & {26.54} & {26.54} & {24.67} & {24.67} & {25.26} & {25.26}\\
      {bell5} & {14.53} & {85.37} & {70.85} & {17.58} & {19.99} & {17.58} & {19.99} & {19.99} & {19.99} & {19.99} & {19.99} & \goodcell{85.37} & {85.37} & {85.37} & {85.37}\\
      {blend2} & {16.04} & {19.78} & {3.75} & {19.78} & {17.23} & {19.78} & {19.78} & \goodcell{19.78} & {19.78} & {19.78} & {19.78} & {16.33} & {17.36} & {17.36} & {17.36}\\
      {bm23} & {5.92} & {11.06} & {5.14} & {8.79} & {9.78} & {8.92} & {9.78} & {9.78} & {9.78} & {9.78} & {9.78} & {8.00} & {10.65} & \goodcell{11.06} & {11.06}\\
      {egout} & {51.57} & {52.11} & {0.54} & {51.59} & {51.60} & {51.59} & {51.60} & {51.60} & {51.60} & {51.60} & {51.60} & \goodcell{52.11} & {52.11} & {52.11} & {52.11}\\
      {gt2} & {83.13} & {84.26} & {1.13} & {83.13} & {83.13} & {83.13} & {83.13} & {83.13} & {83.13} & {83.13} & {83.13} & \goodcell{84.26} & {84.26} & {84.26} & {84.26}\\
      {k16x240} & {7.56} & {7.71} & {0.15} & {7.56} & {7.65} & {7.56} & {7.65} & {7.65} & {7.65} & {7.65} & {7.65} & \goodcell{7.71} & {7.71} & {7.71} & {7.71}\\
      {lseu} & {4.57} & {4.65} & {0.08} & \goodcell{4.65} & {4.57} & \goodcell{4.65} & \goodcell{4.65} & {4.65} & {4.65} & {4.65} & {4.65} & {4.57} & {4.65} & {4.65} & {4.65}\\
      {mas74} & {3.30} & {4.31} & {1.01} & {4.12} & {3.90} & {4.06} & {4.12} & {4.12} & {4.12} & {4.12} & {4.12} & {4.30} & \goodcell{4.31} & {4.31} & {4.31}\\
      {mas76} & {2.37} & {2.49} & {0.13} & {2.37} & {2.37} & {2.37} & {2.37} & {2.37} & {2.37} & {2.37} & {2.37} & {2.38} & \goodcell{2.49} & {2.49} & {2.49}\\
      {mas284} & {0.38} & {0.51} & {0.13} & {0.45} & \goodcell{0.51} & {0.45} & \goodcell{0.51} & {0.51} & {0.51} & {0.51} & {0.51} & {0.39} & {0.39} & {0.39} & {0.39}\\
      {misc05} & {3.60} & {3.62} & {0.02} & {3.60} & {3.60} & {3.60} & {3.60} & {3.60} & {3.60} & {3.60} & {3.60} & {3.60} & {3.60} & {3.60} & \goodcell{3.62}\\
      {mod008} & {1.30} & {1.37} & {0.07} & {1.31} & {1.30} & {1.31} & {1.31} & {1.31} & {1.31} & {1.31} & {1.31} & {1.31} & \goodcell{1.37} & {1.37} & {1.37}\\
      {mod013} & {4.41} & {7.37} & {2.96} & {4.72} & {7.28} & {4.72} & {7.28} & \goodcell{7.37} & {7.37} & {7.37} & {7.37} & {4.41} & {4.72} & {4.72} & {7.28}\\
      {modglob} & {9.59} & {14.02} & {4.43} & {10.00} & {10.26} & {10.00} & {10.26} & {10.26} & {10.26} & {10.26} & {10.26} & {10.00} & {14.02} & \goodcell{14.02} & {14.02}\\
      {p0033} & {1.83} & {5.19} & {3.35} & {2.59} & \goodcell{5.19} & {2.59} & \goodcell{5.19} & {5.19} & {5.19} & {5.19} & {5.19} & {1.83} & {2.59} & {2.59} & {2.59}\\
      {p0282} & {3.67} & {5.12} & {1.45} & {4.09} & {4.50} & {4.09} & {4.50} & {4.50} & {4.77} & {4.77} & {4.77} & {4.65} & {4.81} & {4.81} & \goodcell{5.12}\\
      {p0291} & {27.78} & {40.12} & {12.34} & {27.93} & {34.21} & {27.93} & {34.21} & {34.21} & {34.21} & {34.21} & {34.21} & {39.80} & \goodcell{40.12} & {40.12} & {40.12}\\
      {pipex} & {0.81} & {1.43} & {0.62} & \goodcell{1.43} & \goodcell{1.43} & \goodcell{1.43} & \goodcell{1.43} & {1.43} & {1.43} & {1.43} & {1.43} & {1.43} & \goodcell{1.43} & {1.43} & {1.43}\\
      {pp08a} & {51.44} & {54.46} & {3.02} & {53.42} & {53.89} & {53.42} & {53.89} & {53.89} & {53.89} & {53.89} & {53.89} & {52.47} & {54.41} & \goodcell{54.46} & {54.46}\\
      {probportfolio} & {25.14} & {25.28} & {0.15} & {25.19} & \goodcell{25.28} & {25.19} & \goodcell{25.28} & {25.28} & {25.28} & {25.28} & {25.28} & {25.14} & {25.14} & {25.14} & {25.16}\\
      {sample2} & {5.86} & {13.14} & {7.28} & {5.86} & \goodcell{13.14} & {5.86} & \goodcell{13.14} & {13.14} & {13.14} & {13.14} & {13.14} & {5.86} & {5.86} & {5.86} & {5.86}\\
      {sentoy} & {10.38} & {14.00} & {3.62} & {12.19} & {12.45} & {12.39} & {12.45} & {12.45} & {12.45} & {12.45} & {12.45} & {10.38} & {13.89} & {13.89} & \goodcell{14.00}\\
      {stein15\_nosym} & {50.00} & {58.33} & {8.33} & {55.78} & {50.00} & {55.78} & {55.78} & {57.17} & {57.17} & {57.17} & {57.17} & {50.00} & {50.00} & \goodcell{58.33} & {58.33}\\
      {stein27\_nosym} & {7.41} & {8.78} & {1.37} & \goodcell{8.78} & {8.61} & \goodcell{8.78} & \goodcell{8.78} & {8.78} & {8.78} & {8.78} & {8.78} & {7.41} & {7.86} & {7.86} & {7.86}\\
      {stein45\_nosym} & {7.10} & {7.55} & {0.45} & {7.36} & \goodcell{7.55} & {7.36} & \goodcell{7.55} & {7.55} & {7.55} & {7.55} & {7.55} & {7.51} & {7.51} & {7.51} & {7.51}\\
      {vpm1} & {10.00} & {10.18} & {0.18} & {10.00} & \goodcell{10.18} & {10.00} & \goodcell{10.18} & {10.18} & {10.18} & {10.18} & {10.18} & {10.00} & {10.00} & {10.00} & {10.00}\\
      {vpm2} & {10.18} & {11.71} & {1.53} & {11.18} & {11.37} & {11.18} & {11.37} & {11.38} & {11.38} & {11.38} & {11.38} & {11.04} & {11.19} & {11.19} & \goodcell{11.71}\\
      \midrule
      {Average} & {17.75} & {23.34} & {5.59} & {19.32} & {19.98} & {19.33} & {20.28} & {20.33} & {20.35} & {20.35} & {20.35} & {21.88} & {22.41} & {22.72} & {22.84}\\
      {Wins} & & & & 3 & 7 & 3 & 9 & 11 & 12 & 12 & 12 & 4 & 10 & 15 & 19 \\
    \bottomrule
    \end{tabular}
  \end{adjustbox}
\end{singlespace}
\ifspringer
\else
  \vspace{-12pt}
\fi
Table~\ref{tab:hplane-analysis} shows the effect of hyperplane activation choices on gap closed
for the 30 instances in which GICs close additional gap over SICs.
All possible combinations of objective functions (discussed in Section~\ref{subsec:obj-choices}) for \eqref{CutLP} are used.
The best percent gap closed is shown per instance in: 
column~2 for SICs;
column~3 for GICs across all parameter settings; 
columns~4 through 6 for Algorithm~\ref{alg:targeted-tilting} (targeted tilting) using each of the hyperplane scoring functions (H1), (H2), (H3);
columns~7 through 11 for Algorithm~\ref{alg:targeted-tilting} used in conjunction with Algorithm~\ref{alg:PHA} (for $k_h \in \{0,1,2,3,4\}$, indicated by the column headers ``T+$k_h$''); and
columns~12 through 15 for Algorithm~\ref{alg:PHA} (with $k_h \in \{1,2,3,4\}$, indicated by column headers ``$+k_h$'').
Note that column 7 (T+0) is simply the best result from columns 4 through 6.
For each parameter setting, the last two rows give the average percent gap closed and the number of instances for which that parameter setting achieves the best percent gap closed.

The results show that Algorithm~\ref{alg:PHA} (with $k_h = 4$) closes 98\% of the best percent gap closed,
and that it achieves the best result across all settings for 63\% of the instances without using any targeted tilting.
Using $k_h = 1$ already achieves 96\% of the best percent gap closed, but this comes from strong results from just a few instances.
The marginal impact of activating more hyperplanes diminishes as $k_h$ increases.

The table also shows that the effect of targeted tilting is mixed.
Targeted tilting alone wins 9 times, meaning it achieves the best result for 30\% of the instances.
On 8 of these instances, Algorithm~\ref{alg:PHA} on its own does not attain the best result and only closes 0.2\% of the integrality gap averaged across the 8 instances, compared to 1.6\% from column ``T+0''.
%Moreover, we see that targeted tilting is effective in its goal of generating cuts that improve over SICs.
%Indeed, GICs improve over SICs for 27 of the instances when using Algorithm~\ref{alg:targeted-tilting} alone, i.e., column T+0.
%When when using Algorithm~\ref{alg:PHA} with $k_h = 4$ on its own, two instances (\texttt{sample2} and \texttt{vpm1}) see no improvement at all over SICs.
On the other hand, targeted tilting may lead to weaker cuts,
as reflected in the fact that the best average percent gap closed is achieved by Algorithm~\ref{alg:PHA} on its own,
and that Algorithm~\ref{alg:PHA} seems to have a relatively insignificant impact on improving the results once targeted tilting has been used.
We see that for the 18 instances on which Algorithm~\ref{alg:PHA} uniquely wins (with $k_h=4$, from column ``+4''), Algorithm~\ref{alg:targeted-tilting} performs poorly.
Thus, it appears as though the two algorithms are typically not both effective on the same instance.

In our experiments, we find that though final intersection points are targeted, they represent a small percent (less than 16\% on average for these 30 instances) of all generated points.
This does not necessarily imply bad cuts:\
for example, GICs close an additional 8\% of the gap for \texttt{stein15\_nosym} over SICs alone,
but only 2 of the 168 points created are final.
Nevertheless, this suggests that one promising future direction to obtain stronger cuts
is to develop a procedure targeting more final intersection points, 
as this may lead to more facet-defining inequalities for 
  $\conv(P \setminus \interior{S})$.

\subsection{Evaluating objective function choices}
\label{sec:params:cut-heurs}

% The bilinear program is:
% %   \begin{align*}
% %     \tag*{$BLP(\beta,\pointset,\rayset)$}\label{BLP}
% %     \min_{\alpha,v}\quad &\alpha^\T v \\
% %     &\begin{aligned}
% %     \alpha^\T p &\ge \beta, &&\quad p \in \pointset \\
% %     \alpha^\T r &\ge 0, &&\quad r \in \rayset \\
% %     v &\in \bar{P}.
% %     \end{aligned}
% %   \end{align*}
% \begin{align*}
%     %\tag*{$BLP(\beta,\pointset,\rayset)$}\label{BLP}
%     \min_{\alpha,v} \{\alpha^\T v \suchthat
%     %   \alpha^\T p \ge \beta,\ p \in \pointset;
%     %   \alpha^\T r \ge 0,\ r \in \rayset;
%       (\alpha,\beta) \text{ feasible to \eqref{CutRegion}},\
%       v &\in \bar{P}\}.
%   \end{align*}
% In practice we solve this iteratively, by first optimizing over $\alpha$ for a fixed $v$,
% and then optimizing over $v$ for a fixed $\alpha$.
%\Alex{Selva mentioned some theoretical results concerning the convergence of such a procedure.
%Possibly if it is separable?}

\begin{singlespace}
  \begin{adjustbox}{
max totalheight = 0.83\textwidth,
%    max totalheight = 0.95\textheight,
%    max width = \textwidth,
    rotate=90,
    center,
    captionabove={Best percent gap closed by objective function used for the 30 instances with any improvement over SICs across all settings. 
    		Highlighted cells indicate objectives (or combinations of objectives) achieving best result.
		For the combinations, the cell is highlighted only if using the combination is actually necessary to achieve the result.
    		Some values differ in the thousandths digit.},
    label={tab:cut-heur-analysis},
    float=table,
}
    \begin{tabular}{lc|*{4}{c}|*{6}{c}|*{4}{c}}
    \toprule      
    	{Instance} & {Best} & {R} & {S} & {T} & {V} & {R+S} & {R+T} & {R+V} & {S+T} & {S+V} & {T+V} & {R+S+T} & {R+S+V} & {R+T+V} & {S+T+V}\\
      \midrule
      {bell3a} & {59.52} & \goodcell{59.52} & {57.14} & \goodcell{59.52} & {59.50} & {59.52} & {59.52} & {59.52} & {59.52} & \goodcell{59.52} & {59.52} & {59.52} & {59.52} & {59.52} & {59.52}\\
      {bell3b} & {60.34} & \goodcell{60.34} & {52.71} & \goodcell{60.34} & \goodcell{60.34} & {60.34} & {60.34} & {60.34} & {60.34} & {60.34} & {60.34} & {60.34} & {60.34} & {60.34} & {60.34}\\
      {bell4} & {26.54} & {26.34} & {26.19} & \goodcell{26.54} & {25.90} & {26.34} & {26.54} & {26.34} & {26.54} & {26.27} & {26.54} & {26.54} & {26.34} & {26.54} & {26.54}\\
      {bell5} & {85.37} & \goodcell{85.37} & {65.85} & \goodcell{85.37} & \goodcell{85.37} & {85.37} & {85.37} & {85.37} & {85.37} & {85.37} & {85.37} & {85.37} & {85.37} & {85.37} & {85.37}\\
      {blend2} & {19.78} & {16.56} & {17.70} & {17.36} & {19.35} & {19.71} & {17.36} & {19.35} & {17.70} & \goodcell{19.78} & {19.35} & {19.71} & {19.78} & {19.35} & {19.78}\\
      {bm23} & {11.06} & {9.48} & {10.67} & {11.04} & {10.21} & {10.68} & {11.04} & {10.43} & \goodcell{11.06} & {10.68} & {11.04} & {11.06} & {10.68} & {11.04} & {11.06}\\
      {egout} & {52.11} & \goodcell{52.11} & {51.57} & {51.60} & \goodcell{52.11} & {52.11} & {52.11} & {52.11} & {51.60} & {52.11} & {52.11} & {52.11} & {52.11} & {52.11} & {52.11}\\
      {gt2} & {84.26} & \goodcell{84.26} & {83.13} & \goodcell{84.26} & \goodcell{84.26} & {84.26} & {84.26} & {84.26} & {84.26} & {84.26} & {84.26} & {84.26} & {84.26} & {84.26} & {84.26}\\
      {k16x240} & {7.71} & \goodcell{7.71} & \goodcell{7.71} & \goodcell{7.71} & \goodcell{7.71} & {7.71} & {7.71} & {7.71} & {7.71} & {7.71} & {7.71} & {7.71} & {7.71} & {7.71} & {7.71}\\
      {lseu} & {4.65} & \goodcell{4.65} & {4.61} & \goodcell{4.65} & \goodcell{4.65} & {4.65} & {4.65} & {4.65} & {4.65} & {4.65} & {4.65} & {4.65} & {4.65} & {4.65} & {4.65}\\
      {mas74} & {4.31} & {4.30} & \goodcell{4.31} & {4.13} & {4.30} & {4.31} & {4.30} & {4.30} & {4.31} & {4.31} & {4.30} & {4.31} & {4.31} & {4.30} & {4.31}\\
      {mas76} & {2.49} & {2.49} & \goodcell{2.49} & {2.38} & {2.48} & {2.49} & {2.49} & {2.49} & {2.49} & {2.49} & {2.48} & {2.49} & {2.49} & {2.49} & {2.49}\\
      {mas284} & {0.51} & {0.42} & \goodcell{0.51} & {0.39} & {0.38} & {0.51} & {0.42} & {0.42} & {0.51} & {0.51} & {0.39} & {0.51} & {0.51} & {0.42} & {0.51}\\
      {misc05} & {3.62} & {3.60} & {3.60} & \goodcell{3.62} & {3.60} & {3.60} & {3.62} & {3.60} & {3.62} & \goodcell{3.62} & {3.62} & {3.62} & {3.62} & {3.62} & {3.62}\\
      {mod008} & {1.37} & {1.37} & {1.37} & {1.34} & {1.33} & \goodcell{1.37} & {1.37} & {1.37} & {1.37} & {1.37} & {1.34} & {1.37} & {1.37} & {1.37} & {1.37}\\
      {mod013} & {7.37} & \goodcell{7.37} & \goodcell{7.37} & \goodcell{7.37} & \goodcell{7.37} & {7.37} & {7.37} & {7.37} & {7.37} & {7.37} & {7.37} & {7.37} & {7.37} & {7.37} & {7.37}\\
      {modglob} & {14.02} & {14.02} & {13.79} & \goodcell{14.02} & \goodcell{14.02} & {14.02} & {14.02} & {14.02} & {14.02} & {14.02} & {14.02} & {14.02} & {14.02} & {14.02} & {14.02}\\
      {p0033} & {5.19} & {2.59} & \goodcell{5.19} & \goodcell{5.19} & {2.59} & {5.19} & {5.19} & {2.59} & {5.19} & {5.19} & {5.19} & {5.19} & {5.19} & {5.19} & {5.19}\\
      {p0282} & {5.12} & {4.81} & {4.50} & \goodcell{5.12} & {4.50} & {5.10} & {5.12} & {4.84} & {5.12} & {4.74} & {5.12} & {5.12} & {5.10} & {5.12} & {5.12}\\
      {p0291} & {40.12} & {39.96} & {40.11} & \goodcell{40.12} & {32.91} & {40.11} & {40.12} & {39.96} & {40.12} & {40.11} & {40.12} & {40.12} & {40.11} & {40.12} & {40.12}\\
      {pipex} & {1.43} & {1.43} & \goodcell{1.43} & \goodcell{1.43} & {1.43} & {1.43} & {1.43} & {1.43} & {1.43} & {1.43} & {1.43} & {1.43} & {1.43} & {1.43} & {1.43}\\
      {pp08a} & {54.46} & {53.50} & {51.81} & \goodcell{54.46} & {54.45} & {54.41} & {54.46} & {54.45} & {54.46} & {54.45} & {54.46} & {54.46} & {54.45} & {54.46} & {54.46}\\
      {probportfolio} & {25.28} & \goodcell{25.28} & {25.24} & {25.26} & {25.20} & {25.28} & {25.28} & {25.28} & {25.26} & {25.24} & {25.26} & {25.28} & {25.28} & {25.28} & {25.26}\\
      {sample2} & {13.14} & {9.07} & {8.59} & \goodcell{13.14} & {5.86} & {10.64} & {13.14} & {9.07} & {13.14} & \goodcell{13.14} & {13.14} & {13.14} & {13.14} & {13.14} & {13.14}\\
      {sentoy} & {14.00} & {13.62} & {13.89} & {13.89} & {13.92} & {13.89} & {13.89} & \goodcell{14.00} & {13.89} & {13.92} & {13.92} & {13.89} & {14.00} & {14.00} & {13.92}\\
      {stein15\_nosym} & {58.33} & \goodcell{58.33} & {54.67} & \goodcell{58.33} & {54.17} & {58.33} & {58.33} & {58.33} & {58.33} & {55.14} & {58.33} & {58.33} & {58.33} & {58.33} & {58.33}\\
      {stein27\_nosym} & {8.78} & {8.61} & {7.86} & \goodcell{8.78} & {8.60} & {8.61} & {8.78} & {8.62} & {8.78} & {8.60} & {8.78} & {8.78} & {8.62} & {8.78} & {8.78}\\
      {stein45\_nosym} & {7.55} & {7.51} & {7.45} & {7.51} & {7.49} & {7.51} & {7.51} & \goodcell{7.55} & {7.51} & {7.49} & {7.51} & {7.51} & {7.55} & {7.55} & {7.51}\\
      {vpm1} & {10.18} & {10.04} & {10.04} & \goodcell{10.18} & \goodcell{10.18} & \goodcell{10.18} & {10.18} & {10.18} & {10.18} & {10.18} & {10.18} & {10.18} & {10.18} & {10.18} & {10.18}\\
      {vpm2} & {11.71} & \goodcell{11.71} & {11.18} & {11.60} & \goodcell{11.71} & {11.71} & {11.71} & {11.71} & {11.60} & {11.71} & {11.71} & {11.71} & {11.71} & {11.71} & {11.71}\\
      \midrule
      {Average} & {23.34} & {22.88} & {21.76} & {23.22} & {22.53} & {23.23} & {23.25} & {23.06} & {23.25} & {23.19} & {23.32} & {23.34} & {23.32} & {23.33} & {23.34}\\
%      {Additional Wins}  &  & 13 & 10 & 6 & 20 & 2 & 1 & 0 & 4 & 0 & 2 & 0 & 0 & 0 & 0 \\
		{Wins}   &  & 11 & 7 & 19 & 10 & 18 & 22 & 15 & 23 & 19 & 21 & 27 & 24 & 24 & 26 \\
    \bottomrule
    \end{tabular}
  \end{adjustbox}
\end{singlespace}
\ifspringer
\else
  \vspace{-12pt}
\fi
Next, we examine which objective functions for \eqref{CutLP} lead to the strongest cuts.
Table~\ref{tab:cut-heur-analysis} shows the effect of the possible values of the parameter $\mathcal{O}$
when all other parameters are allowed to vary freely.
The second column gives the best percent gap closed across all parameter settings.
The next four columns provides the best percent gap closed when each of the objective functions is used alone.
The following six columns show the best percent gap closed for all pairs of objective functions,
while the last four columns give the values for any three objective functions.
The highlighted cells are those that achieve the best percent gap closed, but only when that combination of objectives was required to achieve this; by this, we mean that the same result must not have been achieved with any subset of these objectives.
The last row for each column shows the number of instances for which that set of objective functions attains the best percent gap closed.

In our experiments, the objective functions ($T$) (targeting the points in the point-ray collection) on their own work extremely well, closing nearly all of the possible gap on average, and achieving the best result for $19/30$ of the instances.
Using ($R$) (cutting along the directions of the rays of $\optcone$) and ($T$) together achieves the best result across all parameter settings for 22 of the 30 instances.
Moreover, 67\% of the cuts that are active at the post-cut optimum come from procedures ($R$) and ($T$), on average.
Note that the number of possible objectives $(R)$ is $n$, while the number of objectives $(T)$ depends on the size of the point-ray collection, which may be far larger than $n$.

We also tested a bilinear program
that finds a maximally violated cut with respect to a vertex of $P$ with all SICs added,
but it did not yield additional strong cuts.

\subsection{Strength of GICs}
\label{sec:gap-closed}

Finally, we look at the best percent gap closed across all parameter settings.
We compare the percent gap closed by using GICs and SICs together to using SICs on their own.
When the GICs close an additional amount of the integrality gap,
it is clear that we are getting a tighter relaxation of the integral hull.
The converse is not true; that the extra gap closed by GICs is zero %for an instance
does not imply that the cuts have no effect on tightening the relaxation.
%are useless in terms of tightening the relaxation.

\begin{singlespace}
  \begin{adjustbox}{
    max totalheight = 0.95\textheight,
    max width = \textwidth,
    center,
    captionabove={Best percent gap closed and number of cuts.},
    label={tab:gap-closed},
    float=table
}
    \begin{tabular}{l*{2}{H}*{4}{H}cc>{\bfseries}c*{4}{c}@{}}
    \toprule
      {} & {} & {} & \multicolumn{4}{H}{Opt} & \multicolumn{3}{c}{Best \% gap closed} & \multicolumn{4}{c}{\# cuts}\\
      \cmidrule(r){8-10}\cmidrule(r){11-14}
      {Instance} & {Rows} & {Cols} & {LP} & {IP} & {SIC} & {GIC+SIC} & {SIC} & {GIC} & {Diff} & {SICs} & {Active SICs} & {GICs} & {Active GICs}\\
      \midrule
bell3a & 123 & 133 & 862578.6435 & 878430.32 & 869671.429 & 872013.0711 & 44.74 & 59.52 & 14.77 & 32 & 15 & 95 & 36 \\
bell3b & 123 & 133 & 11404143.89 & 11786160.62 & 11574411.7 & 11634635.78 & 44.57 & 60.34 & 15.76 & 35 & 21 & 80 & 28 \\
bell4 & 105 & 117 & 17984775.91 & 18541484.2 & 18114885.46 & 18132539.29 & 23.37 & 26.54 & 3.17 & 46 & 20 & 968 & 21 \\
bell5 & 91 & 104 & 8608417.947 & 8966406.49 & 8660422.457 & 8914044.142 & 14.53 & 85.37 & 70.85 & 25 & 12 & 62 & 38 \\
blend2 & 274 & 353 & 6.915675114 & 7.598985 & 7.025271685 & 7.050862419 & 16.04 & 19.78 & 3.75 & 6 & 1 & 563 & 9 \\
bm23 & 20 & 27 & 20.57092176 & 34 & 21.36652935 & 22.0561593 & 5.92 & 11.06 & 5.14 & 6 & 0 & 1000 & 9 \\
egout & 98 & 141 & 149.5887662 & 568.1 & 365.4230979 & 367.6649798 & 51.57 & 52.11 & 0.54 & 38 & 36 & 45 & 41 \\
flugpl & 18 & 18 & 1167185.726 & 1201500 & 1171213.717 & 1171213.717 & 11.74 & 11.74 & 0.00 & 10 & 6 & 5 & 3 \\
gt2 & 29 & 188 & 13460.23307 & 21166 & 19866.00843 & 19952.77461 & 83.13 & 84.26 & 1.13 & 11 & 11 & 17 & 13 \\
k16x240 & 256 & 480 & 2769.838 & 10674 & 3367.009157 & 3378.93385 & 7.56 & 7.71 & 0.15 & 14 & 6 & 25 & 12 \\
lseu & 28 & 89 & 834.6823529 & 1120 & 847.7260533 & 847.9567036 & 4.57 & 4.65 & 0.08 & 12 & 6 & 87 & 9 \\
mas74 & 13 & 151 & 10482.79528 & 11801.19 & 10526.35577 & 10539.61687 & 3.30 & 4.31 & 1.01 & 12 & 1 & 1000 & 48 \\
mas76 & 12 & 151 & 38893.90364 & 40005.05 & 38920.20713 & 38921.62092 & 2.37 & 2.49 & 0.13 & 11 & 2 & 1000 & 29 \\
mas284 & 68 & 151 & 86195.86303 & 91405.7237 & 86215.8992 & 86222.50124 & 0.38 & 0.51 & 0.13 & 20 & 1 & 1000 & 29 \\
misc05 & 300 & 136 & 2930.9 & 2984.5 & 2932.831551 & 2932.839918 & 3.60 & 3.62 & 0.02 & 11 & 3 & 644 & 14 \\
mod008 & 6 & 319 & 290.9310727 & 307 & 291.140734 & 291.151684 & 1.30 & 1.37 & 0.07 & 5 & 0 & 865 & 2 \\
mod013 & 62 & 96 & 256.0166667 & 280.95 & 257.1166667 & 257.853869 & 4.41 & 7.37 & 2.96 & 5 & 2 & 57 & 12 \\
modglob & 291 & 422 & 20430947.62 & 20740508 & 20460632.83 & 20474361.05 & 9.59 & 14.02 & 4.43 & 29 & 14 & 321 & 71 \\
p0033 & 16 & 33 & 2520.571739 & 3089 & 2530.991887 & 2550.045406 & 1.83 & 5.19 & 3.35 & 6 & 4 & 31 & 7 \\
p0040 & 23 & 40 & 61796.54505 & 62027 & 61811.87593 & 61811.87593 & 6.65 & 6.65 & 0.00 & 4 & 4 & 7 & 7 \\
p0282 & 241 & 282 & 176867.5033 & 258411 & 179863.2035 & 181043.164 & 3.67 & 5.12 & 1.45 & 26 & 4 & 975 & 26 \\
p0291 & 252 & 291 & 1705.128761 & 5223.749 & 2682.677025 & 3116.832133 & 27.78 & 40.12 & 12.34 & 10 & 2 & 105 & 9 \\
pipex & 25 & 48 & 773.751062 & 788.263 & 773.8690854 & 773.959297 & 0.81 & 1.43 & 0.62 & 6 & 2 & 218 & 18 \\
pp08a & 136 & 240 & 2748.345238 & 7350 & 5115.289314 & 5254.445415 & 51.44 & 54.46 & 3.02 & 53 & 43 & 395 & 75 \\
probportfolio & 302 & 320 & 5 & 16.73424676 & 7.949851807 & 7.966917109 & 25.14 & 25.28 & 0.15 & 125 & 69 & 1000 & 143 \\
sample2 & 45 & 67 & 247 & 375 & 254.5 & 263.8194444 & 5.86 & 13.14 & 7.28 & 12 & 2 & 227 & 8 \\
sentoy & 30 & 60 & -7839.278018 & -7772 & -7832.291896 & -7829.858818 & 10.38 & 14.00 & 3.62 & 8 & 0 & 547 & 11 \\
stein15\_nosym & 35 & 15 & 35 & 45 & 40 & 40.83333333 & 50.00 & 58.33 & 8.33 & 5 & 2 & 41 & 9 \\
stein27\_nosym & 117 & 27 & 126 & 207 & 132 & 133.1111111 & 7.41 & 8.78 & 1.37 & 27 & 3 & 1000 & 5 \\
stein45\_nosym & 330 & 45 & 349.6666667 & 594 & 367.0051245 & 368.1078405 & 7.10 & 7.55 & 0.45 & 45 & 2 & 1000 & 6 \\
timtab1 & 171 & 397 & 28694 & 764772 & 157786.5744 & 157786.5744 & 17.54 & 17.54 & 0.00 & 136 & 38 & 1 & 1 \\
vpm1 & 234 & 378 & 15.41666667 & 20 & 15.875 & 15.88333333 & 10.00 & 10.18 & 0.18 & 15 & 10 & 98 & 37 \\
vpm2 & 234 & 378 & 9.889264597 & 13.75 & 10.28229087 & 10.34122777 & 10.18 & 11.71 & 1.53 & 31 & 13 & 413 & 49 \\
\midrule
Average &  &  &  &  &  &  & 17.23 & 22.31 & 5.08 &  &  &  &  \\
\midrule
glass4 & 396 & 322 & 800002400 & 1200012600 & 800002400 & 800002400 & 0 & 0 & 0 & 72 & 36 & 421 & 81 \\
misc01 & 54 & 83 & 57 & 563.5 & 57 & 57 & 0 & 0 & 0 & 12 & 4 & 47 & 6 \\
misc02 & 39 & 59 & 1010 & 1690 & 1010 & 1010 & 0 & 0 & 0 & 8 & 8 & 7 & 7 \\
misc03 & 96 & 160 & 1910 & 3360 & 1910 & 1910 & 0 & 0 & 0 & 22 & 8 & 74 & 11 \\
misc07 & 212 & 260 & 1415 & 2810 & 1415 & 1415 & 0 & 0 & 0 & 26 & 12 & 603 & 220 \\
p0201 & 133 & 201 & 6875 & 7615 & 6875 & 6875 & 0 & 0 & 0 & 40 & 19 & 47 & 17 \\
rgn & 24 & 180 & 48.8 & 82.2 & 48.8 & 48.8 & 0 & 0 & 0 & 19 & 11 & 15 & 9 \\
\bottomrule
    \end{tabular}
  \end{adjustbox}
\end{singlespace}
\ifspringer
\else
  \vspace{-12pt}
\fi
Table~\ref{tab:gap-closed} shows the best result for percent gap closed for each instance, across all parameter settings.
The percent of the integrality gap closed by SICs and by GICs is given in columns~2 and 3.
Column~4 shows the difference between columns 3 and 2. 
Columns~5 and 6 show the number of SICs generated as well as how many of the SICs are active at the optimum of
the LP with all cuts (both SICs and GICs) added.
Columns~7 and 8 show the same for GICs.
The top part of the table shows those instances for which GICs or SICs close any gap over the LP relaxation, and the bottom part of the table contains the remaining instances.

GICs close extra gap over SICs for 75\% of the instances.
The average extra percent gap closed across all instances is around 4.2\%,
around 5.1\% on those instances in which SICs have any effect,
and around 5.6\% on those instances with nonzero extra gap closed by GICs (representing a 31\% improvement over using SICs alone).

Another possible way to assess strength of cuts is how many of the cuts are active 
at the optimum of the continuous relaxation after the addition of all cuts.
On average, over 70\% of the active cuts are GICs.
Moreover, one can select for the strong cuts with reasonable success.
We adopt a common cut selection criterion (see, e.g., \cite{AchBerKocWol08}) that sorts cuts based on a combination of their \emph{efficacy} (the Euclidean distance by which they cut $\lpopt$) and their orthogonality to cuts that have already been selected.
Adding the cuts in this way and setting a cut limit of the same number of GICs as there are SICs, we can close 20.94\% of the integrality gap, compared to the 22.31\% gap closed using all the GICs.
With a limit of five times as many GICs as SICs, we close 22.26\% of the integrality gap, that is, over 99\% of the improvement from adding all GICs;
for comparison, on average across the instances, our cut limit of 1,000 results in about 30 times as many GICs as there are SICs.

Although our procedure is intended to be nonrecursive, we ran a small experiment to test the effect of using GICs with two rounds.
These results are shown in Table~\ref{tab:gap-closed-rd2}.
The second and third columns display the percent gap closed after one round of SICs and one round of GICs (added on top of the SICs, as before).
The fourth and fifth columns show the percent gap closed after two rounds of SICs and a second round of GICs (added on top of the first round of SICs and first round of GICs).
The next two columns give the number of SICs and GICs generated in the second round.
The last column provides the number of rounds of SICs needed to match the gap closed by the second round of GICs.
The three instances marked with asterisks in this last column (\texttt{bell5}, \texttt{flugpl}, and \texttt{mas74}) do not achieve the same gap closed as two rounds of GICs.
The first two instances are terminated when as many SICs are generated as GICs.
For example, for \texttt{bell5}, the 5 rounds of SICs correspond to 97 SICs that together close only 25\% of the integrality gap (this is not shown in the table), compared to 96 GICs that close nearly 87\%.
For \texttt{mas74}, SIC generation tails off, with the objective value improving less than $10^{-3}$ in the last five rounds of SICs 
%and stalling at 4.64\% of the integrality gap closed 
(for reference, matching the result of two GIC rounds would require an additional 0.15 improvement in the objective value).

Lastly, although GICs can be generated from any convex cut-generating set, in our experiments, we limit ourselves to variable split disjunctions.
As a result, it may be useful to compare GICs to lift-and-project cuts~\cite{BalCerCor93}, which can generate the facets of each of the sets $\conv(P \setminus \interior S_k)$.
This comparison is given in Table~\ref{tab:lpc}, where the gap closed by lift-and-project cuts is taken from the computational experiments by \citet{Bonami12}, for those instances appearing in both test sets.
The table indicates that there is a high degree of correlation between the instances on which lift-and-project cuts are effective, and those on which GICs are effective.

\begin{singlespace}
		\begin{table}
		\centering
		\caption{Gap closed over two rounds of cuts, compared between SICs and GICs. The first round is suffixed with a 1, and the second round is suffixed with a 2. The last column is the number of rounds of SICs needed to match the percent gap closed by two rounds of GICs. The two instances marked with asterisks do not achieve the same gap closed as GICs.}
		\label{tab:gap-closed-rd2}
		\begin{tabular}{l*{7}{@{\quad}c@{}}}
		\toprule
			& \multicolumn{4}{c}{\% gap closed} & \multicolumn{2}{c}{\# cuts} \\
			\cmidrule(r){2-5} \cmidrule(r){6-7}
		Instance & SIC1 & GIC1 & SIC2 & GIC2 & SIC2 & GIC2 &  SIC rds \\
		\midrule
bell3a & 44.74 & 59.52 & 63.65 & 63.26 & 55 & 206 & 2 \\
bell3b & 44.57 & 60.34 & 59.49 & 67.02 & 64 & 327 & 4 \\
bell4 & 23.37 & 26.54 & 45.21 & 37.24 & 81 & 1968 & 2 \\
bell5 & 14.53 & 85.37 & 17.93 & 86.65 & 43 & 96 & 5\textsuperscript{*} \\
blend2 & 16.04 & 19.78 & 21.11 & 21.18 & 16 & 1564 & 3 \\
bm23 & 5.92 & 11.06 & 10.40 & 13.45 & 12 & 1958 & 4 \\
egout & 51.57 & 52.11 & 90.89 & 92.48 & 60 & 94 & 3 \\
flugpl & 11.74 & 11.74 & 13.30 & 13.00 & 18 & 10 & 1\textsuperscript{*} \\
gt2 & 83.13 & 84.26 & 95.14 & 94.54 & 27 & 29 & 2 \\
k16x240 & 7.56 & 7.71 & 13.30 & 11.50 & 29 & 109 & 2 \\
lseu & 4.57 & 4.65 & 17.10 & 31.25 & 23 & 342 & 3 \\
mas74 & 3.30 & 4.31 & 4.06 & 4.65 & 27 & 2000 & 50\textsuperscript{*} \\
mas76 & 2.37 & 2.49 & 2.76 & 2.78 & 27 & 2000 & 3 \\
mas284 & 0.38 & 0.51 & 0.94 & 0.95 & 41 & 2000 & 3 \\
misc05 & 3.60 & 3.62 & 6.75 & 6.17 & 32 & 1644 & 2 \\
mod008 & 1.30 & 1.37 & 3.03 & 2.79 & 13 & 1170 & 2 \\
mod013 & 4.41 & 7.37 & 18.89 & 9.64 & 12 & 264 & 2 \\
modglob & 9.59 & 14.02 & 33.07 & 38.62 & 67 & 1321 & 3 \\
p0033 & 1.83 & 5.19 & 7.30 & 10.50 & 18 & 96 & 3 \\
p0040 & 6.65 & 6.65 & 8.04 & 8.55 & 8 & 21 & 3 \\
p0282 & 3.67 & 5.12 & 8.56 & 6.52 & 47 & 1342 & 2 \\
p0291 & 27.78 & 40.12 & 62.24 & 47.22 & 21 & 822 & 2 \\
pipex & 0.81 & 1.43 & 4.21 & 2.18 & 12 & 1218 & 2 \\
pp08a & 51.44 & 54.46 & 66.88 & 70.93 & 102 & 1395 & 3 \\
probportfolio & 25.14 & 25.28 & 26.07 & 25.83 & 218 & 2000 & 2 \\
sample2 & 5.86 & 13.14 & 17.58 & 24.63 & 26 & 1227 & 6 \\
sentoy & 10.38 & 14.00 & 11.97 & 15.19 & 20 & 770 & 11 \\
stein15\_nosym & 50.00 & 58.33 & 58.79 & 76.87 & 9 & 97 & 7 \\
stein27\_nosym & 7.41 & 8.78 & 32.41 & 34.84 & 52 & 1818 & 3 \\
stein45\_nosym & 7.10 & 7.55 & 9.32 & 11.21 & 90 & 2000 & 3 \\
timtab1 & 17.54 & 17.54 & 24.86 & 17.54 & 255 & 1 & 1 \\
vpm1 & 10.00 & 10.18 & 13.91 & 13.99 & 34 & 831 & 3 \\
vpm2 & 10.18 & 11.71 & 19.32 & 17.72 & 63 & 1228 & 2 \\
		\midrule
Average & 17.23 & 22.31 & 26.92 & 29.72 &  &  & 4.5 \\
		\bottomrule
		\end{tabular}
		\end{table}
\end{singlespace}
\ifspringer
\else
  \vspace{-12pt}
\fi

\begin{singlespace}
		\begin{table}
		\centering
		\caption{Comparison of percent gap closed by lift-and-project cuts (taken from the literature) with the gap closed by generalized intersection cuts.}
		\label{tab:lpc}
		\begin{tabular}{lcc@{}}
				\toprule
				& \multicolumn{2}{c}{\% gap closed} \\
				\cmidrule(l){2-3}
				Instance & L\&PC & GIC \\
				\midrule
				bell3a	& 64.56		&	59.52 \\
				bell5	& 86.25			& 85.37 \\
				blend2	& 21.82		& 19.78 \\
				egout	& 93.85		& 52.11 \\
				flugpl	& 11.72		& 11.74 \\
				gt2	& 92.38			& 84.26 \\
				lseu	& 16.58			& 4.65 \\
				mas74	& 5.47		& 4.31 \\
				mas76	& 3.68	& 2.49 \\
				mod008	& 9.02	& 1.37 \\
				modglob	& 57.09		& 14.02\\
				p0033	& 8.19	& 5.19 \\
				p0282	& 93.9	& 5.12 \\
				pp08a	& 79.29		& 54.46 \\
				timtab1	& 26.99		& 17.54 \\
				vpm1	& 31.42		& 10.18 \\
				vpm2	& 54.29 & 11.71 \\
				\midrule
				Average & 44.50 & 26.11\\
				\bottomrule
		\end{tabular}
		\end{table}
\end{singlespace}
\ifspringer
\else
  \vspace{-12pt}
\fi

\subsection{Summary}

The GICs generated from our PHA approach close a significant percentage of the integrality gap with respect to SICs.
This validates the premise of PHA, that stronger intersection cuts can be obtained from collecting a manageable number of intersection points and rays.
We examine the tradeoff in using {\PHAOne} with and without the targeted tilting algorithm. 
Using tilting permits more hyperplanes to be activated,
%and is somewhat better at improving over SICs, 
but it can create weaker points that are avoided by the procedure without tilting.
Moreover, our experiments identify one aspect of PHA that can be improved to lead to stronger cuts.
Only a small percentage of the points we generate are final, which indicates that the GICs generated from our approach are far from the split closure and motivates future work on new methods targeting such final intersection points.
We also evaluate several different objective functions that can be used in \eqref{CutLP} and determine one (the tight point heuristic) that seems particularly effective, which has implications for any future GIC computational experiments.
One of the motivations of the GIC procedure is the ability to generate diverse cuts, which we find is indeed possible.
However, we do not wish to add a large number of cuts to our LP relaxation. We show that in fact a a small set of diverse GICs achieves nearly the same result as using all the GICs.

%%%%%% BIBLIOGRAPHY %%%%%%
%\newpage
%\ifsharelatex
  %\bibliography{{\mainDir/\bibDirName/akazachk}}
  \bibliography{akazachk}
%\else
%  \bibliography{\mainDir/\bibDirName/akazachk}
%\fi
\bibliographystyle{plainnat}

%% Section A: Appendix
%\subimport{\theorydir}{\theoryfname}
\appendix

\section{Additional theory for PHA}
\label{app:theory}
\subsection{Existence of strictly dominating GICs}
\label{subsec:strictDom}
% In Algorithm~\ref{alg:targeted-tilting}, in step~\ref{step:PHA-c-h:choose-hplanes},
% we activate a new hyperplane nearest to $\lpopt$ along the current ray of $\optcone$.
% This is a natural choice, as this improves the relaxation $\optcone$ via neighbors of $\lpopt$.
%{\color{red} In this section, we give two results related to our hyperplane activation approach in Algorithm~\ref{alg:targeted-tilting}. We first provide necessary and sufficient conditions for the existence of a GIC that strictly dominates the SIC after activation of a single hyperplane.}
%We then discuss a modification of step~\ref{step:PHA-c-h:iterative} in our implementation.
In this section, we provide some theoretical motivation for Algorithm~\ref{alg:targeted-tilting} by giving necessary and sufficient conditions for the existence of a GIC that strictly dominates the SIC after activation of a single hyperplane.

% The following definition of strict dominance is from \citet{BalMar13}.
\begin{definition}[\cite{BalMar13}]
  Consider two inequalities that are valid for $P_I$ but not necessarily $P$. 
  Inequality 2 dominates 1 on $P$ if for every $x \in P$, the fact that $x$ satisfies
  Inequality 2 implies that $x$ satisfies Inequality 1. 
  Inequality 2 strictly dominates 1 if, 
  in addition, there exists $x \in P$ such that $x$ violates Inequality 2 but satisfies Inequality 1.
\end{definition}

The theorem proved in this section strengthens Theorem~5 in \citet{BalMar13} 
for the case when $S$ is a split disjunction. 
Theorem~5 of the aforementioned paper gives sufficient conditions for
a GIC to strictly dominate the SIC, given that dominance holds.
We show that this condition is also necessary for strict dominance when $S$ is a split disjunction. 
For ease of exposition, Theorem~\ref{thm:strictDom} assumes that all rays of $C$ intersect $\bd S$ because Proposition~\ref{prop:parallel-rays} shows that intersecting rays cannot lead to deeper points.

Suppose $S = \{x \suchthat 0 \le x_k \le 1\}$ is a split disjunction on a variable $x_k$.
Let $S\onfloor := \{x \suchthat x_k = 0\}$, and $S\onceil := \{x\suchthat x_k = 1\}$.
We partition the intersection point set $\pointset$ into $\pointset\onfloor$ and $\pointset\onceil$,
where $\pointset\onfloor := \pointset \cap S\onfloor$, 
and $\pointset\onceil := \pointset \cap S\onceil$.
Recall that $\initpointset$ and $\initrayset$ are the points and rays obtained 
from intersecting $\optcone$ with $\bd S$. 
We also partition $\initpointset$ into 
$\initpointset\onfloor := \initpointset \cap S\onfloor$ and 
$\initpointset\onceil := \initpointset \cap S\onceil$.
%Let $\initcutpolyhedron$
Intuitively, the theorem shows that a strictly dominating cut with respect to the SIC
must reduce the dimension of $\conv(\initpointset\onfloor)$ or $\conv(\initpointset\onceil)$.

\begin{theorem}\label{thm:strictDom}
  Suppose that $\initrayset = \emptyset$,
  $\initpointset\onfloor \ne \emptyset$ and $\initpointset\onceil \ne \emptyset$,
  and $(\pointset, \rayset)$ is a proper point-ray collection
  obtained from activating a single hyperplane $H$ valid for $P$.
  There exists a basic feasible solution %$(\bar{\alpha},\bar{\beta})$ 
  to \eqref{CutRegion}
  corresponding to a cut strictly dominating $\initcut$ if and only if
  $\relint (H^+) \cap \initpointset\facetindex{t} = \emptyset$
  and $H^- \cap \initpointset\facetindex{t} \ne \emptyset$, for at least one side $t$ 
  of the split disjunction, $t \in \{0,1\}$.
\end{theorem}

\begin{proof}
  For the ``if'' direction of the proof, suppose without loss of generality that 
  	$\relint (H^+) \cap \initpointset\onfloor = \emptyset$ 
  and 
  	$H^- \cap \initpointset\onfloor \ne \emptyset$.
  Any point in $\pointset\onfloor$ lying on the SIC is in $\conv(\initpointset\onfloor)$.
  Because $\relint (H^+) \cap \initpointset\onfloor = \emptyset$, it holds that
  	$(\pointset\onfloor \setminus \initpointset\onfloor) \cap \conv(\initpointset\onfloor) = \emptyset$. 
  This implies that any point $p$ in $\mathcal{P}^0 \setminus \mathcal{P}^0_0$ satisfies 
  	$\initcutcoeff^\T p > \initcutRHS$. 
  Recall that $\card{\initpointset\onfloor} + \card{\initpointset\onceil} = n$. 
  Since some point of $\initpointset\onfloor$ lies in $H^-$,
	$\card{\initpointset\onfloor \cap H^+} \le \card{\initpointset\onfloor} - 1$. 
  It follows that at most $n-1$ intersection points from $\initpointset$ remain in $\pointset$.
  This added degree of freedom and the aforementioned depth of points in 
  	$\pointset\onfloor \setminus \initpointset\onfloor$ 
  allows the SIC to be tilted to obtain a GIC that strictly dominates the SIC.
  The ``only if'' direction follows from Theorem~5 in \citet{BalMar13}.
\end{proof}

The above result shows that any single hyperplane is unlikely to 
directly lead to a strictly dominating cut.
Instead of looking for one such hyperplane, in our implementation we focus
on activating a set of hyperplanes that together cut away large parts of $\conv(\initpointset\onfloor)$
and $\conv(\initpointset\onceil)$.
We do this by targeting each of the intersection points in $\initpointset$ one at a time
in step~\ref{step:PHA:tt-start} of Algorithm~\ref{alg:targeted-tilting}.

Although strict dominance is difficult to attain, our next result shows that activating hyperplanes is monotonic in the sense that the lower bound implied by the point-ray collection can only be improved by activating hyperplanes.
This complements Theorem~3 of \citet{BalMar13}, in which it is shown that activating hyperplanes increases the depth of points with respect to the SIC.
This leaves open the question of whether the lower bound on the objective value 
(implied by the points) improves after activating hyperplanes, 
which Proposition~\ref{prop:ptCostMonotonicity} resolves.
Using the notation from Section~\ref{subsec:obj-choices}, we show that when the ray creating the least cost intersection point $\underline{p}^k$
is cut by a hyperplane, the objective value implied by the new intersection points 
is greater than or equal to $\underline{z}$.

\begin{proposition}\label{prop:ptCostMonotonicity}
  Let $r$ be the edge of $C$ that leads to $\underline{p}^k$, 
  i.e., $\underline{p}^k = r \cap \bd \Sk$,
  $H$ be a hyperplane intersected by $r$ before $\bd \Sk$, and
  $\pointset'$ denote the set of intersection points originating at $r \cap H$, 
  obtained by activating $H$.
  Then $\min \{c^\T p \suchthat p \in \pointset'\} \ge \underline{z}$.
\end{proposition}
\begin{proof}
  Suppose $\Sk\onfloor$ is the facet of $\Sk$ containing $\underline{p}^k$, 
  and let $\Sk\onceil$ be the opposite facet. 
  We have that 
    $\underline{z} 
    	= \min \{ c^\T x \suchthat x \in C \cap \Sk\onfloor \} 
        \le \min \{ c^\T x \suchthat x \in C \cap \Sk\onceil \}$. 
  Since each of the points $p \in \pointset'$ is either (possibly strictly) in 
  	$C \cap \Sk\onfloor$ 
  or 
  	$C \cap \Sk\onceil$, 
  the result follows.
\end{proof}

\subsection{Characterizing bounded objective functions for \texorpdfstring{\eqref{CutLP}}{(PRLP)}}
We turn to an analysis of \eqref{CutLP}.
It is possible for the optimal solution to \eqref{CutLP} to be unbounded,
a behavior we have in fact observed in our numerical implementation.
To better understand this, in this section we present some structural properties of 
  $
    \AlphaSystem(\bar{\beta}, \pointset, \rayset) 
  $,
the feasible region to \eqref{CutLP} for a fixed right-hand size $\bar{\beta}$,
that characterize the objective function choices leading to unboundedness.

We begin by studying when the system $\AlphaSystem(\bar{\beta},\pointset,\rayset)$ has valid cuts
for a given proper point-ray collection.
Recall that $\mathcal{K}'$ denotes the connected component of the skeleton of $P$ that includes $\lpopt \cap \interior S$, 
and any inequality feasible to \eqref{CutRegion} that cuts a point $v \in \mathcal{K}'$ is valid.
We will consider the system
  \[
  	\cutpolyhedron^\# := 
      \{\alpha \suchthat \alpha \in \AlphaSystem(\bar{\beta},\pointset,\rayset);\ v^\T \alpha < \bar{\beta}\}.
  \]
  
\begin{theorem}\label{thm:validCuts}
Let $(\pointset,\rayset)$ be a proper point-ray collection
and let $v \in \mathcal{K}'$.
The system $\AlphaSystem(\bar{\beta}, \pointset, \rayset)$ has valid cuts as feasible solutions
in the following cases: 
(1) for $\bar{\beta} = 1$ if and only if $0 \not \in \cutpolyhedron$ and
  $v \notin \conv(\pointset) + \cone(\pointset \cup \rayset) = \cutpolyhedron + \cone(\pointset)$,
(2) for $\bar{\beta} = -1$ if and only if $v \not \in \conv(\cutpolyhedron \cup \{0\})$, and 
(3) for $\bar{\beta} = 0$ if and only if $v \not \in \cone(\pointset \cup \rayset)$.
\end{theorem}
\begin{proof}
Let $Q$ be the $\card{\pointset} \times n$ matrix containing the intersection points in $\pointset$
as its rows,
and $R$ be the $\card{\rayset} \times n$ matrix with rows comprised of the rays in $\rayset$.
Let ${e}$ denote the $n$-vector of all ones. Using the nonhomogeneous Farkas' lemma~\cite{Mangasarian69}, 
$\cutpolyhedron^\#$ has a feasible solution if and only if the following two systems are infeasible:
  \[
    \begin{matrix}
      \left\{
      \lambda, \mu \ge 0 \suchthat
      \begin{array}{l}
          Q^\T \lambda + R^\T \mu = v \\
          \bar\beta e^\T \lambda \ge \bar\beta
      \end{array}
      \right\}
      & \qquad\qquad &
      \left\{
      \lambda, \mu \ge 0 \suchthat
      \begin{array}{l}
          Q^\T \lambda + R^\T \mu = 0 \\
          \bar\beta e^\T \lambda > 0
      \end{array}
      \right\}
    \end{matrix}
  \]
When $\bar\beta = 1$, the first system is infeasible if and only if $v \notin \cutpolyhedron + \cone(\pointset)$,
and the second system is infeasible if and only if $0 \notin \cutpolyhedron$,
since the existence of a solution $(\lambda,\mu)$ implies $(\lambda/e^\T \lambda, \mu)$ is also feasible.
When $\bar\beta = -1$, the first system is infeasible if and only if 
  $v \notin \conv(\cutpolyhedron \cup \{0\})$,
and the second system is always infeasible, since $\lambda \ge 0$.
When $\bar\beta = 0$, the first system is infeasible if and only if $v \not\in \cone(\pointset \cup \rayset)$,
and the second system is again always infeasible. %\qed
\end{proof}

The feasible region to \eqref{CutLP} is 
	$\AlphaSystem(\bar{\beta},\pointset,\rayset)$, 
not $\cutpolyhedron^\#$.
However, if we assume that $v$ is used as the objective to \eqref{CutLP},
then Theorem~\ref{thm:validCuts} can be used to show when \eqref{CutLP} has a finite solution.
Observe that \eqref{CutLP} implicitly ranks valid inequalities
and picks the most violated cut with respect to $v$.
If there exists a homogeneous inequality valid for $\cutpolyhedron$ that cuts off $v$,
this ranking breaks down,
since all homogeneous inequalities can be scaled to have arbitrarily large violation
and hence are unbounded directions in \eqref{CutLP}.

From the $\bar{\beta} = 0$ case in Theorem~\ref{thm:validCuts},
it follows that the linear program \eqref{CutLP} is bounded
if and only if $v$ belongs to $\cone(\pointset \cup \rayset)$.
Corollary~\ref{cor:beta1cuts} characterizes the two open objective function 
sets within $\cone(\pointset \cup \rayset)$ that admit valid cuts of only one type, 
either with right-hand side $1$ or $-1$.

\begin{corollary}\label{cor:beta1cuts}
The system \eqref{CutLP} has valid inequalities that cut off a point $v$ only for 
\begin{enumerate}
  \item $\bar{\beta} = 1$
    if and only if 
      $0 \notin \cutpolyhedron$ 
    and 
      $v \in \conv(\cutpolyhedron \cup \{0\})\setminus \cutpolyhedron$.
  \item $\bar{\beta} = -1$ 
    if and only if 
      $v \in (\cutpolyhedron + \cone(\pointset)) \setminus \cutpolyhedron$.
\end{enumerate}
\end{corollary}
\begin{proof}
Notice that 
  \[
    \conv(\cutpolyhedron \cup \{0\}) \setminus \cutpolyhedron 
      \subseteq \conv(\cutpolyhedron \cup \{0\}) 
      \subseteq \cone(\pointset \cup \rayset)
  \]
and 
  $ 
    \{\conv(\cutpolyhedron \cup \{0\}) \setminus \cutpolyhedron\} 
      \cap 
    \{\cutpolyhedron + \cone(\pointset)\} = \emptyset
  $.
Therefore, the result in part 1 follows from Theorem~\ref{thm:validCuts}. 
The proof of the second part is similar. %\qed
\end{proof}

\section{Tilting for degenerate hyperplanes}
%\Alex{This section can be submitted as an online supplement (not counting toward the page limit).}
\label{app:tilting}

We previously showed how to tilt a hyperplane $H$ defining $P$ that is not degenerate,
i.e., $\lpopt$ does not lie on $H$.
We defined $H$ using $n$ affinely independent points obtained by intersecting the $n$
affinely independent rays of $\optcone$ with $H$.
These points all lie on one-dimensional faces of $\optcone$.
In the case that $H$ is degenerate, we will instead use two-dimensional faces of $\optcone$
to define the hyperplane, which we will then modify to define a targeted tilting.

When a degenerate hyperplane $H$ is activated on $\optcone$,
each of the extreme rays of the new cone $\optcone \cap H^+$ that lie on $H$
can be defined by $H$ and $n-2$ hyperplanes of $\optcone$ 
that are not redundant for $\optcone \cap H^+$.
Thus, each ray of the new cone lying on $H$ is on a two-dimensional face of $\optcone$.
Let $\rayset^{H*}$ denote an affinely independent set of $n-1$ of these rays lying on $H$.
Let $\optconerays^c$ be the rays of $\optcone$ that are cut by $H^+$.
For each $r \in \rayset^{H*}$, since it lies on a two-dimensional face of $\optcone$, we can define 
  $
    (\bar{r}^1,\bar{r}^2) \in \optconerays^c \times (\optconerays \setminus \optconerays^c)
  $
such that $r$ can be expressed as a convex combination of $\bar{r}^1$ (which is cut by $H$) and $\bar{r}^2$ (which is not cut by $H$);
in particular, let $\lambda^H_r$ be the multiplier such that
  $r = \lambda^H_r \bar{r}^1 + (1-\lambda^H_r) \bar{r}^2$.
Figure~\ref{fig:degenTilting} depicts this construction.
    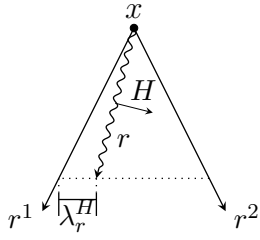
\begin{figure}[htp!]
      \centering
      \begin{tikzpicture}[line join=bevel,x={(1 cm, 0 cm)},y={(0 cm, 1 cm)},z={(\zX cm,\zY cm)},>=stealth,scale=2]
      
      \coordinate (barx) at (1/2, 1);
      \coordinate (pj) at (1,0);
      \coordinate (pk) at (0,0);
      \coordinate (vj) at (3/4,1/2);
      \coordinate (e_start) at (barx);
%      \coordinate (e_end) at (13/16,-1/4);
      \coordinate (e_end) at (1/4,0);
            
      \draw [ray,black] (barx) -- (-1/8,-1/4) node [at end,fill=none,left=-2pt] {\small $r^1$};
      \draw [ray,black] (barx) -- (9/8,-1/4) node [at end,fill=none,right=-2pt] {\small $r^2$};
      \draw [polyhedron_edge,dotted,black] (pj) -- (pk) node [midway,fill=none,above=15pt] {};
      \draw [dotted,-,line width=.5pt,black] (pk) -- ($(pk)+(0,-3/32)$) node [midway,fill=none,above=15pt] {};
      \draw [-,line width=.5pt,black] ($(pk)+(0,-3/32)$) -- ($(pk)+(0,-1/4)$) node [midway,fill=none,above=15pt] {};
      \draw [dotted,-,line width=.5pt,black] (e_end) -- ($(e_end)+(0,-3/32)$) node [midway,fill=none,above=15pt] {};
      \draw [-,line width=.5pt,black] ($(e_end)+(0,-3/32)$) -- ($(e_end)+(0,-1/4)$) node [midway,fill=none,above=15pt] {};
      \draw [-,line width=.5pt,black] ($(pk)+(0,-9/64)$) -- ($(e_end)+(0,-9/64)$) node [midway,fill=none,above=15pt] {};
%      \draw [-,line width=1.5pt,black] (e_start) -- (e_end) node [near start,fill=none,right=0pt] {$e$};
%      \draw [-,line width=1.5pt,black] (vj) -- (ek_end) node [near end,fill=none,above=-2pt] {$e_k$};
      \draw [cut line,->,line width=0.5pt,black] (e_start) -- (e_end) node [near end,fill=none,right=0pt] {\small $r$};
      
      \draw [->,line width=0.25pt,black] (3/8, 1/2) -- (5/8,7/16) node [near end,fill=none,above=0pt] {\small $H$};
      
      \node [draw, point, fill=black, label={[label distance=-2pt]90: $\lpopt$}] at (barx) {};
      \node [circle, inner sep=1.5pt, fill=none] at (1/8,-1/4) {\small $\lambda^H_r$};
%      \node [draw, circle, inner sep=1.5pt, fill=black, label={[label distance=-2pt]180: $p^j$}] at (pj) {};
%      \node [draw, circle, inner sep=1.5pt, fill=black, label={[label distance=0pt]0: $p^k$}] at (pk) {};
%      \node [draw, circle, inner sep=1.5pt, fill=black, label={[label distance=0pt]-90: $p$}] at (e_end) {};
%      \node [draw, circle, inner sep=1.5pt, fill=black, label={[label distance=-2pt]180: $v^j$}] at (vj) {};
      \end{tikzpicture}
      \caption{Illustration of targeted tilting construction for a degenerate hyperplane.}
      \label{fig:degenTilting}
    \end{figure} % end figure degenTilting

If we know $\lambda^H_r$ and the rays $\bar{r}^1$ and $\bar{r}^2$ for every $r \in \rayset^{H^*}$,
then we can use this to give an alternate definition of $H$.
To define a targeted tilting of $H$, we can modify the values $\lambda^H_r$ for each $r \in \rayset^{H^*}$ using some $\delta_r$.
In order to coordinate with the definition of a targeted tilting that we gave in Section~\ref{sec:tilting},
the intersection points and rays created by activating the tilted hyperplane should either be identical to those obtained from activating $H$ or coincide with some initial intersection point or ray.
It is not difficult to see that this means the allowed values for $\delta_r$ are $0$ and $-\lambda^H_r$.
With such a tilting, Theorem~\ref{thm:compute-tilted-points} will still hold, i.e., activations can be performed by using $H$ and $\rayset_A$ without a need for explicitly computing the tilted hyperplane.

\section{Tilting example}
\label{app:PHA}

In this section, we demonstrate an example in which a valid tilting combined with 
the implicit computation used in Theorem~\ref{thm:compute-tilted-points} leads to invalid cuts.
The example additionally provides intuition for the targeted tilting rule that
if a hyperplane cuts a ray of $\optcone$ that has already been cut,
then we should not tilt the hyperplane along that ray.
This is sufficient, as we have shown, to allow us to apply the implicit computation scheme.

The left panel of Figure~\ref{fig:feasReg} shows the feasible region of 
  $P := \{x \in \R^3 \suchthat -2 x_2 + x_3 \le 0; -2 x_1 + x_3 \le 0; 12 x_1 + 10 x_2 - 5 x_3 \le 9; 10 x_1 + 12 x_2 - 5 x_3 \le 9; x_1 + x_2 + x_3 \le 1\}$,
and the right panel of the same figure shows the cone $\optcone$.
The cut-generating set $S$ is the unit box, 
  $\{x \in \R^3: 0 \le x_1 \le 1;\ 0 \le x_2 \le 1\}$.
The hyperplane activations are of $H_4$ and then $H_5$.
We tilt $H_4$ along $r^2$ so that the intersection of this ray with $\tilted{H}_4$
coincides $r^2 \cap \bd{S}$,
and $H_5$ will similarly be tilted along $r^1$ so that the intersection of $r^1$ with $\tilted{H}_5$
is at the point $r^1 \cap \bd{S}$.
The tilted hyperplanes are shown in the top panel of Figure~\ref{fig:hplaneActivation}.

The tilting defined above is clearly valid.
It also satisfies all but one of the conditions of being a targeted tilting;
it does not meet the requirement that $H_5$ and $\tilted{H}_5$ must intersect $r^1$ at the same point,
as a result of the activation of $H_4$ on $r^1$ prior to the activation of $H_5$.

However, as shown in the bottom panel of Figure~\ref{fig:hplaneActivation},
when the implicit computation algorithm is applied to these tilted hyperplanes,
a point of $\conv(P \setminus \interior{S})$ is cut.
This is because $\tilted{H}_5$ intersects $r^1$ outside of the interior of $S$.
Hence, the intersection point $\tilted{H}_5 \cap r^1$ is not added to the point collection, as a result of step~\ref{step:RHA1:is-ray-edge} of Algorithm~\ref{alg:RHA1}.
This intersection point is the same as $H_4 \cap r^1$, which has already been removed from the 
point collection during the activation of $\tilted{H}_4$.

There may be many approaches in which a tilted hyperplane activation can be computed implicitly
using only information from the non-tilted hyperplane.
Our method prevents the situation in this example from occurring by requiring that $r^1$ intersects
$\tilted{H}_5$ at the same point as it intersects $H_5$.
An example of an alternative would be to add the intersection point 
$\tilted{H}_5 \cap r^1 = H_4 \cap r^1$ back into the point collection when activating $\tilted{H}_5$,
but then the intersection points obtained from activating $H_4$ on $r^1$ would be redundant.

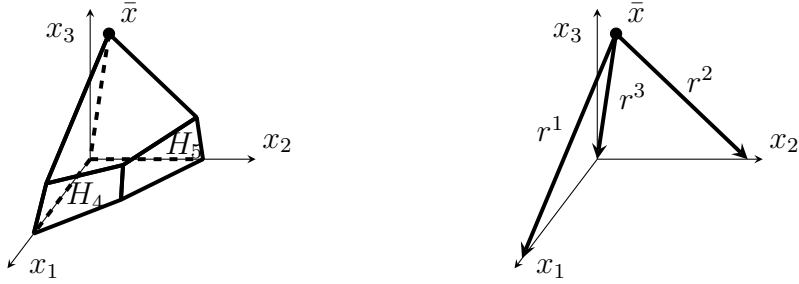
\begin{figure}
\centering
  \begin{minipage}{.4\textwidth}
  %\begin{tikzpicture}[line join=bevel,x={(1 cm, 0.5 cm)},y={(\yX cm, \yY cm)},z={(1 cm,-0.33 cm)},>=stealth,scale=4] 
    \begin{tikzpicture}[line join=bevel,x={(1 cm, 0 cm)},y={(0 cm, 2 cm)},z={(-0.5 cm,-0.66 cm)},>=stealth,scale=2] 
      \pgfmathsetmacro{\scaleRays}{1}

      %\begin{tikzpicture}[line join=bevel,x={(1 cm, 0 cm)},y={(0 cm, 1 cm)},z={(-1 cm,-1 cm)},>=stealth,scale=4]
      %% Points
     \coordinate (barx) at (1/4,1/2,1/4); 

     \coordinate (p1) at (1,0,0);
     \coordinate (p2) at (0,0,1);
     \coordinate (p3) at (0,0,0);

      \coordinate (H22) at (3/4, 0, 0);
      \coordinate (H21) at (0, 0, 9/10);
      \coordinate (H12) at (9/10, 0, 0);
      \coordinate (H11) at (0, 0, 3/4);

      \coordinate (v22) at (3/4, 1/6, 1/12);
      \coordinate (v21) at (1/28, 1/14, 25/28);
      \coordinate (v12) at (25/28, 1/14, 1/28);
      \coordinate (v11) at (1/12, 1/6, 3/4);
      \coordinate (v3) at (7/16, 1/8, 7/16);

      \coordinate (cutpoint) at (9/22, 0, 9/22);
      
      %% Points v2
%       \coordinate (barx) at (1/2,1,1/2); 

%       \coordinate (p1) at (1,0,0);
%       \coordinate (p2) at (0,0,1);
%       \coordinate (p3) at (0,0,0);

%       \coordinate (H11) at (0, 0, 1/4);
%       \coordinate (H12) at (3/4, 0, 0);
%       \coordinate (H21) at (0, 0, 3/4);
%       \coordinate (H22) at (1/4, 0, 0);

%       \coordinate (v11) at (3/16, 3/8, 13/16);
%       \coordinate (v12) at (31/32, 1/16, 1/32);
%       \coordinate (v21) at (1/32, 1/16, 31/32);
%       \coordinate (v22) at (13/16, 3/8, 3/16);
%       \coordinate (v3) at (1/2, 1/4, 1/2);

%       \coordinate (cutpoint) at (3/16, 0, 3/16);

      %% axes
      \draw[->] (xyz cs:x=0) -- (xyz cs:x=1.1) node[label={[label distance=-8pt]45:\normalsize $x_2$}] {};
      \draw[->] (xyz cs:y=-0) -- (xyz cs:y=0.5) node[label={[label distance=-4pt]-135:\normalsize $x_3$}] {};
      % \node [draw, circle, inner sep=.5pt, fill=black] at (0,1/2,0) {};
      % \draw [-,line width=.5pt,black,dotted] (0,1/2,0) -- (barx) node [very near start,fill=none,left=0pt] {};
      \draw[->] (xyz cs:z=0) -- (xyz cs:z=1.1) node[label={[label distance=0pt]0:\normalsize $x_1$}] {};

      %% Inequalities
      \draw [-,line width=1.5pt,black] (H11) -- (v11) -- (barx) node [very near start,fill=none,left=0pt] {};
      \draw [-,line width=1.5pt,black,dashed] (barx) -- (p3) node [very near start,fill=none,left=0pt] {};
      \draw [-,line width=1.5pt,black,dashed] (p3)  -- (H11) node [very near start,fill=none,left=0pt] {};
      \draw [-,line width=1.5pt,black,dashed] (p3)  -- (H22) node [very near start,fill=none,left=0pt] {};
      \draw [-,line width=1.5pt,black] (H11) -- (v11)  -- (v3) -- (cutpoint) -- cycle node [very near start,fill=none,left=0pt] {};
      \draw [-,line width=1.5pt,black] (barx) -- (v11)  -- (v3) -- (v22) -- cycle node [very near start,fill=none,left=0pt] {};
      \draw [-,line width=1.5pt,black] (v3) -- (v22)  -- (H22) -- (cutpoint) -- cycle node [very near start,fill=none,left=0pt] {};

      %% Nodes
      \tikzstyle{es} = [->,>= stealth,shorten >=2pt,line width=1.5pt]
      %\draw [->stealth,shorten,line width=1.5pt,black,blue] (barx) -- (barxTranslated);
      %\draw [es,blue] (barx) -- (barxTranslated);
      \node [draw, circle, inner sep=1.5pt, fill=black, label={[label distance=-2pt]45:\normalsize $\bar x$}] at (barx) {};
      %\node [draw, circle, inner sep=1.5pt, fill=black, label={[label distance=0pt]0:\normalsize $(\bar x,1-\bar{\gamma})$}] at (barxTranslated) {};
%       \node [circle, inner sep=1.5pt, fill=none] at (3/32,9/32,3/4) {$H_1$};
%       \node [circle, inner sep=1.5pt, fill=none] at (12/16,4/16,8/16) {$H_2$};
      \node [circle, inner sep=1.5pt, fill=none] at (11/32,4/32,3/4) {$H_4$};
      \node [circle, inner sep=1.5pt, fill=none] at (10/16,3/64,0) {$H_5$};
%      \node [circle, inner sep=1.5pt, fill=none] at (34/64,3/64,5/32) {$H_2$};
    \end{tikzpicture}
  \end{minipage}
  \begin{minipage}{.4\textwidth} % Cone
  %\begin{tikzpicture}[line join=bevel,x={(1 cm, 0.5 cm)},y={(\yX cm, \yY cm)},z={(1 cm,-0.33 cm)},>=stealth,scale=4] 
    \begin{tikzpicture}[line join=bevel,x={(1 cm, 0 cm)},y={(0 cm, 2 cm)},z={(-0.5 cm,-0.66 cm)},>=stealth,scale=2] 
      \pgfmathsetmacro{\scaleRays}{1}

      %\begin{tikzpicture}[line join=bevel,x={(1 cm, 0 cm)},y={(0 cm, 1 cm)},z={(-1 cm,-1 cm)},>=stealth,scale=4]
      %% Points
       \coordinate (barx) at (1/4,1/2,1/4); 

       \coordinate (p1) at (1,0,0);
       \coordinate (p2) at (0,0,1);
       \coordinate (p3) at (0,0,0);

        \coordinate (H22) at (3/4, 0, 0);
        \coordinate (H21) at (0, 0, 9/10);
        \coordinate (H12) at (9/10, 0, 0);
        \coordinate (H11) at (0, 0, 3/4);

        \coordinate (v22) at (3/4, 1/6, 1/12);
        \coordinate (v21) at (1/28, 1/14, 25/28);
        \coordinate (v12) at (25/28, 1/14, 1/28);
        \coordinate (v11) at (1/12, 1/6, 3/4);
        \coordinate (v3) at (7/16, 1/8, 7/16);

        \coordinate (cutpoint) at (9/22, 0, 9/22);
%       \coordinate (barx) at (1/2,1,1/2); 

%       \coordinate (p1) at (1,0,0);
%       \coordinate (p2) at (0,0,1);
%       \coordinate (p3) at (0,0,0);

%       \coordinate (H11) at (0, 0, 1/4);
%       \coordinate (H12) at (3/4, 0, 0);
%       \coordinate (H21) at (0, 0, 3/4);
%       \coordinate (H22) at (1/4, 0, 0);

%       \coordinate (v11) at (3/16, 3/8, 13/16);
%       \coordinate (v12) at (31/32, 1/16, 1/32);
%       \coordinate (v21) at (1/32, 1/16, 31/32);
%       \coordinate (v22) at (13/16, 3/8, 3/16);
%       \coordinate (v3) at (1/2, 1/4, 1/2);

%       \coordinate (cutpoint) at (3/16, 0, 3/16);

      %% axes
      \draw[->] (xyz cs:x=0) -- (xyz cs:x=1.1) node[label={[label distance=-8pt]45:\normalsize $x_2$}] {};
      \draw[->] (xyz cs:y=-0) -- (xyz cs:y=0.5) node[label={[label distance=-4pt]-135:\normalsize $x_3$}] {};
      % \node [draw, circle, inner sep=.5pt, fill=black] at (0,1/2,0) {};
      % \draw [-,line width=.5pt,black,dotted] (0,1/2,0) -- (barx) node [very near start,fill=none,left=0pt] {};
      \draw[->] (xyz cs:z=0) -- (xyz cs:z=1.1) node[label={[label distance=0pt]0:\normalsize $x_1$}] {};

      %% Inequalities
%      \tikzstyle{es} = [->,>= stealth,shorten >=2pt,line width=1.5pt]
      \tikzstyle{es} = [->,>= stealth,line width=1.5pt]
      \draw [es] (barx) -- (p1) node[label={[label distance=-4pt,pos=.5]45:\normalsize $r^2$}] {};
      \draw [es] (barx) -- (p2) node[label={[label distance=-4pt,pos=.5]135:\normalsize $r^1$}] {};
      \draw [es] (barx) -- (p3) node[label={[label distance=0pt,pos=.5]0:\normalsize $r^3$}] {};

      %% Nodes
      %\draw [->stealth,shorten,line width=1.5pt,black,blue] (barx) -- (barxTranslated);
      %\draw [es,blue] (barx) -- (barxTranslated);
      \node [draw, circle, inner sep=1.5pt, fill=black, label={[label distance=-2pt]45:\normalsize $\bar x$}] at (barx) {};
%        \node [draw, circle, inner sep=1.5pt, fill=red] at (cutpoint) {};
      %\node [draw, circle, inner sep=1.5pt, fill=black, label={[label distance=0pt]0:\normalsize $(\bar x,1-\bar{\gamma})$}] at (barxTranslated) {};
    \end{tikzpicture}
  \end{minipage}
\caption{Feasible region of $P$ and $\optcone$.}
\label{fig:feasReg}
\end{figure}

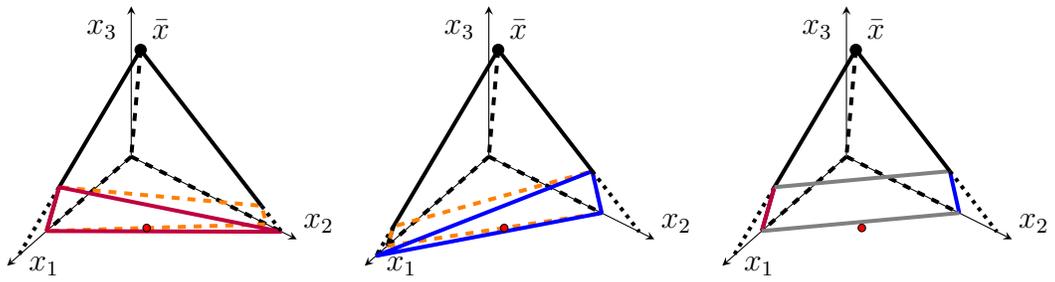
\begin{figure}
\centering
%  \begin{minipage}{.5\textwidth}
%    \centering
    % \begin{tikzpicture}[line join=bevel,x={(1 cm, 0.5 cm)},y={(0 cm, 2 cm)},z={(1 cm,-0.33 cm)},>=stealth,scale=4] 
    %\begin{tikzpicture}[line join=bevel,x={(1 cm, 0 cm)},y={(0 cm, 2 cm)},z={(-0.5 cm,-0.66 cm)},>=stealth,scale=4] 
    %\begin{tikzpicture}[line join=bevel,x={(0.75 cm, -1 cm)},y={(0 cm, 2 cm)},z={(1.5 cm,0.33 cm)},>=stealth,scale=4] 
%    \begin{tikzpicture}[line join=bevel,x={(1 cm, -0.707 cm)},y={(0 cm, 2 cm)},z={(-1 cm,-0.707 cm)},>=stealth,scale=3] 
    \begin{tikzpicture}[line join=bevel,x={(1 cm, -0.5 cm)},y={(0 cm, 2 cm)},z={(-0.75 cm,-0.66 cm)},>=stealth,scale=2] 
      \pgfmathsetmacro{\scaleRays}{1}
      %% Points
      \coordinate (barx) at (1/4,1/2,1/4); 

      \coordinate (p1) at (1,0,0);
      \coordinate (p2) at (0,0,1);
      \coordinate (p3) at (0,0,0);

      \coordinate (H22) at (3/4, 0, 0);
      \coordinate (H21) at (0, 0, 9/10);
      \coordinate (H12) at (9/10, 0, 0);
      \coordinate (H11) at (0, 0, 3/4);

      \coordinate (v22) at (3/4, 1/6, 1/12);
      \coordinate (v21) at (1/28, 1/14, 25/28);
      \coordinate (v12) at (25/28, 1/14, 1/28);
      \coordinate (v11) at (1/12, 1/6, 3/4);
      \coordinate (v3) at (7/16, 1/8, 7/16);

      \coordinate (cutpoint) at (9/22, 0, 9/22);
%       \coordinate (barx) at (1/2,1,1/2); 

%       \coordinate (p1) at (1,0,0);
%       \coordinate (p2) at (0,0,1);
%       \coordinate (p3) at (0,0,0);

%       \coordinate (H11) at (0, 0, 1/4);
%       \coordinate (H12) at (3/4, 0, 0);
%       \coordinate (H21) at (0, 0, 3/4);
%       \coordinate (H22) at (1/4, 0, 0);

%       \coordinate (v11) at (3/16, 3/8, 13/16);
%       \coordinate (v12) at (31/32, 1/16, 1/32);
%       \coordinate (v21) at (1/32, 1/16, 31/32);
%       \coordinate (v22) at (13/16, 3/8, 3/16);
%       \coordinate (v3) at (1/2, 1/4, 1/2);

%       \coordinate (cutpoint) at (3/16, 0, 3/16);

      %% axes
      \draw[->] (xyz cs:x=0) -- (xyz cs:x=1.1) node[label={[label distance=-8pt]45:\normalsize $x_2$}] {};
      \draw[->] (xyz cs:y=-0) -- (xyz cs:y=0.5) node[label={[label distance=-4pt]-135:\normalsize $x_3$}] {};
      % \node [draw, circle, inner sep=.5pt, fill=black] at (0,1/2,0) {};
      % \draw [-,line width=.5pt,black,dotted] (0,1/2,0) -- (barx) node [very near start,fill=none,left=0pt] {};
      \draw[->] (xyz cs:z=0) -- (xyz cs:z=1.1) node[label={[label distance=0pt]0:\normalsize $x_1$}] {};

      %% Inequalities
      \draw [-,line width=1.5pt,black] (H11) -- (v11) -- (barx) node [very near start,fill=none,left=0pt] {};
      \draw [-,line width=1.5pt,black,dashed] (barx) -- (p3) node [very near start,fill=none,left=0pt] {};
      \draw [-,line width=1.5pt,black,dashed] (p3)  -- (H11) node [very near start,fill=none,left=0pt] {};
      \draw [-,line width=1.5pt,black,dashed] (p3)  -- (p1) node [very near start,fill=none,left=0pt] {};
      % \draw [-,line width=1.5pt,black,dashed] (p3)  -- (H22) node [very near start,fill=none,left=0pt] {};
      % \draw [-,line width=1.5pt,black,dashed] (p3)  -- (H11) node [very near start,fill=none,left=0pt] {};
      \draw [-,line width=1.5pt,black,dotted] (v11) -- (p2) node [very near start,fill=none,left=0pt] {};
      \draw [-,line width=1.5pt,orange,dashed] (H11) -- (v11) -- (v12) -- (H12) -- cycle node [very near start,fill=none,left=0pt] {};
      \draw [-,line width=1.5pt,purple] (H11) -- (v11) -- (p1) -- cycle node [very near start,fill=none,left=0pt] {};
      \draw [-,line width=1.5pt,black] (barx) -- (v12) node [very near start,fill=none,left=0pt] {};
      \draw [-,line width=1.5pt,black,dotted] (v12) -- (p1) node [very near start,fill=none,left=0pt] {};
      
      %% Nodes
      \tikzstyle{es} = [->,>= stealth,shorten >=2pt,line width=1.5pt]
      % \draw [->stealth,shorten,line width=1.5pt,black,blue] (barx) -- (p1);
      %\draw [es,blue] (barx) -- (p2);
      \node [draw, circle, inner sep=1.5pt, fill=black, label={[label distance=-2pt]45:\normalsize $\bar x$}] at (barx) {};
      \node [draw, circle, inner sep=1pt, fill=red] at (cutpoint) {};
      %\node [draw, circle, inner sep=1.5pt, fill=black, label={[label distance=0pt]0:\normalsize $(\bar x,1-\bar{\gamma})$}] at (barx) {};
    \end{tikzpicture}
%  \end{minipage}
%
%  \begin{minipage}{.5\textwidth}
%    \centering
      %\begin{tikzpicture}[line join=bevel,x={(1 cm, 0.5 cm)},y={(\yX cm, \yY cm)},z={(1 cm,-0.33 cm)},>=stealth,scale=4] 
      %\begin{tikzpicture}[line join=bevel,x={(1 cm, 0 cm)},y={(0 cm, 2 cm)},z={(-0.5 cm,-0.66 cm)},>=stealth,scale=4] 
%    \begin{tikzpicture}[line join=bevel,x={(1 cm, -0.707 cm)},y={(0 cm, 2 cm)},z={(-1 cm,-0.707 cm)},>=stealth,scale=2] 
    \begin{tikzpicture}[line join=bevel,x={(1 cm, -0.5 cm)},y={(0 cm, 2 cm)},z={(-0.75 cm,-0.66 cm)},>=stealth,scale=2] 
%     \begin{tikzpicture}[line join=bevel,x={(1 cm, 0 cm)},y={(0 cm, 2 cm)},z={(-0.5 cm,-0.66 cm)},>=stealth,scale=4] 
      \pgfmathsetmacro{\scaleRays}{1}

      %\begin{tikzpicture}[line join=bevel,x={(1 cm, 0 cm)},y={(0 cm, 1 cm)},z={(-1 cm,-1 cm)},>=stealth,scale=4]
      %% Points
      \coordinate (barx) at (1/4,1/2,1/4); 

      \coordinate (p1) at (1,0,0);
      \coordinate (p2) at (0,0,1);
      \coordinate (p3) at (0,0,0);

      \coordinate (H22) at (3/4, 0, 0);
      \coordinate (H21) at (0, 0, 9/10);
      \coordinate (H12) at (9/10, 0, 0);
      \coordinate (H11) at (0, 0, 3/4);

      \coordinate (v22) at (3/4, 1/6, 1/12);
      \coordinate (v21) at (1/28, 1/14, 25/28);
      \coordinate (v12) at (25/28, 1/14, 1/28);
      \coordinate (v11) at (1/12, 1/6, 3/4);
      \coordinate (v3) at (7/16, 1/8, 7/16);

      \coordinate (cutpoint) at (9/22, 0, 9/22);
%       \coordinate (barx) at (1/2,1,1/2); 

%       \coordinate (p1) at (1,0,0);
%       \coordinate (p2) at (0,0,1);
%       \coordinate (p3) at (0,0,0);

%       \coordinate (H11) at (0, 0, 1/4);
%       \coordinate (H12) at (3/4, 0, 0);
%       \coordinate (H21) at (0, 0, 3/4);
%       \coordinate (H22) at (1/4, 0, 0);

%       \coordinate (v11) at (3/16, 3/8, 13/16);
%       \coordinate (v12) at (31/32, 1/16, 1/32);
%       \coordinate (v21) at (1/32, 1/16, 31/32);
%       \coordinate (v22) at (13/16, 3/8, 3/16);
%       \coordinate (v3) at (1/2, 1/4, 1/2);

%       \coordinate (cutpoint) at (3/16, 0, 3/16);

      %% axes
      \draw[->] (xyz cs:x=0) -- (xyz cs:x=1.1) node[label={[label distance=-8pt]45:\normalsize $x_2$}] {};
      \draw[->] (xyz cs:y=-0) -- (xyz cs:y=0.5) node[label={[label distance=-4pt]-135:\normalsize $x_3$}] {};
      % \node [draw, circle, inner sep=.5pt, fill=black] at (0,1/2,0) {};
      % \draw [-,line width=.5pt,black,dotted] (0,1/2,0) -- (barx) node [very near start,fill=none,left=0pt] {};
      \draw[->] (xyz cs:z=0) -- (xyz cs:z=1.1) node[label={[label distance=0pt]0:\normalsize $x_1$}] {};

      %% Inequalities
      \draw [-,line width=1.5pt,black] (H21) -- (v21) -- (barx) node [very near start,fill=none,left=0pt] {};
      \draw [-,line width=1.5pt,black,dotted] (v21) -- (p2) node [very near start,fill=none,left=0pt] {};
      \draw [-,line width=1.5pt,black,dashed] (barx) -- (p3) node [very near start,fill=none,left=0pt] {};
      \draw [-,line width=1.5pt,black,dashed] (p3)  -- (p2) node [very near start,fill=none,left=0pt] {};
      \draw [-,line width=1.5pt,black,dashed] (p3)  -- (H22) node [very near start,fill=none,left=0pt] {};
      % \draw [-,line width=1.5pt,black,dashed] (p3)  -- (H22) node [very near start,fill=none,left=0pt] {};
      % \draw [-,line width=1.5pt,black,dashed] (p3)  -- (H11) node [very near start,fill=none,left=0pt] {};
      \draw [-,line width=1.5pt,orange,dashed] (H22) -- (v22) -- (v21) -- (H21) -- cycle node [very near start,fill=none,left=0pt] {};
      \draw [-,line width=1.5pt,blue] (H22) -- (v22) -- (p2) -- cycle node [very near start,fill=none,left=0pt] {};
      \draw [-,line width=1.5pt,black] (barx) -- (v22) node [very near start,fill=none,left=0pt] {};
      \draw [-,line width=1.5pt,black,dotted] (v22) -- (p1) node [very near start,fill=none,left=0pt] {};

      %% Nodes
      \tikzstyle{es} = [->,>= stealth,shorten >=2pt,line width=1.5pt]
      %\draw [->stealth,shorten,line width=1.5pt,black,blue] (barx) -- (barxTranslated);
      %\draw [es,blue] (barx) -- (barxTranslated);
      \node [draw, circle, inner sep=1.5pt, fill=black, label={[label distance=-2pt]45:\normalsize $\bar x$}] at (barx) {};
      \node [draw, circle, inner sep=1pt, fill=red] at (cutpoint) {};
      %\node [draw, circle, inner sep=1.5pt, fill=black, label={[label distance=0pt]0:\normalsize $(\bar x,1-\bar{\gamma})$}] at (barxTranslated) {};
    \end{tikzpicture}
%  \end{minipage}
%
%  \begin{minipage}{.5\textwidth}
%    \centering
    % \begin{tikzpicture}[line join=bevel,x={(1 cm, 0.5 cm)},y={(0 cm, 2 cm)},z={(1 cm,-0.33 cm)},>=stealth,scale=4] 
    %\begin{tikzpicture}[line join=bevel,x={(1 cm, 0 cm)},y={(0 cm, 2 cm)},z={(-0.5 cm,-0.66 cm)},>=stealth,scale=4] 
    %\begin{tikzpicture}[line join=bevel,x={(0.75 cm, -1 cm)},y={(0 cm, 2 cm)},z={(1.5 cm,0.33 cm)},>=stealth,scale=4] 
%    \begin{tikzpicture}[line join=bevel,x={(1 cm, -0.707 cm)},y={(0 cm, 2 cm)},z={(-1 cm,-0.707 cm)},>=stealth,scale=2] 
    \begin{tikzpicture}[line join=bevel,x={(1 cm, -0.5 cm)},y={(0 cm, 2 cm)},z={(-0.75 cm,-0.66 cm)},>=stealth,scale=2] 
%     \begin{tikzpicture}[line join=bevel,x={(1 cm, 0 cm)},y={(0 cm, 2 cm)},z={(-0.5 cm,-0.66 cm)},>=stealth,scale=4] 
      \pgfmathsetmacro{\scaleRays}{1}

      %\begin{tikzpicture}[line join=bevel,x={(1 cm, 0 cm)},y={(0 cm, 1 cm)},z={(-1 cm,-1 cm)},>=stealth,scale=4]
      %% Points
      \coordinate (barx) at (1/4,1/2,1/4); 

      \coordinate (p1) at (1,0,0);
      \coordinate (p2) at (0,0,1);
      \coordinate (p3) at (0,0,0);

      \coordinate (H22) at (3/4, 0, 0);
      \coordinate (H21) at (0, 0, 9/10);
      \coordinate (H12) at (9/10, 0, 0);
      \coordinate (H11) at (0, 0, 3/4);

      \coordinate (v22) at (3/4, 1/6, 1/12);
      \coordinate (v21) at (1/28, 1/14, 25/28);
      \coordinate (v12) at (25/28, 1/14, 1/28);
      \coordinate (v11) at (1/12, 1/6, 3/4);
      \coordinate (v3) at (7/16, 1/8, 7/16);

      \coordinate (cutpoint) at (9/22, 0, 9/22);
%       \coordinate (barx) at (1/2,1,1/2); 

%       \coordinate (p1) at (1,0,0);
%       \coordinate (p2) at (0,0,1);
%       \coordinate (p3) at (0,0,0);

%       \coordinate (H11) at (0, 0, 1/4);
%       \coordinate (H12) at (3/4, 0, 0);
%       \coordinate (H21) at (0, 0, 3/4);
%       \coordinate (H22) at (1/4, 0, 0);

%       \coordinate (v11) at (3/16, 3/8, 13/16);
%       \coordinate (v12) at (31/32, 1/16, 1/32);
%       \coordinate (v21) at (1/32, 1/16, 31/32);
%       \coordinate (v22) at (13/16, 3/8, 3/16);
%       \coordinate (v3) at (1/2, 1/4, 1/2);

%       \coordinate (cutpoint) at (3/16, 0, 3/16);

      %% axes
      \draw[->] (xyz cs:x=0) -- (xyz cs:x=1.1) node[label={[label distance=-8pt]45:\normalsize $x_2$}] {};
      \draw[->] (xyz cs:y=-0) -- (xyz cs:y=0.5) node[label={[label distance=-4pt]-135:\normalsize $x_3$}] {};
      % \node [draw, circle, inner sep=.5pt, fill=black] at (0,1/2,0) {};
      % \draw [-,line width=.5pt,black,dotted] (0,1/2,0) -- (barx) node [very near start,fill=none,left=0pt] {};
      \draw[->] (xyz cs:z=0) -- (xyz cs:z=1.1) node[label={[label distance=0pt]0:\normalsize $x_1$}] {};
      
      %% Inequalities
      \draw [-,line width=1.5pt,black] (H11) -- (v11) -- (barx) node [very near start,fill=none,left=0pt] {};
      \draw [-,line width=1.5pt,black,dashed] (barx) -- (p3) node [very near start,fill=none,left=0pt] {};
      \draw [-,line width=1.5pt,black,dashed] (p3)  -- (H11) node [very near start,fill=none,left=0pt] {};
      \draw [-,line width=1.5pt,black,dashed] (p3)  -- (H22) node [very near start,fill=none,left=0pt] {};
      % \draw [-,line width=1.5pt,black,dashed] (p3)  -- (H22) node [very near start,fill=none,left=0pt] {};
      % \draw [-,line width=1.5pt,black,dashed] (p3)  -- (H11) node [very near start,fill=none,left=0pt] {};
      \draw [-,line width=1.5pt,black,dotted] (v11) -- (p2) node [very near start,fill=none,left=0pt] {};
%       \draw [-,line width=1.5pt,orange] (H22) -- (v22) -- (v21) -- (H21) -- cycle node [very near start,fill=none,left=0pt] {};
      \draw [-,line width=1.5pt,purple] (H11) -- (v11) node [very near start,fill=none,left=0pt] {};
      \draw [-,line width=1.5pt,blue] (H22) -- (v22) node [very near start,fill=none,left=0pt] {};
      \draw [-,line width=1.5pt,black!50] (v22) -- (v11) node [very near start,fill=none,left=0pt] {};
      \draw [-,line width=1.5pt,black!50] (H22) -- (H11) node [very near start,fill=none,left=0pt] {};
      \draw [-,line width=1.5pt,black] (barx) -- (v22) node [very near start,fill=none,left=0pt] {};
      \draw [-,line width=1.5pt,black,dotted] (v22) -- (p1) node [very near start,fill=none,left=0pt] {};
      
      %% Nodes
      \tikzstyle{es} = [->,>= stealth,shorten >=2pt,line width=1.5pt]
      % \draw [->stealth,shorten,line width=1.5pt,black,blue] (barx) -- (p1);
      %\draw [es,blue] (barx) -- (p2);
      \node [draw, circle, inner sep=1.5pt, fill=black, label={[label distance=-2pt]45:\normalsize $\bar x$}] at (barx) {};
      %\node [draw, circle, inner sep=1.5pt, fill=blue, label={[label distance=-2pt]45:}] at (H22) {};
      %\node [draw, circle, inner sep=1.5pt, fill=purple, label={[label distance=-2pt]45:}] at (H11) {};
      \node [draw, circle, inner sep=1pt, fill=red] at (cutpoint) {};
      %\node [circle, inner sep=1.5pt, fill=none] at ($(barxTranslated) + (.25,-.05,0)$) {$(\bar x,1-\gamma)$};
    \end{tikzpicture}
%  \end{minipage}
\caption{Hyperplane activations leading to cutting a point in $\conv(P \setminus \interior S)$.}
\label{fig:hplaneActivation}
\end{figure}

%%%%%%%%%%%%%%%%%%%%%%%%%%%%%%%%%%%%%%%%%%%%%%%%%%%%%%
%Modified running example
%%%%%%%%%%%%%%%%%%%%%%%%%%%%%%%%%%%%%%%%%%%%%%%%%%%%%%
%{\LARGE TEXT EXPLAINING LAST SIX FIGURES IN PAPER}\\

%  The fourth figure shows the partial activation of H4, which can be interpreted as a tilting of H4, with its edge on ray $r_1$ fixed, until its vertex and intersection point on $r_2$ coincide. The tilted version of H4, labeled H4', is $8x_1 + 6x_2 - 3x_3 \leq 6$. The polyhedron defined by H1-H3 and H4' is clearly a relaxation of $P$. Compared to full activation, its vertex and intersection point sets differ on $r_2$. Specifically, one of the vertices no longer exist and the intersection point on $r_2$ is "shallower". 
%
%The fifth figure considers partially activating H5 on the preceding system. The result of this activation is three fold. First, it creates a new vertex on $r_2$. Second, it removes an intersection point on $r_2$ and adds a deeper intersection point. Third, it adds a shallower intersection point on ray $r_1$. The tilted hyperplane H5' is defined by $6x_1 + 8x_2 - 3x_3 \leq 6$. 

\end{document}